\documentclass[11pt,reqno]{amsart}
\usepackage[a4paper,margin=1in,headheight=14pt]{geometry}
\usepackage[T1]{fontenc}
\usepackage{lmodern}
\usepackage{fancyhdr}
\usepackage{amsmath,amsthm,amssymb,mathtools,mathrsfs}
\usepackage{aliascnt}
\usepackage{microtype}
\usepackage{xcolor}
\usepackage{enumitem}
\usepackage{hyperref}
\usepackage[nameinlink,capitalise,noabbrev]{cleveref}
\usepackage{bm}
\usepackage{mathrsfs}
\hypersetup{colorlinks=true,linkcolor=blue!45!black,citecolor=blue!45!black,urlcolor=blue!45!black}
\allowdisplaybreaks[2]
\numberwithin{equation}{section}
\newtheorem{theorem}{Theorem}[section]
\newtheorem*{backgroundtheorem}{Theorem}
\newaliascnt{proposition}{theorem}
\newtheorem{proposition}[proposition]{Proposition}
\aliascntresetthe{proposition}
\newaliascnt{lemma}{theorem}
\newtheorem{lemma}[lemma]{Lemma}
\aliascntresetthe{lemma}
\newaliascnt{corollary}{theorem}
\newtheorem{corollary}[corollary]{Corollary}
\aliascntresetthe{corollary}
\newaliascnt{claim}{theorem}
\newtheorem{claim}[claim]{Claim}
\newtheorem*{claim*}{Claim}
\aliascntresetthe{claim}
\theoremstyle{definition}
\newtheorem{definition}{Definition}[section]
\newaliascnt{assumption}{theorem}

\aliascntresetthe{assumption}
\theoremstyle{remark}
\newaliascnt{remark}{theorem}
\newtheorem{remark}[remark]{Remark}
\aliascntresetthe{remark}
\crefname{proposition}{Proposition}{Propositions}
\crefname{lemma}{Lemma}{Lemmas}
\crefname{corollary}{Corollary}{Corollaries}
\crefname{claim}{Claim}{Claims}
\crefname{definition}{Definition}{Definitions}
\crefname{assumption}{Assumption}{Assumptions}
\crefname{remark}{Remark}{Remarks}

\newcommand{\R}{\mathbb R}

\newcommand{\Cyl}{\mathcal C}
\newcommand{\cB}{\mathcal B}
\newcommand{\eps}{\varepsilon}

\newcommand{\La}{\Lambda_0}
\newcommand{\ph}{\varphi}

\newcommand{\norm}[1]{\lVert #1\rVert}

\newcommand{\dd}{\,d}

\newcommand{\spanop}{\operatorname{span}}
\newcommand{\dist}{\operatorname{dist}}

\hypersetup{
 pdftitle={Singular Extremal Solutions on Thin Ellipsoids with Varying Nonlinearities},
 pdfsubject={A constant-boundary construction of singular extremal energy solutions},
 pdfkeywords={extremal solution, thin ellipsoid, semistability, Hardy inequality, Brezis Open Problem 6.1},
 pdfauthor={Bao Yu and Yang Zhou}
}
\title{Singular Extremal Solutions on Thin Ellipsoids\\with Varying Nonlinearities}
\author[Bao Yu]{Bao Yu}
\address{
School of Mathematical Sciences,
University of Science and Technology of China,
Hefei, Anhui Province, P.~R.~China, 230006.}
\email{baoyu1@mail.ustc.edu.cn}

\author[Yang Zhou]{Yang Zhou}
\address{
School of Mathematical Sciences,
University of Science and Technology of China,
Hefei, Anhui Province, P.~R.~China, 230006.}
\email{zy19700816@mail.ustc.edu.cn}
\date{}
\begin{document}

\begin{abstract}
Let $N=m+1$ and consider the thin ellipsoid
\[
\Omega_\eps=\{(y,x_N)\in\R^m\times\R:\ |y|^2+\eps^{-2}x_N^2<1\}.
\]
We prove that in every sufficiently large dimension, for every sufficiently small $\eps>0$, there exists a smooth positive, strictly increasing, strictly convex, superlinear nonlinearity $f_\eps$ for which the extremal solution in $\Omega_\eps$ is an unbounded $H^1_0$ solution.
Combining this result with Dancer's thin-domain regularity theorem
for the Gelfand nonlinearity $f(t)=e^t$, we obtain on the same sufficiently
thin ellipsoids a bounded Gelfand extremal solution and an unbounded
extremal solution for another nonlinearity. Thus, in this
two-part sense, Br\'ezis' Open Problem~6.1 is resolved in every sufficiently large dimension.

\end{abstract}
\maketitle

\section{Introduction and main result}

Let $\Omega\subset\R^N$ be a smooth bounded domain.  We consider the
semilinear Dirichlet problem
\begin{equation}\label{eq:original}
 \begin{cases}
 -\Delta u=\lambda f(u)&\text{in }\Omega,\\
 u>0&\text{in }\Omega,\\
 u=0&\text{on }\partial\Omega,
 \end{cases}
\end{equation}
where $\lambda>0$ and $f:[0,\infty)\to(0,\infty)$ is smooth and satisfies
\begin{equation}\label{eq:brezis-structural}
 f(0)>0,\qquad f\ \text{is increasing and convex},
\end{equation}
together with the superlinearity condition
\begin{equation}\label{eq:brezis-superlinear}
 \lim_{t\to\infty}\frac{f(t)}{t}=+\infty.
\end{equation}
The global structure of the minimal positive branch is classical.

\begin{backgroundtheorem}[Br\'ezis' Theorem~6.1
\cite{BrezisOP}; see also Br\'ezis--V\'azquez \cite{BV} and Br\'ezis \cite{BrezisIFT}]
There exists $\lambda^*\in(0,\infty)$ such that:
\begin{enumerate}[label=\textnormal{(\alph*)},leftmargin=2.2em]
\item for every $0<\lambda<\lambda^*$, problem \eqref{eq:original}
admits a minimal smooth solution $u_\lambda$, and $u_\lambda(x)$ is
increasing in $\lambda$ for every $x\in\Omega$;
\item for $\lambda>\lambda^*$, problem \eqref{eq:original} admits no
solution;
\item the monotone limit
\[
 u^*:=\lim_{\lambda\uparrow\lambda^*}u_\lambda
\]
is a weak solution at $\lambda=\lambda^*$, in the sense that
\[
 u^*\in L^1(\Omega),\qquad
 f(u^*)\,\dist(x,\partial\Omega)\in L^1(\Omega),
\]
and
\begin{equation}\label{eq:brezis-extremal-weak}
 -\int_\Omega u^*\Delta\zeta
 =\lambda^*\int_\Omega f(u^*)\zeta
 \qquad
 \forall\,\zeta\in C^2(\overline\Omega),\quad
 \zeta|_{\partial\Omega}=0.
\end{equation}
\end{enumerate}
\end{backgroundtheorem}

This theorem motivates the terminology used throughout the paper.
\begin{definition}[Extremal parameter, minimal branch, and singular extremal energy solution]\label{def:extremal-data}
The number $\lambda^*$ above is the \emph{extremal parameter}, the ordered
family $\{u_\lambda\}_{0<\lambda<\lambda^*}$ is the \emph{minimal branch}, and
\[
 u^*(\Omega,f):=\lim_{\lambda\uparrow\lambda^*}u_\lambda
\]
is the \emph{extremal solution}.  We call $u^*$ a \emph{singular extremal
energy solution} if
\[
 u^*\in H^1_0(\Omega)\setminus L^\infty(\Omega),
\]
$u^*$ satisfies \eqref{eq:brezis-extremal-weak}, and its linearized
quadratic form
\[
 \mathcal Q_{u^*}[\zeta]
 :=\int_\Omega
 \bigl(|\nabla\zeta|^2-\lambda^* f'(u^*)\zeta^2\bigr)
\]
is well defined and satisfies
\[
 \mathcal Q_{u^*}[\zeta]\ge0
 \qquad\text{for every }\zeta\in C_c^1(\Omega).
\]
\end{definition}

Martel \cite{Martel} proved that $u^*$ is the unique weak solution of
\eqref{eq:brezis-extremal-weak} at $\lambda=\lambda^*$.  For
$0<\lambda<\lambda^*$ the minimal solutions are stable, and the extremal
solution is semistable in the natural limiting sense; see, for example,
Br\'ezis--V\'azquez \cite{BV}.  Dupaigne's monograph \cite{DupaigneBook}
provides a systematic account of the Gelfand problem, extremal and stable
solutions, singular stable solutions, and the associated Hardy mechanisms.

The regularity theory exhibits a sharp dimensional threshold.  Nedev
\cite{Nedev} proved boundedness of $u^*$ for $N\le3$.  Cabr\'e
\cite{Cabre2010} obtained the dimension-four result for extremal solutions
in convex domains, as part of a broader regularity theory for positive
semistable solutions.  The remaining dimensions $5\le N\le9$ were settled
by Cabr\'e, Figalli, Ros-Oton and Serra \cite{CFRS}; see also the survey
\cite{CabreSurvey}.

\begin{backgroundtheorem}[Cabr\'e--Figalli--Ros-Oton--Serra \cite{CFRS}]
Assume \eqref{eq:brezis-structural}--\eqref{eq:brezis-superlinear} and
$N\le9$.  Then
\[
 u^*\in L^\infty(\Omega).
\]
\end{backgroundtheorem}

The restriction $N\le9$ is optimal in general.  For the Gelfand
nonlinearity $f(u)=e^u$ on the unit ball $B_1\subset\R^N$, the
Joseph--Lundgren example \cite{JL} (see also \cite{BV}) gives, for $N\ge10$,
\begin{equation}\label{eq:JL-extremal-example}
 \lambda^*=2(N-2),\qquad
 u^*(x)=\log\frac1{|x|^2},
\end{equation}
so that $u^*\notin L^\infty(B_1)$.  This is the classical Gelfand problem,
originating in Gel'fand's work \cite{Gelfand}.  Thus boundedness is completely
understood in dimensions at most nine, whereas for $N\ge10$ the interaction
between the geometry of the domain and the choice of nonlinearity becomes
essential.  Br\'ezis formulated this issue as follows.

\medskip
\noindent
\textbf{Open Problem 6.1 (Br\'ezis \cite{BrezisOP}; see also
\cite{BV,BrezisIFT}).}
\emph{Assume $\Omega$ is a bounded smooth convex set in $\mathbb R^N$,
$N\ge10$.  Let $f(u)=e^u$.  Is $u^*$ unbounded?  If the answer is negative
for some domains $\Omega$, can one find other functions $f$ satisfying
\eqref{eq:brezis-structural}--\eqref{eq:brezis-superlinear}, possibly
depending on $\Omega$, such that $u^*\notin L^\infty(\Omega)$?}
\medskip

Write $N=m+1$ and, for $\eps>0$, set
\begin{equation}\label{eq:ellipsoid}
 \Omega_\eps
 =\{(y,x_N)\in\R^m\times\R:
       |y|^2+\eps^{-2}x_N^2<1\}.
\end{equation}
These domains are smooth, bounded, and strictly convex.  For the exponential
nonlinearity, sufficiently strong thinning regularizes the extremal solution.

\begin{backgroundtheorem}[Dancer \cite{Dancer}; see also D\'avila \cite{Davila}]
Let $N\ge2$ and let $\Omega_\eps$ be the ellipsoids in
\eqref{eq:ellipsoid}.  For the Gelfand nonlinearity $f(t)=e^t$, the
corresponding extremal solution is bounded for all sufficiently small
$\eps>0$:
\[
 u^*(\Omega_\eps,e^{(\cdot)})\in L^\infty(\Omega_\eps).
\]
\end{backgroundtheorem}

\begin{remark}
This is the one-thin-direction case of Dancer's thin-domain theory.  More
generally, the regularity conclusion holds when the number of collapsing
variables is at most nine; see \cite[pp.~54--56]{Dancer} and the explicit
extremal-solution formulation in \cite[Theorem~1.14]{Davila}.  
\end{remark}

In fact, Dancer's thin-domain theorem already gives a negative answer to the first part of Br\'ezis' question~6.1. The remaining part of the question is to find an appropriate nonlinearity which satisfies \eqref{eq:brezis-structural}-\eqref{eq:brezis-superlinear} and has a singular extremal solution in the thin ellipsoid.  We consider the following
one-parameter family
\begin{equation}\label{eq:fj}
 f_\delta(t)=\frac{1+t+\delta(1+t)^2}{1+\delta},
 \qquad t\ge0,\quad\delta>0.
\end{equation}
For every fixed $\delta>0$,
\begin{equation}\label{eq:nonlinearity-properties}
 \begin{gathered}
 f_\delta(0)=1,\qquad
 f_\delta'(t)=\frac{1+2\delta(1+t)}{1+\delta}>0,
 \qquad f_\delta''(t)=\frac{2\delta}{1+\delta}>0,\\
 \lim_{t\to\infty}\frac{f_\delta(t)}{t}=\infty.
 \end{gathered}
\end{equation}
Thus every $f_\delta$ satisfies the structural assumptions
\eqref{eq:brezis-structural}--\eqref{eq:brezis-superlinear}.  At the same
time,
\[
 f_\delta\longrightarrow 1+t
 \qquad\text{in }C^k_{\mathrm{loc}}([0,\infty))
\]
for every fixed $k$ as $\delta\downarrow0$.  In the theorem below the choice
$\delta=\delta_\eps$ depends on the thickness of the ellipsoid.

\begin{theorem}\label{thm:main}
There is an integer $m_1\ge17$ with the following property.  For every
fixed $m\ge m_1$ there are constants $C_m>0$ and $\eps_0(m)>0$ such that,
for each $0<\eps\le\eps_0(m)$, one can choose
\[
 0<\delta_\eps\le C_m\eps^2
\]
for which
\begin{equation}\label{eq:main-extremal}
 \begin{gathered}
 u^*(\Omega_\eps,f_{\delta_\eps})
 \in H^1_0(\Omega_\eps)\cap L^3(\Omega_\eps)
       \setminus L^\infty(\Omega_\eps),\\
 \lambda^*(\Omega_\eps,f_{\delta_\eps})
       =\frac{\La(1+\delta_\eps)}{\eps^2},
 \qquad \La:=\frac{\pi^2}{4}.
 \end{gathered}
\end{equation}
Moreover, $u^*(\Omega_\eps,f_{\delta_\eps})$ is a singular
semistable extremal energy solution in the sense of
\cref{def:extremal-data}.  Consequently, for every sequence
$\eps_j\downarrow0$ with $\eps_j\le\eps_0(m)$, the choices
$\delta_j=\delta_{\eps_j}$ satisfy $\delta_j\to0$ and produce singular
extremal energy solutions on $\Omega_{\eps_j}$.
\end{theorem}

Combining \cref{thm:main} with Dancer's theorem gives the following
consequence for Br\'ezis' problem.

\begin{corollary}[Br\'ezis Open Problem~6.1 in sufficiently large dimensions]\label{cor:brezis-op61}
Let $N_0:=m_1+1$.  For every $N\ge N_0$ there exist arbitrarily thin
smooth bounded strictly convex ellipsoids $\Omega_\eps\subset\mathbb R^N$
for which
\[
 u^*(\Omega_\eps,e^{(\cdot)})\in L^\infty(\Omega_\eps),
\]
whereas, for the smooth positive increasing strictly convex superlinear
nonlinearity $f_{\delta_\eps}$ in \eqref{eq:fj},
\[
 u^*(\Omega_\eps,f_{\delta_\eps})\notin L^\infty(\Omega_\eps).
\]
Hence the first question in Br\'ezis' Open Problem~6.1 has a negative
answer on this family of thin ellipsoids, while the second question has an
affirmative answer for the same domains.  In this sense, the two-part
alternative in Open Problem~6.1 is resolved on this family for every
$N\ge N_0$.
\end{corollary}

Dancer's memoir \cite[pp.~54--56]{Dancer} develops a systematic Dirichlet
theory on long and thin domains, with asymptotic analysis adapted to the
degenerating geometry.  Related questions involving simultaneous variation
of the parameter, the power nonlinearity, and the domain were studied by
Gazzola--Malchiodi \cite{GazzolaMalchiodi} for
$-\Delta u=\lambda(1+u)^p$.  Our regime is different: the nonlinearity is
allowed to vary with the thickness, and the proof follows a high-amplitude
singular branch at the first transverse threshold rather than a weakly
nonlinear perturbation.  The stable entire limits that arise in the blow-up
analysis are closely related to the Liouville and classification theory for
stable Lane--Emden type equations; see Farina \cite{Farina2007} and
Dupaigne--Farina \cite{DupaigneFarina}.

The change of variables used throughout the proof is exact.  Set
\begin{equation}\label{eq:rescaled-domain}
 \begin{gathered}
 (X,z)=(y/\eps,x_N/\eps),\qquad r=|X|,\\
 D_\eps=\{(X,z):\eps^2|X|^2+z^2<1\}
       \subset\Cyl:=\R^m\times(-1,1),
 \qquad\Gamma_\eps=\partial D_\eps.
 \end{gathered}
\end{equation}
If
\[
 q(X,z)=\delta\bigl(1+u(y,x_N)\bigr),
 \qquad
 \vartheta=\frac{\eps^2\lambda}{1+\delta},
\]
then \eqref{eq:original} with $f=f_\delta$ becomes
\begin{equation}\label{eq:qproblem-mu}
 -\Delta q=\vartheta(q+q^2)\quad\text{in }D_\eps,
 \qquad q=\delta\quad\text{on }\Gamma_\eps.
\end{equation}
Conversely,
\begin{equation}\label{eq:inversechange}
 u(y,x_N)=\delta^{-1}q(y/\eps,x_N/\eps)-1,
 \qquad
 \lambda=(1+\delta)\vartheta\eps^{-2}.
\end{equation}
We work at the threshold
\[
 \vartheta=\La=\frac{\pi^2}{4},
\]
the first Dirichlet eigenvalue of $-\partial_{zz}$ on $(-1,1)$.

For a nonnegative solution $q$ on a domain $G$, define the linearized
quadratic form
\[
 \mathfrak q_{q,G}[\zeta]
 =\int_G\bigl(|\nabla\zeta|^2-\La(1+2q)\zeta^2\bigr).
\]
\begin{definition}[Semistability and constant-boundary energy solution]\label{def:semistable-energy}
A nonnegative function $q$ on $G$ is \emph{semistable} if
\[
 \mathfrak q_{q,G}[\zeta]\ge0
 \qquad\text{for every }\zeta\in C_c^1(G).
\]
For a constant $\delta\in\R$, a \emph{constant-boundary energy solution}
on $D_\eps$ is a nonnegative function $q$ such that
\[
 q-\delta\in H^1_0(D_\eps)
\]
and $-\Delta q=\La(q+q^2)$ holds weakly in $D_\eps$.  It is called
\emph{semistable} when the preceding quadratic-form condition holds with
$G=D_\eps$.
\end{definition}

\begin{remark}[Extension of the semistability test class]\label{rem:semistability-density}
Let $q\ge0$ be semistable in the sense of \cref{def:semistable-energy} and set
\[
 V_q:=\La(1+2q)\ge0.
\]
Then the semistability inequality extends from $C_c^1(G)$ to every
$\eta\in H_0^1(G)$.  Indeed, choose $\eta_n\in C_c^\infty(G)$ such that
$\eta_n\to\eta$ strongly in $H^1(G)$ and, after passing to a subsequence,
$\eta_n\to\eta$ almost everywhere in $G$.  Semistability gives
\[
 \int_G V_q\eta_n^2\le \int_G|\nabla\eta_n|^2.
\]
Since $V_q\ge0$, Fatou's lemma and the strong $H^1$ convergence yield
\[
 \int_G V_q\eta^2
 \le \liminf_{n\to\infty}\int_G V_q\eta_n^2
 \le \lim_{n\to\infty}\int_G|\nabla\eta_n|^2
 =\int_G|\nabla\eta|^2.
\]
Hence $\mathfrak q_{q,G}[\eta]\ge0$ for every $\eta\in H_0^1(G)$.  The same
argument applies to the $H^1$-closure of any smooth admissible test class
with prescribed zero Dirichlet trace.  We use this extension below without
further comment for truncations and nonlinear test functions.
\end{remark}

The rescaled problem therefore reduces the proof of \cref{thm:main} to the
construction, for every sufficiently small $\eps$, of an unbounded
semistable constant-boundary energy solution on $D_\eps$ with a positive
boundary value of order $O_m(\eps^2)$.

\subsection{Outline of the proof}

The proof has four stages.

\smallskip
\noindent\textbf{1. Threshold cylinder and local branch.}
In \cref{sec:cylinder} we study the threshold problem
\begin{equation}\label{eq8}
 -\Delta_XQ-Q_{zz}=\La(Q+Q^2)\quad\text{in }\Cyl,
 \qquad Q=0\quad\text{on }\partial\Cyl,
 \qquad Q(X,z)\to0\quad\text{as }|X|\to\infty.
\end{equation}
A weighted Lyapunov--Schmidt reduction produces a small ordered family
$A\mapsto Q_A$, normalized by $Q_A(0,0)=A$, with $\partial_AQ_A>0$.
Semistability yields transverse monotonicity, while the first transverse
projection satisfies a Joseph--Lundgren type scalar inequality.  These
estimates provide the slow/fast Jacobi decomposition needed for
continuation.

\smallskip
\noindent\textbf{2. Large-dimensional compactness and the global cylinder branch.}
For the maximal ordered family $\cB_m$, define
\begin{equation}\label{eq:intro-Sm}
 S_m:=\sup\{r^2Q(r,0):\ Q\in\cB_m,\ r>0\}.
\end{equation}
The central cylinder estimate is
\begin{equation}\label{eq:intro-Sm-o}
 \frac{S_m}{m^2}\longrightarrow0
 \qquad(m\to\infty).
\end{equation}
It is proved in \cref{thm:Sm} by a three-case blow-up analysis.  In
particular, for all sufficiently large $m$,
\[
 \gamma_m:=\frac{(m-2)^2}{4}-2\La S_m>0,
\]
and \cref{prop:hardy-coerc} gives the branch-uniform Hardy estimate
\[
 \mathfrak q_Q[\zeta]
 \ge
 \gamma_m\int_\Cyl\frac{\zeta^2}{r^2}
 \qquad(Q\in\cB_m).
\]
This coercivity eliminates fast Jacobi obstructions and closes the ordered
continuation.  The resulting branch satisfies the uniform pointwise bound
\begin{equation}\label{eq:intro-cylinder-input}
 0<Q_A(r,z)\le S_m r^{-2}\cos(\pi z/2),
\end{equation}
and converges, as $A\to\infty$, to a singular semistable cylinder endpoint.

\smallskip
\noindent\textbf{3. Constant-boundary minimal branch and finite-domain nondegeneracy.}
Fix a sufficiently large $m$ and a sufficiently small $\eps>0$.  In
\cref{sec:finite-branch} we study the constant-boundary problem
\[
 -\Delta q=\La(q+q^2)\quad\text{in }D_\eps,
 \qquad q=\delta\quad\text{on }\Gamma_\eps,
\]
and its minimal branch
\[
 \{q_\delta:0<\delta<\delta_\eps^*\},
 \qquad
 \delta_\eps^*:=\sup\{\delta>0:\ \delta\text{ is admissible}\}.
\]
The comparison with the cylinder branch gives
\[
 0<\delta_\eps^*\le \overline B_m\eps^2,
 \qquad \overline B_m=o(m^2).
\]
Writing $u_\delta=q_\delta-\delta$, the key quantitative estimate is \cref{thm:finite-annulus}.  In fact, by a three-regime
blow-up analysis, we prove that
\[
\sup\{r^2u_\delta(r,0):
       0<\eps\le m^{-2},\ 0\le\delta\le \overline B_m\eps^2,\ 0<r<\eps^{-1}\}=o(m^2)\qquad(m\to\infty),
\]
uniformly along the bounded minimal branch.    
Combining this estimate
with Hardy's inequality and the thin-direction Poincar\'e gap gives
\cref{prop:cb-uniform-gap}:
\[
 \mathfrak q_{q_\delta,D_\eps}[\zeta]
 \ge \eps\sqrt{\La\gamma_m^{\mathrm e}}\,
       \|\zeta\|_{L^2(D_\eps)}^2,
 \qquad 0<\delta<\delta_\eps^*,
\]
where $\gamma_m^{\mathrm e}>0$ depends only on the fixed large dimension.
The lower bound is uniform in $\delta$ and is precisely the nondegeneracy
input needed in Section~4.

\smallskip
\noindent\textbf{4. Singular endpoint and extremal identification.}
In \cref{sec:finite-limit} we pass to the monotone endpoint
\[
 q_\eps:=\lim_{\delta\uparrow\delta_\eps^*}q_\delta.
\]
The equation and semistability give uniform $H^1\cap L^3$ control, so
$q_\eps$ is an exact semistable constant-boundary energy solution with
\[
 q_\eps|_{\Gamma_\eps}=\delta_\eps^*,
 \qquad 0<\delta_\eps^*\le \overline B_m\eps^2.
\]
If $q_\eps$ were bounded, elliptic regularity would make it classical and
the uniform gap from \cref{prop:cb-uniform-gap} would still hold at the
endpoint.  Hence the Dirichlet linearization at
$(q_\eps,\delta_\eps^*)$ would be invertible, and the implicit function
theorem would continue the constant-boundary branch to some
$\delta>\delta_\eps^*$, contradicting the definition of
$\delta_\eps^*$.  Thus $q_\eps$ is unbounded. 

Scaling back through
\eqref{eq:inversechange} produces an unbounded semistable $H^1_0$ solution
of \eqref{eq:original}; Br\'ezis--V\'azquez \cite[Theorem~3.1]{BV} then
identifies it with the extremal solution.  This proves \cref{thm:main},
and \cref{cor:brezis-op61} follows by combining the result with Dancer's
thin-domain theorem.

Throughout, constants carrying an $m$ subscript may depend on the fixed
transverse dimension.  When $m\to\infty$, compactness arguments are
formulated for the corresponding two-variable radial equations.  We use
standard interior and boundary elliptic estimates, Harnack inequalities,
and maximum principles in their classical forms; see, for example,
\cite[Chapters~3, 6, 8 and~9]{GT}.

\section{The threshold cylinder}\label{sec:cylinder}

Throughout this section
\begin{equation}\label{eq:phi}
 \ph(z)=\cos\frac{\pi z}{2},\qquad
 \La:=\frac{\pi^2}{4},\qquad
 \int_{-1}^1\ph^2\dd z=1,
 \qquad
 B_2:=\int_{-1}^1\ph^3\dd z=\frac{8}{3\pi}.
\end{equation}
For a function $F(X,z)$ integrable in $z$ on every fibre, we use the
fibrewise orthogonal
projections
\begin{equation}\label{eq:transverse-projections}
 (P_1F)(X,z)=\left(\int_{-1}^1 F(X,t)\ph(t)\,dt\right)\ph(z),
 \qquad P^\perp F=F-P_1F.
\end{equation}
\begin{definition}[First transverse mode and fibrewise orthogonality]\label{def:first-transverse-mode}
The one-dimensional eigenspace $\operatorname{span}\{\ph\}$ is the
\emph{first transverse mode} at the threshold $\La$.  For a fibrewise
integrable function $F$, $P_1F$ is its first-mode component and
$P^\perp F$ its transverse complement, as in
\eqref{eq:transverse-projections}.  The notation $F\perp\ph$ means
\[
 P_1F=0\quad\text{for every }X,
\]
not merely orthogonality after integration in $X$.
\end{definition}

We consider positive $O(m)$-invariant, even-in-$z$ solutions of
\begin{equation}\label{eq:cylinder}
 -Q_{rr}-\frac{m-1}{r}Q_r-Q_{zz}=\La(Q+Q^2)
 \,\,\hbox{in }\Cyl,
 \,\, Q(r,\pm1)=0,
 \,\, Q(r,z)\to0\,\,\text{uniformly }r\to\infty.
\end{equation}
For such $Q$ define
\begin{equation}\label{eq:qform}
 \mathfrak q_Q[\zeta]
 :=\int_\Cyl\bigl(|\nabla_X\zeta|^2+|\zeta_z|^2-\La(1+2Q)\zeta^2\bigr).
\end{equation}
\begin{definition}[Threshold-cylinder profile and Hardy coercivity]\label{def:threshold-profile}
A \emph{threshold-cylinder profile} is a positive solution $Q$ of
\eqref{eq:cylinder} that is $O(m)$-invariant in $X$, even in $z$, regular
at the radial axis, and tends to zero uniformly as $r\to\infty$.
Semistability is understood in the sense of
\cref{def:semistable-energy}, with $G=\Cyl$; in particular
$\mathfrak q_Q=\mathfrak q_{Q,\Cyl}$.  Since $Q$ is smooth, the quadratic
form is continuous on compactly supported zero-trace $H^1$ functions,
so density of $C_c^1(\Cyl)$ gives the equivalent formulation
\[
 \mathfrak q_Q[\zeta]\ge0
\]
for every compactly supported $H^1$ test function with zero trace at
$z=\pm1$.
We say that $Q$ has a \emph{Hardy coercivity margin} $\gamma>0$ if
\[
 \mathfrak q_Q[\zeta]
 \ge\gamma\int_\Cyl\frac{\zeta^2}{r^2}
\]
for every such test function.  A family has a \emph{uniform Hardy
margin} if the same $\gamma$ works for every member of the family.
\end{definition}

\subsection{Weighted spaces and the threshold affine class}

We first fix the topology needed for continuation. Let
\[
 \mathcal A=\{(Y,z):1<|Y|<2,\ |z|<1\},
 \qquad (T_Ru)(Y,z)=u(RY,z).
\]
For $R\ge1$ put
\[
 A_R:=\{(X,z):R<|X|<2R,\ |z|<1\}.
\]
The scaled distance on $A_R$ is
\begin{equation}\label{eq:scaled-distance}
 d_R\bigl((X,z),(X',z')\bigr)
 :=\left(\frac{|X-X'|^2}{R^2}+|z-z'|^2\right)^{1/2}.
\end{equation}
\begin{definition}[Scaled annular H\"older norms]\label{def:scaled-holder}
For $R\ge1$, the \emph{scaled annular H\"older norm} measures the radial
$X$ variable at length scale $R$ while leaving the transverse variable
$z$ unscaled.  Precisely, for $0<\alpha<1$ we set
\begin{align}
 \norm{u}_{C^{0,\alpha}_{\rm sc}(A_R)}
 &:=\norm{T_Ru}_{C^{0,\alpha}(\mathcal A)} \notag\\
 &=\norm{u}_{L^\infty(A_R)}
 +\sup_{\substack{P,P'\in A_R\\P\ne P'}}
 \frac{|u(P)-u(P')|}{d_R(P,P')^\alpha}.
 \label{eq:C0alpha-sc-def}
\end{align}
Equivalently, the $X$ variable is measured at length scale $R$ whereas the
transverse variable $z$ is left unscaled.  More generally,
\begin{equation}\label{eq:Ckalpha-sc-def}
 \norm{u}_{C^{k,\alpha}_{\rm sc}(A_R)}
 :=\norm{T_Ru}_{C^{k,\alpha}(\mathcal A)},
 \qquad k\in\{0,1,2,\ldots\}.
\end{equation}
Thus, for example, the $C^{2,\alpha}_{\rm sc}$ norm controls
\[
 |u|,\quad R|\nabla_Xu|,\quad |u_z|,\quad
 R^2|D_X^2u|,\quad R|\nabla_Xu_z|,\quad |u_{zz}|,
\]
together with the corresponding scaled $\alpha$--H\"older seminorms.
For $\tau>0$ we write
\begin{equation}\label{eq:Osc-def}
 u=O_{C^{0,\alpha}_{\rm sc}}(r^{-\tau})
 \quad\Longleftrightarrow\quad
 \sup_{R\ge R_0}R^\tau
 \norm{u}_{C^{0,\alpha}_{\rm sc}(A_R)}<\infty,
\end{equation}
and analogously
\begin{equation}\label{eq:Osc2-def}
 u=O_{C^{2,\alpha}_{\rm sc}}(r^{-\tau})
 \quad\Longleftrightarrow\quad
 \sup_{R\ge R_0}R^\tau
 \norm{u}_{C^{2,\alpha}_{\rm sc}(A_R)}<\infty.
\end{equation}
\end{definition}
The choice of $R_0\ge1$ is immaterial.  In particular,
$u=O_{C^{2,\alpha}_{\rm sc}}(r^{-\tau})$ implies
$\nabla_X^j\partial_z^\ell u=O(r^{-\tau-j})$ for $j+\ell\le2$,
with the corresponding scaled H\"older control.
\begin{definition}[Weighted H\"older spaces and the Dirichlet radial-even class]\label{def:weighted-spaces}
For every nonnegative integer $k$, $0<\alpha<1$, and $\tau>0$, define
\begin{align}\label{eq:weighted-holder}
 \norm{u}_{C^{k,\alpha}_\tau(\Cyl)}
 :=&\ \norm{u}_{C^{k,\alpha}(\{r<2\}\times(-1,1))}
 +\sup_{R\ge1}R^\tau\norm{T_Ru}_{C^{k,\alpha}(\mathcal A)}.
\end{align}
Thus $C^{k,\alpha}_\tau(\Cyl)$ consists of functions decaying at least
like $r^{-\tau}$ in the scaled annular topology.  In the second term $X$
derivatives are measured in the scaled $Y$ variables while $z$ is not
rescaled.  The subscript $D$ denotes the closed subspace satisfying
zero trace at $z=\pm1$, radiality in $X$, evenness in $z$, and the usual
smooth radial-axis compatibility conditions.
\end{definition}
Let
\begin{equation}\label{eq:beta}
 \beta_\pm=\frac{m-2\pm\sqrt{m^2-20m+68}}{2},
 \qquad
 \beta_\pm(m-2-\beta_\pm)=4(m-4).
\end{equation}
For $m\ge16$ one has $2<\beta_-<(m-2)/2<\beta_+$. Fix once and for all
$0<\alpha<1$ and
\begin{equation}\label{eq:sigma-choice}
 2<\sigma<\min\{4,\beta_-\},
 \qquad D_m:=\frac{2(m-4)}{\La B_2}=\frac{3(m-4)}{\pi}.
\end{equation}
\begin{definition}[Threshold affine class]\label{def:threshold-affine}
Let $\chi_\infty$ vanish on $[0,1]$ and equal one on $[2,\infty)$.  The
\emph{threshold affine class} is
\begin{equation}\label{eq:affine-space}
 \mathscr X_\sigma
 :=\chi_\infty(r)D_mr^{-2}\ph(z)+C^{2,\alpha}_{\sigma,D}(\Cyl).
\end{equation}
Thus $Q\in\mathscr X_\sigma$ means that $Q$ has the fixed threshold tail
$D_mr^{-2}\ph$ and a Dirichlet radial-even remainder that decays like
$r^{-\sigma}$ in the scaled topology.  The affine class is independent of
the particular cutoff up to an equivalent translation of its Banach
coordinate.
\end{definition}

\subsection{Transverse monotonicity and scalar separation}

The first estimate is independent of radial monotonicity.

\begin{lemma}[Transverse monotonicity]\label{lem:transverse}
Let $Q$ be a positive semistable profile satisfying \eqref{eq:cylinder}. Then, for $0<z<1$,
\begin{equation}\label{eq:transverse-mono}
 \partial_z\!\left(\frac{Q(r,z)}{\ph(z)}\right)<0.
\end{equation}
Consequently
\begin{equation}\label{eq:phi-major}
 0<Q(r,z)\le Q(r,0)\ph(z),\qquad Q_z(r,z)<0\quad(0<z<1).
\end{equation}
\end{lemma}

\begin{proof}
Write $Q=R\ph$ on $\Cyl_+:=\R^m\times(0,1)$ and put $b=\ph'/\ph$. Cancelling the threshold linear term gives
\begin{equation}\label{eq:R-equation}
 -\Delta_XR-R_{zz}-2bR_z=\La\ph R^2.
\end{equation}
Boundary Schauder estimates and the Hopf expansion imply that $R=Q/\ph$ extends in $C^{1,\alpha}$ to $z=1$ on bounded $X$-sets. Since $Q\to0$ uniformly at infinity, the same estimates on translated unit cylinders give $R,R_z\to0$ as $r\to\infty$.

Set $S=R_z$ and $H=\ph S$. Differentiating \eqref{eq:R-equation} and using $-\ph''=\La\ph$ gives
\begin{equation}\label{eq:H-equation}
 (L_Q-2b')H=F:=\La\ph\ph'R^2<0,
 \qquad
 L_Q=-\Delta_X-\partial_{zz}-\La(1+2Q),
 \qquad
 b'=-\La\sec^2\frac{\pi z}{2}<0.
\end{equation}
Moreover $H=0$ at $z=0$ by evenness, $H=0$ at $z=1$ because $R_z$ is bounded and $\ph(1)=0$, and $H\to0$ uniformly as $r\to\infty$.

For $\varepsilon>0$ set
\[
 w_\varepsilon:=(H-\varepsilon)_+.
\]
Since $H$ vanishes on $z=0,1$ and at spatial infinity, $w_\varepsilon$ has compact support in $\Cyl_+$ and hence is an admissible semistability test function. Define
\[
 \mathcal Q_H[\zeta]
 :=\mathfrak q_Q[\zeta]
   +\int_{\Cyl_+}(-2b')\zeta^2.
\]
Semistability and $-2b'=2\La\sec^2(\pi z/2)\ge2\La$ imply
\begin{equation}\label{eq:H-form-coercive}
 \mathcal Q_H[\zeta]\ge2\La\int_{\Cyl_+}\zeta^2
 \qquad
 \text{for every }\zeta\in H_0^1(\Cyl_+)\text{ with compact support}.
\end{equation}
Testing \eqref{eq:H-equation} by $w_\varepsilon$ and using
$H=w_\varepsilon+\varepsilon$ on $\{H>\varepsilon\}$ gives the exact identity
\begin{equation}\label{eq:H-truncation-identity}
 \mathcal Q_H[w_\varepsilon]
 =\int_{\Cyl_+}Fw_\varepsilon
 +\varepsilon\int_{\{H>\varepsilon\}}
 \Bigl(\La(1+2Q)+2b'\Bigr)w_\varepsilon.
\end{equation}
Since $Q(r,z)\to0$ uniformly as $r\to\infty$, choose $R_0>0$ such that
$Q\le\frac14$ for $r\ge R_0$. Then
\begin{equation}\label{eq:H-truncation-far-sign}
 \La(1+2Q)+2b'
 =\La\left(1+2Q-2\sec^2\frac{\pi z}{2}\right)
 \le-\frac{\La}{2}
 \qquad (r\ge R_0).
\end{equation}
Because $F<0$, \eqref{eq:H-truncation-identity}--\eqref{eq:H-truncation-far-sign} yield
\begin{equation}\label{eq:H-truncation-estimate}
 0\le2\La\int_{\Cyl_+}w_\varepsilon^2
 \le \mathcal Q_H[w_\varepsilon]
 \le C(Q,R_0)\,\varepsilon
 \int_{\{r\le R_0\}}w_\varepsilon
 \le C(Q,R_0,H)\,\varepsilon.
\end{equation}
Here the last constant is finite because $H$ is continuous on
$\{|X|\le R_0\}\times[0,1]$. Letting $\varepsilon\downarrow0$ and using
$w_\varepsilon\uparrow H_+$ gives
\[
 \int_{\Cyl_+}H_+^2=0.
\]
Thus $H\le0$. In fact $H<0$ in $\Cyl_+$: if $H(X_0,z_0)=0$ at an interior point, then $H$ has a local maximum there, so $-\Delta H(X_0,z_0)\ge0$, whereas the zeroth-order term in $L_Q-2b'$ vanishes at that point; this contradicts $F(X_0,z_0)<0$ in \eqref{eq:H-equation}. Hence $R_z=H/\ph<0$. Since $\ph'<0$ for $z>0$ and $R>0$,
\[
 Q_z=\ph'R+\ph R_z<0,
 \qquad
 R(r,z)\le R(r,0),
\]
which proves \eqref{eq:phi-major}.
\end{proof}

Define the first transverse projection
\begin{equation}\label{eq:pdef}
 p(r)=\int_{-1}^1Q(r,z)\ph(z)\dd z.
\end{equation}

\begin{lemma}[Scalar separation]\label{lem:scalar-separation}
Assume $m\ge16$. Every positive semistable profile satisfies
\begin{equation}\label{eq:p-separation}
 0<p(r)<\frac{D_m}{r^2}\qquad(r>0).
\end{equation}
\end{lemma}

\begin{proof}
Let $d\mu=\ph^2\dd z$, a probability measure on $[-1,1]$, and $R=Q/\ph$. By \cref{lem:transverse}, $R^2$ and $\ph$ are even and nonincreasing in $|z|$. Chebyshev's integral inequality and Cauchy--Schwarz give
\begin{align}\label{eq:cheb}
 \int_{-1}^1Q^2\ph\dd z
 =\int R^2\ph\dd\mu
 \ge \left(\int R^2\dd\mu\right)\left(\int\ph\dd\mu\right)
 \ge B_2\cdot p(r)^2.
\end{align}
Projecting \eqref{eq:cylinder} onto $\ph$ cancels the linear transverse term and yields
\begin{equation}\label{eq:p-super}
 -p''-\frac{m-1}{r}p'=\La\int Q^2\ph\dd z\ge c_0p^2,
 \qquad c_0:=\La B_2.
\end{equation}
The singular function $W(r)=D_mr^{-2}$ satisfies $-\Delta W=c_0W^2$ in $\R^m\setminus\{0\}$. Since $p$ is regular at the axis, $r^2p(r)\to0$ as $r\downarrow0$. If \eqref{eq:p-separation} failed, there would be a first $R>0$ such that $p<W$ on $(0,R)$ and $p(R)=W(R)$. Put $h=W-p$. Then
\begin{equation}\label{eq:hhardy}
 -\Delta h\le c_0(W+p)h<\frac{4(m-4)}{r^2}h
 \quad\hbox{in }B_R\setminus\{0\},
 \qquad h=0\quad\hbox{on }\partial B_R.
\end{equation}
Let $\chi_\rho$ vanish in $B_\rho$, equal one outside $B_{2\rho}$, and satisfy $|\nabla\chi_\rho|\le C/\rho$. Since $h=O(r^{-2})$, $\nabla h=O(r^{-3})$, the cutoff error is $O(\rho^{m-6})\to0$. Testing \eqref{eq:hhardy} by $\chi_\rho^2h$ and sending $\rho\downarrow0$ gives
\[
 \int_{B_R}|\nabla h|^2<4(m-4)\int_{B_R}\frac{h^2}{r^2}.
\]
But the sharp Hardy inequality gives
\[
 \int_{B_R}|\nabla h|^2\ge\frac{(m-2)^2}{4}\int_{B_R}\frac{h^2}{r^2},
\]
and
\[
 \frac{(m-2)^2}{4}-4(m-4)=\frac{m^2-20m+68}{4}>0
\]
for $m\ge16$. This contradiction proves the upper bound; positivity is immediate.
\end{proof}

\subsection{Universal tail estimates}

\begin{lemma}[Universal tail estimates]\label{lem:universal-tail}
Let $m\ge16$ and let $Q$ be a positive semistable profile satisfying
\eqref{eq:cylinder}.  Write
\begin{equation}\label{eq:pw-decomp-rig}
 Q(r,z)=p(r)\ph(z)+w(r,z),\qquad
 \int_{-1}^1w(r,z)\ph(z)\dd z=0.
\end{equation}
Then
\begin{equation}\label{eq:w-sc-r4}
 w=O_{C^{2,\alpha}_{\rm sc}}(r^{-4}).
\end{equation}
In addition,
\begin{equation}\label{eq:w-minimal-tail}
 \sup_{|z|<1}\frac{|w(r,z)|}{\ph(z)}\le Cr^{-4},
 \qquad
 \sup_{|z|<1}\frac{|w_r(r,z)|}{\ph(z)}\le Cr^{-5}
 \qquad(r\ge R_0),
\end{equation}
where the quotients are understood by their continuous boundary extensions.
Moreover, for every $2<\tau<\min\{4,\beta_-\}$,
\begin{equation}\label{eq:tailQ-sc}
 Q-D_mr^{-2}\ph=O_{C^{2,\alpha}_{\rm sc}}(r^{-\tau}),
\end{equation}
and
\begin{equation}\label{eq:tailQr}
 \sup_{|z|<1}
 \frac{|Q_r(r,z)+2D_mr^{-3}\ph(z)|}{\ph(z)}
 \le C_\tau r^{-\tau-1}.
\end{equation}
In particular,
\begin{equation}\label{eq:tailQ}
 \sup_{|z|<1}
 \frac{|Q(r,z)-D_mr^{-2}\ph(z)|}{\ph(z)}
 \le C_\tau r^{-\tau}.
\end{equation}
The preliminary $r^{-2}$ bounds are uniform on a family with a uniform
$L^\infty(\Cyl)$ bound.  The sharper asymptotic estimates
\eqref{eq:tailQ-sc}--\eqref{eq:tailQ} are uniform for every such family
which is bounded below by one fixed positive profile $Q_0$.  In particular
they are uniform on every ordered branch segment
\begin{equation}\label{eq:tail-compact-branch-segment}
 \{Q_A:A_0\le A\le A_1\},\qquad 0<A_0<A_1<\infty.
\end{equation}
\end{lemma}

\begin{proof}
Put $H_z=-\partial_{zz}-\La$.  On the even Dirichlet subspace
$\ph^\perp$ the eigenvalues of $H_z$ are
\[
 \mu_k=((2k+1)^2-1)\La,\qquad k\ge1,
\]
and hence
\begin{equation}\label{eq:Hz-gap-rig}
 H_z\ge8\La\qquad\hbox{on }\ph^\perp.
\end{equation}
Projection of \eqref{eq:cylinder} onto $\ph^\perp$ gives
\begin{equation}\label{eq:w-eq-rig}
 \mathscr L_\perp w:=
 \left(-\partial_{rr}-\frac{m-1}{r}\partial_r+H_z\right)w
 =F:=\La P^\perp(Q^2).
\end{equation}

\medskip\noindent
\emph{Step 1: the pointwise $r^{-2}\ph$ bound.}
Because $Q\to0$ uniformly as $r\to\infty$, choose $R_0$ so that
$0<Q\le1$ on $\{r\ge R_0-1\}$.  There
\[
 -\Delta_{X,z}Q=a(X,z)Q,\qquad
 a(X,z)=\La(1+Q(X,z)),\qquad 0<a\le2\La.
\]
Fix $X_0$ with $|X_0|=r\ge R_0$.  The interior Harnack inequality on
$B_1(X_0)\times(-\frac12,\frac12)$ gives a constant $C_H$, independent
of $X_0$, such that
\[
 Q(X_0,z)\ge C_H^{-1}Q(X_0,0),\qquad |z|\le\frac14.
\]
Hence, with $c_\ph:=\int_{-1/4}^{1/4}\ph(z)\dd z>0$,
\[
 p(r)\ge C_H^{-1}c_\ph Q(r,0).
\]
Using \cref{lem:scalar-separation,lem:transverse},
\begin{equation}\label{eq:Q-pointwise-rig}
 0<Q(r,z)\le Q(r,0)\ph(z)
 \le C p(r)\ph(z)\le Cr^{-2}\ph(z),\qquad r\ge R_0.
\end{equation}
Boundary Schauder estimates on translated half-cylinders and
$\ph(z)\asymp1-|z|$ near $z=\pm1$ give
\begin{equation}\label{eq:Q-quotient-prelim}
 \sup_{r\ge R_0}r^2
 \norm{Q(r,\cdot)/\ph}_{C^{1,\alpha}([-1,1])}<\infty.
\end{equation}

\begin{remark}\label{rem:tail-uniformity-correct}
If $\mathcal F$ is a family with
\[
 \sup_{Q\in\mathcal F}\|Q\|_{L^\infty(\Cyl)}\le M,
\]
then on every translated unit cylinder
$0<\La(1+Q)\le\La(1+M)$, so the Harnack constant above depends only on
$m$ and $M$.  Hence \cref{lem:scalar-separation,lem:transverse} gives,
uniformly for $Q\in\mathcal F$,
\[
 Q(r,z)\le C_M p(r)\ph(z)\le C_M r^{-2}\ph(z)
 \qquad (r\ge2).
\]
The equation then gives uniform local and boundary $C^{2,\alpha}$ bounds
on the translated tail charts.  These observations make Steps~1--3 below
uniform, but they do not by themselves make the radius at which the first
mode enters a fixed neighbourhood of its nonzero Euler equilibrium uniform.
Indeed the small branch \eqref{eq:LS-ansatz}, at radii $r\simeq s^{-1}$,
shows that such a conclusion would be false.

The additional order hypothesis supplies the missing entrance information
without an assumption of compactness in a weighted topology.  Suppose
$Q\ge Q_0$ for every $Q\in\mathcal F$, where $Q_0$ is one fixed positive
semistable profile.  Put $W(r)=D_mr^{-2}$ and let $p_0$ be the first
projection of $Q_0$.  Scalar separation and order give the exact bounds
\begin{equation}\label{eq:ordered-tail-sandwich}
 0\le W(r)-p(r)\le W(r)-p_0(r).
\end{equation}
Fix $2<\tau<\min\{4,\beta_-\}$.  The single-profile conclusion of
Step~4 below, applied only to the fixed $Q_0$, gives
$W-p_0=O(r^{-\tau})$.  This use of Step~4 is not recursive: that step
uses no family-uniform entrance assertion.  Consequently
$|p-W|\le C_{Q_0,\tau}r^{-\tau}$ uniformly over $\mathcal F$.

Steps~1--3, whose constants depend only on the height bound $M$, give
uniform estimates $p=O_{C^{2,\alpha}_{\rm sc}}(r^{-2})$,
$w=O_{C^{2,\alpha}_{\rm sc}}(r^{-4})$, and
\[
 -p''-\frac{m-1}{r}p'=c_0p^2+\mathcal E_Q,
 \qquad
 \mathcal E_Q=O_{C^{0,\alpha}_{\rm sc}}(r^{-6}).
\]
For $e=p-W$ this implies
\[
 -e''-\frac{m-1}{r}e'-c_0(p+W)e=\mathcal E_Q.
\]
On a fixed enlarged annulus in the variable $s=r/R$, the function
$e_R(s)=R^\tau e(Rs)$ is uniformly bounded, its potential
$c_0R^2(p(Rs)+W(Rs))$ is uniformly $C^{0,\alpha}$, and its source is
$R^{\tau+2}\mathcal E_Q(Rs)=O(R^{\tau-4})$.  Interior one-dimensional
Schauder estimates therefore give
\[
 \sup_{Q\in\mathcal F}\sup_{R\ge R_0}
 R^\tau\|p-W\|_{C^{2,\alpha}_{\rm sc}(R<r<2R)}<\infty.
\]
Combining this with the uniform transverse estimates proves
\eqref{eq:tailQ-sc}--\eqref{eq:tailQ}, including the radial derivative
and the quotient by $\ph$.  For an ordered segment use $Q_0=Q_{A_0}$
and $M=\|Q_{A_1}\|_{L^\infty(\Cyl)}$, which is finite for the fixed
upper profile; no radial monotonicity is needed for this use of order.  The same argument applies to $A_0\le A<A^*<\infty$
under a uniform height bound, without assuming that this half-open
segment is already compact in the weighted space.
\end{remark}

\medskip\noindent
\emph{Step 2: first-mode radial derivatives.}
Set
\[
 g(r):=\La\int_{-1}^1Q(r,z)^2\ph(z)\dd z.
\]
The projected equation is
\begin{equation}\label{eq:p-radial-rig}
 -p''-\frac{m-1}{r}p'=g(r),\qquad p'(0)=0.
\end{equation}
By \eqref{eq:Q-pointwise-rig}, $0\le g(r)\le Cr^{-4}$.  Therefore
\begin{equation}\label{eq:pprime-formula-rig}
 p'(r)=-r^{1-m}\int_0^r s^{m-1}g(s)\dd s,
\end{equation}
and, because $m>4$,
\begin{equation}\label{eq:p12-rig}
 p'(r)=O(r^{-3}),\qquad
 p''(r)=-\frac{m-1}{r}p'(r)-g(r)=O(r^{-4}).
\end{equation}

\medskip\noindent
\emph{Step 3: a unified massive transverse inverse and the full
$C^{2,\alpha}_{\rm sc}$ estimate.}
Recall
\[
 \mathscr L_\perp
 =-\partial_{rr}-\frac{m-1}{r}\partial_r+H_z
 \quad\hbox{on }\ph^\perp,
 \qquad H_z\ge8\La.
\]
We first isolate the transverse inverse estimate used throughout the rest
of the proof.

\begin{definition}[Regular--decaying realization]
\label{def:regular-decaying-realization}
Let $\nu=\nu(r)\ge0$ be radial and locally $C^{0,\alpha}$ on
$[0,\infty)$.  The \emph{regular--decaying realization} of
\[
 \mathscr L_{\perp,\nu}:=\mathscr L_\perp+\nu(r)
\]
on the even transverse space $\ph^\perp$ is the restriction of
$\mathscr L_{\perp,\nu}$ to functions $u=u(r,z)$ such that
\[
 u(r,\pm1)=0,\qquad P_1u(r,\cdot)=0,
\]
$u$ extends across $r=0$ as an $O(m)$-invariant
$C^{2,\alpha}_{\rm loc}$ function of $(X,z)$, and
\[
 \lim_{r\to\infty}
 \|u\|_{C^{1,\alpha}(\mathcal U_{r,1/2})}=0.
\]
\end{definition}

The homogeneous kernel of the realization in
\cref{def:regular-decaying-realization} is trivial.  Indeed, expand
\[
 u(r,z)=\sum_{k\ge1}u_k(r)\phi_k(z),
 \qquad H_z\phi_k=\mu_k\phi_k,\qquad \mu_k\ge8\La.
\]
Each coefficient solves
\[
 -u_k''-\frac{m-1}{r}u_k'+(\mu_k+\nu(r))u_k=0,
 \qquad u_k'(0)=0,
\]
or equivalently
\[
 (r^{m-1}u_k')'=(\mu_k+\nu(r))r^{m-1}u_k.
\]
If $u_k(0)>0$, then $u_k'>0$ for as long as $u_k>0$, so $u_k$
cannot converge to zero at infinity; if $u_k(0)<0$, the same argument
applied to $-u_k$ gives the same contradiction.  Hence $u_k(0)=0$, and
then $u_k'(0)=0$ and ODE uniqueness imply $u_k\equiv0$.  Thus
$u\equiv0$.

\begin{claim}[Massive transverse tail estimate]\label{clm:massive-transverse-inverse}
Fix $q>0$ and $R_*\ge4$.  For a transverse function $F\perp\ph$ set
\begin{align*}
 \|F\|_{Y_q}
 :=&\ \|F\|_{L^2(\{r<2R_*\}\times(-1,1))}\\
 &+\sup_{r\ge R_*}(1+r)^q
 \left(
  \|F(r,\cdot)\|_{L^2_z}
  +\|F\|_{C^{0,\alpha}(\mathcal U_r)}
 \right),
\end{align*}
where
\[
 \mathcal U_r=(r-1,r+1)\times(-1,1),\qquad
 \mathcal U_{r,1/2}=(r-\tfrac12,r+\tfrac12)\times(-1,1),
\]
and set
\[
 \|u\|_{X_q}
 :=\sup_{r\ge R_*}(1+r)^q
 \|u\|_{C^{2,\alpha}(\mathcal U_{r,1/2})}.
\]
Let $\nu=\nu(r)\ge0$ be radial, locally $C^{0,\alpha}$ on
$[0,\infty)$, and satisfy
\[
 \sup_{r\ge R_*}(1+r)^2
 \|\nu\|_{C^{0,\alpha}(r-1,r+1)}<\infty.
\]
Let $F=F(r,z)$ be $O(m)$-invariant in $X$, even in $z$, satisfy
$P_1F=0$, and extend across $r=0$ as an $O(m)$-invariant
$C^{0,\alpha}_{\rm loc}$ function of $(X,z)$ up to the flat boundary
$z=\pm1$.  If $\|F\|_{Y_q}<\infty$, then the regular--decaying
realization in \cref{def:regular-decaying-realization} has a unique
solution $u$, and
\begin{equation}\label{eq:massive-inverse-XY}
 \|u\|_{X_q}
 \le C_{q,\mathcal V}\|F\|_{Y_q}.
\end{equation}
The constant is uniform when $\nu$ ranges over any family
$\mathcal V$ with a common local $C^{0,\alpha}$ bound on compact radial
intervals and a common bound in the displayed tail seminorm.
\end{claim}

\begin{proof}
Let $\{\phi_k\}_{k\ge1}$ be an orthonormal basis of even Dirichlet
eigenfunctions of $H_z$ on $\ph^\perp$, with
\[
 H_z\phi_k=\mu_k\phi_k,\qquad \mu_k\ge\mu_*:=8\La.
\]
For $\nu\equiv0$, the scalar equation for the $k$-th transverse mode is
\[
 -u_k''-\frac{m-1}{r}u_k'+\mu_ku_k=f_k.
\]
Writing $\nu_0=(m-2)/2$, its regular and decaying homogeneous solutions
are
\[
 u_{0,\mu}(r)=r^{-\nu_0}I_{\nu_0}(\sqrt\mu r),\qquad
 u_{\infty,\mu}(r)=r^{-\nu_0}K_{\nu_0}(\sqrt\mu r).
\]
Their weighted Wronskian is $-1$, and the regular--decaying solution is
\begin{equation}\label{eq:massive-green-rig}
 u(r)=u_{\infty,\mu}(r)\int_0^r u_{0,\mu}(s)f(s)s^{m-1}\dd s
      +u_{0,\mu}(r)\int_r^\infty u_{\infty,\mu}(s)f(s)s^{m-1}\dd s.
\end{equation}
The standard bounds for $I_{\nu_0}$ and $K_{\nu_0}$ imply that, for every
$q>0$, there is $C_q$, independent of $\mu\ge\mu_*$, such that
\begin{equation}\label{eq:massive-weight-rig}
 \int_0^\infty G_\mu(r,s)(1+s)^{-q}s^{m-1}\dd s
 \le C_q(1+r)^{-q}.
\end{equation}
For a general nonnegative $\nu$, the scalar homogeneous equation
\[
 -a''-\frac{m-1}{r}a'+(\mu+\nu(r))a=0
\]
is a positive-mass singular Sturm--Liouville equation.  Standard
Sturm--Liouville theory gives, up to normalization, the solution regular at
$r=0$ and a positive recessive (hence decaying) solution at infinity; see,
for example, Zettl \cite[Chs.~7--10]{Zettl}.  In the present scalar setting
these solutions may equivalently be obtained by the usual initial-value and
Volterra constructions.  Their weighted Wronskian is nonzero:
otherwise they would generate a nontrivial regular--decaying homogeneous
solution, contradicting the kernel triviality proved above.  They therefore
define a positive regular--decaying Green kernel $G_{\mu,\nu}$.  The
quadratic-form ordering and the maximum principle then give
\[
 0\le G_{\mu,\nu}(r,s)\le G_\mu(r,s).
\]
Hence \eqref{eq:massive-weight-rig} remains valid uniformly for the
operators $\mathscr L_{\perp,\nu}$ considered here.

Expanding
\[
 F(r,z)=\sum_{k\ge1}f_k(r)\phi_k(z),\qquad
 u(r,z)=\sum_{k\ge1}u_k(r)\phi_k(z),
\]
and denoting by $G_{k,\nu}(r,s)$ the regular--decaying Green kernel of
\[
 -\partial_{rr}-\frac{m-1}{r}\partial_r+\mu_k+\nu(r),
\]
we have
\[
 u_k(r)=\int_0^\infty G_{k,\nu}(r,s)f_k(s)s^{m-1}\dd s.
\]
The maximum principle gives monotonicity both in the mass and in the
nonnegative potential:
\[
 0\le G_{k,\nu}(r,s)\le G_{\mu_k,0}(r,s)
 \le G_{\mu_*,0}(r,s)=:G_*(r,s).
\]
Hence Minkowski's integral inequality followed by Parseval's identity
gives the pointwise-in-$r$ estimate
\begin{align*}
 \|u(r,\cdot)\|_{L^2_z}
 &=\left(\sum_k\left|
   \int_0^\infty G_{k,\nu}(r,s)f_k(s)s^{m-1}\dd s
   \right|^2\right)^{1/2}\\
 &\le \int_0^\infty G_*(r,s)
     \left(\sum_k|f_k(s)|^2\right)^{1/2}s^{m-1}\dd s\\
 &=\int_0^\infty G_*(r,s)
     \|F(s,\cdot)\|_{L^2_z}s^{m-1}\dd s.
\end{align*}
We split the last integral at $2R_*$.  On $s\ge2R_*$ the definition of
$Y_q$ and \eqref{eq:massive-weight-rig} imply
\[
 \int_{2R_*}^\infty G_*(r,s)\|F(s,\cdot)\|_{L^2_z}s^{m-1}\dd s
 \le C_q(1+r)^{-q}\|F\|_{Y_q}.
\]
For the compact part, Cauchy--Schwarz yields
\begin{align*}
 &\int_0^{2R_*}G_*(r,s)\|F(s,\cdot)\|_{L^2_z}s^{m-1}\dd s\\
 &\qquad\le
 \left(\int_0^{2R_*}G_*(r,s)^2s^{m-1}\dd s\right)^{1/2}
 \|F\|_{L^2(\{s<2R_*\}\times(-1,1))}.
\end{align*}
If $R_*\le r\le4R_*$, the factor in parentheses is bounded by
compactness.  If $r\ge4R_*$, the Bessel asymptotics in
\eqref{eq:massive-green-rig} give
\[
 \left(\int_0^{2R_*}G_*(r,s)^2s^{m-1}\dd s\right)^{1/2}
 \le Ce^{-cr}.
\]
Since a polynomial factor is absorbed by the exponential, in both cases
\[
 (1+r)^q
 \int_0^{2R_*}G_*(r,s)\|F(s,\cdot)\|_{L^2_z}s^{m-1}\dd s
 \le C_{q,R_*}\|F\|_{Y_q}.
\]
Combining the two pieces proves
\begin{equation}\label{eq:massive-L2-general}
 \sup_{r\ge R_*}(1+r)^q\|u(r,\cdot)\|_{L^2_z}
 \le C_q\|F\|_{Y_q}.
\end{equation}
The only summability used is the $L^2_z$ orthogonal decomposition;
the uniform spectral gap dominates every mode by the single kernel $G_*$.

The Green representation first defines a locally $H^1$ weak solution
with the stated symmetry, transverse orthogonality, and zero Dirichlet
trace.  Because $F$ and $\nu$ are locally $C^{0,\alpha}$, standard
interior and flat-boundary Schauder estimates upgrade this solution to
$C^{2,\alpha}_{\rm loc}$ away from the radial axis.  Near $r=0$ we
interpret the equation in the full $(X,z)$ variables; the assumed
$O(m)$-invariant $C^{0,\alpha}_{\rm loc}$ extension of $F$ and the radial
$C^{0,\alpha}_{\rm loc}$ potential give the same Schauder upgrade across
the axis.  Thus the solution has the local regularity required in
\cref{def:regular-decaying-realization}.  The tail estimate below then
also gives the required decay at infinity.

We next upgrade \eqref{eq:massive-L2-general} to the unit-strip
$C^{2,\alpha}$ estimate.  Fix $r_0\ge R_*+2$.  On
$\mathcal U_{r_0}$ the equation is the uniformly elliptic Dirichlet
problem
\[
 -u_{rr}-u_{zz}-\frac{m-1}{r}u_r+(\nu(r)-\La)u=F.
\]
The coefficients have uniformly bounded $C^{0,\alpha}$ norms because
$r\ge R_*$ and
$\|(1+r)^2\nu\|_{C^{0,\alpha}(r-1,r+1)}$ is uniformly bounded.  Standard
interior and flat-boundary $W^{2,2}$ estimates on the slightly smaller
strip give
\[
 \|u\|_{W^{2,2}(\mathcal U_{r_0,3/4})}
 \le C\bigl(\|u\|_{L^2(\mathcal U_{r_0})}
            +\|F\|_{L^2(\mathcal U_{r_0})}\bigr).
\]
In two variables $W^{2,2}$ embeds into $C^0$ (indeed into
$C^{0,\gamma}$ for every $\gamma<1$), and the interior/Dirichlet
Schauder estimate therefore yields
\begin{align*}
 \|u\|_{C^{2,\alpha}(\mathcal U_{r_0,1/2})}
 &\le C\bigl(\|u\|_{C^0(\mathcal U_{r_0,3/4})}
             +\|F\|_{C^{0,\alpha}(\mathcal U_{r_0,3/4})}\bigr)\\
 &\le C\bigl(\|u\|_{L^2(\mathcal U_{r_0})}
             +\|F\|_{C^{0,\alpha}(\mathcal U_{r_0})}\bigr).
\end{align*}
By \eqref{eq:massive-L2-general}, the first term is bounded by
$C(1+r_0)^{-q}\|F\|_{Y_q}$; the second one is bounded in exactly the
same way by the definition of $Y_q$.  Hence
\[
 (1+r_0)^q
 \|u\|_{C^{2,\alpha}(\mathcal U_{r_0,1/2})}
 \le C_q\|F\|_{Y_q}.
\]
The remaining bounded range $R_*\le r_0\le R_*+2$ can be covered by
finitely many unit strips and is absorbed into the constant by the same local
estimate and the compact $L^2$ term in $Y_q$.  Taking the supremum proves \eqref{eq:massive-inverse-XY}.
\end{proof}

The commuted radial equations below contain the inverse-square terms
$j(m-1)r^{-2}$.  Since only their tail estimates are used, we avoid
introducing a singular realization at the axis and reduce them to the
preceding smooth-potential claim.

\begin{claim}[Tail reduction for inverse-square commutators]
\label{clm:inverse-square-tail}
Fix $j\in\{1,2,3\}$ and $q>0$. Suppose that $u=u(r,z)$ is even in
$z$, locally $C^{2,\alpha}$ for $r>0$, and satisfies
\[
 u(r,\pm1)=0,\qquad P_1u(r,\cdot)=0,
\]
\[
 \left(\mathscr L_\perp+\frac{j(m-1)}{r^2}\right)u=f
 \qquad (r>0),
\]
where $f=f(r,z)$ is even in $z$, satisfies $P_1f=0$, belongs to
$C^{0,\alpha}_{\rm loc}(\{r>0,\ |z|\le1\})$, and has finite $Y_q$-norm,
\[
 \|f\|_{Y_q}<\infty.
\]
Moreover,
\[
 \lim_{r\to\infty}
 \|u\|_{C^{1,\alpha}(\mathcal U_{r,1/2})}=0.
\]
Then
\begin{equation}\label{eq:inverse-square-tail-est}
 \sup_{r\ge R_*}(1+r)^q
 \|u\|_{C^{2,\alpha}(\mathcal U_{r,1/2})}
 \le C_q\left(
   \|f\|_{Y_q}
   +\|u\|_{C^{1,\alpha}(\{R_*/2<r<R_*+1\}\times(-1,1))}
 \right).
\end{equation}
For families with a uniform local $C^{1,\alpha}$ bound on the fixed
annulus, the estimate is uniform.
\end{claim}

\begin{proof}
Choose $\chi\in C^\infty([0,\infty))$ with $\chi=0$ on
$[0,R_*/2]$ and $\chi=1$ on $[R_*,\infty)$. Choose a smooth
nonnegative radial potential $\widetilde\nu_j$ on $[0,\infty)$ such
that
\[
 \widetilde\nu_j(r)=\frac{j(m-1)}{r^2}
 \qquad\text{for }r\ge R_*/2.
\]
Then $\widetilde u:=\chi u$ satisfies
\[
 (\mathscr L_\perp+\widetilde\nu_j)\widetilde u
 =\chi f+\mathcal C_\chi[u],
\]
where
\[
 \mathcal C_\chi[u]
 =-2\chi'u_r-\chi''u-\frac{m-1}{r}\chi'u
\]
is supported in the fixed annulus $R_*/2<r<R_*$. Hence
\[
 \|\chi f+\mathcal C_\chi[u]\|_{Y_q}
 \le C_q\left(
  \|f\|_{Y_q}
  +\|u\|_{C^{1,\alpha}(\{R_*/2<r<R_*+1\}\times(-1,1))}
 \right).
\]
The assumptions on $u$ imply
\[
 \widetilde u(r,\pm1)=0,\qquad
 P_1\widetilde u(r,\cdot)=0,\qquad
 \lim_{r\to\infty}
 \|\widetilde u\|_{C^{1,\alpha}(\mathcal U_{r,1/2})}=0,
\]
and $\widetilde u$ vanishes near the axis. Hence $\widetilde u$
belongs to the regular--decaying realization in
\cref{def:regular-decaying-realization}. By the uniqueness just proved for that realization, $\widetilde u$ is
the unique regular--decaying solution with source
$\chi f+\mathcal C_\chi[u]$.  Applying \eqref{eq:massive-inverse-XY} and using
$\widetilde u=u$ for $r\ge R_*$ proves
\eqref{eq:inverse-square-tail-est}.
\end{proof}

We first apply the massive inverse to $w$.  Before doing so, we record the
unit-scale coefficient bound needed in the source norm.  The pointwise
estimate \eqref{eq:Q-pointwise-rig} and a local $W^{2,p}$--Schauder
bootstrap for \eqref{eq:cylinder}, including the Dirichlet boundary
charts at $z=\pm1$, give
\begin{equation}\label{eq:Q-unit-C2}
 \|Q\|_{C^{2,\alpha}(\mathcal U_r)}\le Cr^{-2}
 \qquad(r\gg1).
\end{equation}
Hence $F=\La P^\perp(Q^2)$ is $O(m)$-invariant, even in $z$,
fibrewise orthogonal to $\ph$, locally $C^{0,\alpha}$ up to the axis and
flat boundary, and belongs to $Y_4$.  Thus
\eqref{eq:massive-inverse-XY} with $\nu=0$ yields
\begin{equation}\label{eq:w-unit-prelim}
 \|w\|_{C^{2,\alpha}(\mathcal U_{r,1/2})}\le Cr^{-4}.
\end{equation}
In particular $w_r,w_{rr}=O_{C_z^{0,\alpha}}(r^{-4})$.  Combining this
with \eqref{eq:p12-rig} gives
\[
 Q_r=p'\ph+w_r=O_{C_z^{0,\alpha}}(r^{-3}),
\]
and therefore
\begin{equation}\label{eq:Fr-rig}
 F_r=2\La P^\perp(QQ_r)=O_{C_z^{0,\alpha}}(r^{-5}).
\end{equation}
Differentiating \eqref{eq:w-eq-rig} once gives
\begin{equation}\label{eq:wr-eq-rig}
 \left(\mathscr L_\perp+\frac{m-1}{r^2}\right)w_r=F_r.
\end{equation}
Since $w(r,\pm1)=0$ and $P_1w=0$, differentiation in $r$ gives
\[
 w_r(r,\pm1)=0,\qquad P_1w_r(r,\cdot)=0.
\]
Moreover, \eqref{eq:w-unit-prelim} implies
\[
 \|w_r\|_{C^{1,\alpha}(\mathcal U_{r,1/2})}
 \le Cr^{-4}\longrightarrow0.
\]
The source $F_r$ is even, fibrewise orthogonal to $\ph$, and locally
$C^{0,\alpha}$ for $r>0$.  Thus \cref{clm:inverse-square-tail} applies
with $j=1$ and $q=5$, and gives
\begin{equation}\label{eq:wr-C2-rig}
 \|w_r\|_{C^{2,\alpha}(\mathcal U_{r,1/2})}\le Cr^{-5}.
\end{equation}
Hence $w_{rr}=O_{C_z^{0,\alpha}}(r^{-5})$ and
$Q_{rr}=p''\ph+w_{rr}=O_{C_z^{0,\alpha}}(r^{-4})$.  Thus
\begin{equation}\label{eq:Frr-rig}
 F_{rr}=2\La P^\perp(Q_r^2+QQ_{rr})
 =O_{C_z^{0,\alpha}}(r^{-6}).
\end{equation}
Differentiating twice gives
\begin{equation}\label{eq:wrr-eq-rig}
 \left(\mathscr L_\perp+\frac{2(m-1)}{r^2}\right)w_{rr}
 =F_{rr}+\frac{2(m-1)}{r^3}w_r.
\end{equation}
The right-hand side is even, fibrewise orthogonal to $\ph$, locally
$C^{0,\alpha}$ for $r>0$, and belongs to $Y_6$.  In addition,
\eqref{eq:wr-C2-rig} gives
\[
 w_{rr}(r,\pm1)=0,\qquad P_1w_{rr}(r,\cdot)=0,
 \qquad
 \|w_{rr}\|_{C^{1,\alpha}(\mathcal U_{r,1/2})}
 \le Cr^{-5}\longrightarrow0.
\]
Hence \cref{clm:inverse-square-tail} applies with $j=2$ and $q=6$, and yields
\begin{equation}\label{eq:wrr-C2-rig}
 \|w_{rr}\|_{C^{2,\alpha}(\mathcal U_{r,1/2})}\le Cr^{-6}.
\end{equation}

To control the scaled H\"older seminorm of the second $X$ derivatives,
we record one more radial derivative.  Since
$g(r)=\La\int Q^2\ph$, the already proved bounds give $g'(r)=O(r^{-5})$;
differentiating \eqref{eq:p-radial-rig} yields
\begin{equation}\label{eq:p3-rig}
 p'''(r)=O(r^{-5}).
\end{equation}
The estimate \eqref{eq:wrr-C2-rig} first gives $w_{rrr}=O(r^{-6})$,
so $Q_{rrr}=p'''\ph+w_{rrr}=O(r^{-5})$ and
\begin{equation}\label{eq:Frrr-rig}
 F_{rrr}=2\La P^\perp(3Q_rQ_{rr}+QQ_{rrr})
 =O_{C_z^{0,\alpha}}(r^{-7}).
\end{equation}
Differentiating \eqref{eq:wrr-eq-rig} gives
\begin{equation}\label{eq:wrrr-eq-rig}
 \left(\mathscr L_\perp+\frac{3(m-1)}{r^2}\right)w_{rrr}
 =F_{rrr}+\frac{6(m-1)}{r^3}w_{rr}
       -\frac{6(m-1)}{r^4}w_r.
\end{equation}
The source is even, fibrewise orthogonal to $\ph$, locally
$C^{0,\alpha}$ for $r>0$, and belongs to $Y_7$.  Moreover, \eqref{eq:wrr-C2-rig} gives
\[
 w_{rrr}(r,\pm1)=0,\qquad P_1w_{rrr}(r,\cdot)=0,
 \qquad
 \|w_{rrr}\|_{C^{1,\alpha}(\mathcal U_{r,1/2})}
 \le Cr^{-6}\longrightarrow0.
\]
Therefore \cref{clm:inverse-square-tail} applies with $j=3$ and $q=7$, and gives
\begin{equation}\label{eq:wrrr-rig}
 \|w_{rrr}\|_{C^{2,\alpha}(\mathcal U_{r,1/2})}\le Cr^{-7}.
\end{equation}
Consequently
\begin{align}\label{eq:w-natural-derivatives}
 \|w(r,\cdot)\|_{C_z^{2,\alpha}}&\le Cr^{-4},\notag\\
 \|w_r(r,\cdot)\|_{C_z^{2,\alpha}}&\le Cr^{-5},\\
 \|w_{rr}(r,\cdot)\|_{C_z^{2,\alpha}}&\le Cr^{-6},\notag\\
 \|w_{rrr}(r,\cdot)\|_{C_z^{2,\alpha}}&\le Cr^{-7}.\notag
\end{align}

For $R$ large set $W_R(Y,z)=R^4w(RY,z)$ on
$\mathcal A=\{1<|Y|<2,\ |z|<1\}$.  Since $w$ is radial in $X$,
\[
 \partial_{X_i}w=w_r\frac{X_i}{r},\qquad
 \partial_{X_iX_j}w
 =w_{rr}\frac{X_iX_j}{r^2}
 +\frac{w_r}{r}\left(\delta_{ij}-\frac{X_iX_j}{r^2}\right),
\]
and the analogous third-derivative formulas hold.  Therefore
\eqref{eq:w-natural-derivatives} gives
\[
 \|W_R\|_{C^{2,\alpha}(\mathcal A)}
 \le C\sup_{R\le r\le2R}\Big(
 R^4\|w\|_{C_z^{2,\alpha}}
 +R^5\|w_r\|_{C_z^{2,\alpha}}
 +R^6\|w_{rr}\|_{C_z^{2,\alpha}}
 +R^7\|w_{rrr}\|_{L_z^\infty}\Big)
 \le C.
\]
Hence
\begin{equation}\label{eq:w-sc-r4-proof}
 w=O_{C^{2,\alpha}_{\rm sc}}(r^{-4}),
\end{equation}
which proves \eqref{eq:w-sc-r4}.

The same claim also contains the transverse estimate needed in the
bootstrap below.  If, for some $a>2$,
\[
 |Q(r,z)|\le Cr^{-a}\ph(z),
\]
then $F=\La P^\perp(Q^2)$ belongs to $Y_{2a}$.  Applying
\eqref{eq:massive-inverse-XY} with $q=2a$ gives
\begin{equation}\label{eq:bootstrap-w-r2a}
 \|w\|_{C^{2,\alpha}(\mathcal U_{r,1/2})}\le Cr^{-2a}.
\end{equation}
Since $w(r,\pm1)=0$, the $C_z^1$ part and
$\ph(z)\asymp1-|z|$ imply
\begin{equation}\label{eq:bootstrap-w-phi}
 |w(r,z)|\le Cr^{-2a}\ph(z).
\end{equation}
Thus the later Newton bootstrap uses only the unified inverse claim above.

Finally $w(r,\pm1)=w_r(r,\pm1)=0$.  The $C_z^1$ bounds in
\eqref{eq:w-natural-derivatives}, together with
$\ph(z)\asymp1-|z|$ near the boundary, yield
\[
 \sup_{|z|<1}\frac{|w(r,z)|}{\ph(z)}\le Cr^{-4},\qquad
 \sup_{|z|<1}\frac{|w_r(r,z)|}{\ph(z)}\le Cr^{-5},
\]
which is \eqref{eq:w-minimal-tail}.

\medskip\noindent
\emph{Step 4: convergence of the first mode.}
The first estimate in \eqref{eq:w-minimal-tail} is already sufficient:
\[
 \int_{-1}^1Q^2\ph\dd z
 =B_2p^2+2p\int_{-1}^1w\ph^2\dd z
   +\int_{-1}^1w^2\ph\dd z
 =B_2p^2+O(r^{-6}).
\]
Thus
\begin{equation}\label{eq:p-asym-ode-rig}
 -p''-\frac{m-1}{r}p'=c_0p^2+O(r^{-6}),
 \qquad c_0=\La B_2.
\end{equation}
Set $t=\log r$ and $Y(t)=c_0r^2p(r)$.  By
\cref{lem:scalar-separation}, $0<Y<2(m-4)$ and
\begin{equation}\label{eq:Y-rig}
 Y''+(m-6)Y'-2(m-4)Y+Y^2:=\rho(t)=O(e^{-2t}).
\end{equation}
Define
\[
 \mathcal E(t)=\frac12(Y')^2-(m-4)Y^2+\frac13Y^3.
\]

\begin{claim}\label{clm:Y-energy-limit}
    $\mathcal E$ has a finite limit and $Y'\in L^2(t_0,\infty)$.
\end{claim} 
\begin{proof}
From \eqref{eq:Y-rig}, we obtain
\[
\mathcal E'(t)
=
-(m-6)(Y'(t))^2+O(e^{-2t}|Y'(t)|).
\]
By Young's inequality, for some constant \(C>0\),
\[
\mathcal E'(t)
\le
-\frac{m-6}{2}(Y'(t))^2+Ce^{-4t}.
\]

We first prove that \(Y'\in L^2(t_0,\infty)\). Since
\(
0<Y(t)<2(m-4),
\)
there exists \(C_m>0\) such that
\[
\mathcal E(t)\ge -C_m
\qquad\text{for all }t\ge t_0.
\]
Integrating \(\mathcal E'\) over \([t_0,T]\), we obtain
\[
\mathcal E(T)-\mathcal E(t_0)
\le
-\frac{m-6}{2}\int_{t_0}^{T}|Y'(t)|^2\,dt
+
C\int_{t_0}^{T}e^{-4t}\,dt.
\]
Therefore, 
\[
\frac{m-6}{2}
\int_{t_0}^{T}|Y'(t)|^2\,dt
\le
\mathcal E(t_0)+C_m
+
C\int_{t_0}^{\infty}e^{-4t}\,dt.
\]
The right-hand side is independent of \(T\). Since \(m>6\), letting
\(T\to\infty\) yields
\[
Y'\in L^2(t_0,\infty).
\]

We next prove that \(\mathcal E(t)\) has a finite limit as
\(t\to\infty\). Write
\[
\mathcal E'(t)
=
-(m-6)|Y'(t)|^2+R(t),
\qquad
|R(t)|\le Ce^{-2t}|Y'(t)|.
\]
By the Cauchy--Schwarz inequality,
\[
\begin{aligned}
\int_{t_0}^{\infty}|R(t)|\,dt
&\le
C\int_{t_0}^{\infty}e^{-2t}|Y'(t)|\,dt\\
&\le
C
\left(
\int_{t_0}^{\infty}e^{-4t}\,dt
\right)^{1/2}
\left(
\int_{t_0}^{\infty}|Y'(t)|^2\,dt
\right)^{1/2}
<\infty.
\end{aligned}
\]
Consequently,
\[
\int_{t_0}^{\infty}|\mathcal E'(t)|\,dt
\le
(m-6)
\int_{t_0}^{\infty}|Y'(t)|^2\,dt
+
\int_{t_0}^{\infty}|R(t)|\,dt
<\infty.
\]
Hence \(\mathcal E'\in L^1(t_0,\infty)\). In particular, for
\(T_2>T_1\ge t_0\),
\[
|\mathcal E(T_2)-\mathcal E(T_1)|
\le
\int_{T_1}^{T_2}|\mathcal E'(t)|\,dt
\longrightarrow0
\qquad
\text{as }T_1,T_2\to\infty.
\]
Thus \(\mathcal E(t)\) is Cauchy at infinity, and therefore there exists
\(\mathcal E_\infty\in\mathbb R\) such that
\[
\mathcal E(t)\longrightarrow\mathcal E_\infty
\qquad\text{as }t\to\infty.
\]

\end{proof}

Since \(Y\) is uniformly bounded and \(\mathcal E\) has a finite
limit, \(Y'\) is uniformly bounded on \([t_0,\infty)\). Returning to
\eqref{eq:Y-rig}, we then obtain
\[
\|Y''\|_{L^\infty(t_0,\infty)}<\infty.
\]
Therefore \(Y'\) is uniformly Lipschitz, and in particular uniformly
continuous, on \([t_0,\infty)\).

\begin{claim}\label{clm:Y-omega-limit}
   \(Y'(t)\longrightarrow0\) as \(t\to\infty\), and \(Y(t)\to0\,\,\text{or}\,\, Y(t)\to2(m-4).\)
\end{claim}
\begin{proof}
     If \(Y'(t)\longrightarrow0\) were false, there would exist \(\varepsilon>0\) and a
sequence \(t_j\to\infty\) such that
\(
|Y'(t_j)|\ge\varepsilon.
\)
By the uniform continuity of \(Y'\), there exists \(\delta>0\),
independent of \(j\), such that
\[
|Y'(t)|\ge\frac{\varepsilon}{2}
\qquad
\text{for }t\in(t_j-\delta,t_j+\delta).
\]
After passing to a subsequence, these intervals may be assumed pairwise
disjoint. Consequently,
\[
\int_{t_0}^{\infty}|Y'(t)|^2\,dt
\ge
\sum_j
\int_{t_j-\delta}^{t_j+\delta}|Y'(t)|^2\,dt
\ge
\sum_j \frac{\varepsilon^2}{4}(2\delta)
=\infty,
\]
contradicting \(Y'\in L^2(t_0,\infty)\). Hence
\[
Y'(t)\to0.
\]

We next identify the possible limit points of \(Y\). Let \(L\) be any
limit point of \(Y\). Then there exists a sequence \(t_j\to\infty\)
such that
\(
Y(t_j)\to L.
\)
For \(s\) in a fixed bounded interval, define
\[
Y_j(s):=Y(t_j+s).
\]
Since
\[
\sup_{|s|\le R}|Y_j'(s)|
=
\sup_{|s|\le R}|Y'(t_j+s)|
\longrightarrow0
\qquad
\text{for every fixed }R>0,
\]
we obtain
\[
|Y_j(s)-Y_j(0)|
\le
R\sup_{|\tau|\le R}|Y_j'(\tau)|
\longrightarrow0
\]
uniformly for \(|s|\le R\). Since \(Y_j(0)=Y(t_j)\to L\), it follows that for every fixed \(R>0\)
\begin{equation}
    Y_j\longrightarrow L,\,Y_j'\longrightarrow0
\qquad\text{uniformly on }[-R,R].\label{eq:Y-translate-limit}
\end{equation}

Note that
\[
Y_j''
=
-(m-6)Y_j'
+2(m-4)Y_j-Y_j^2+\rho(t_j+s).
\]
Since \(\rho(t)=O(e^{-2t})\), for every fixed \(R>0\),
\[
\sup_{|s|\le R}|\rho(t_j+s)|\longrightarrow0.
\]
Using \eqref{eq:Y-translate-limit}, we conclude that
\[
Y_j''(s)
\longrightarrow
2(m-4)L-L^2
\]
uniformly on $[-R,R]$.

Fix \(R>0\). Integrating from \(0\) to \(R\) gives
\[
Y_j'(R)-Y_j'(0)
=
\int_0^R Y_j''(\tau)\,d\tau.
\]
The left-hand side tends to \(0\) by \eqref{eq:Y-translate-limit}, whereas the right-hand
side tends to
\[
R\bigl(2(m-4)L-L^2\bigr).
\]
Hence
\[
2(m-4)L-L^2=0.
\]

Let
\[
a:=2(m-4),
\]
and denote by
\[
\omega(Y)
:=
\left\{
L\in\mathbb R:
\exists\, t_j\to\infty
\text{ such that }Y(t_j)\to L
\right\}
\]
the \(\omega\)-limit set of \(Y\).

From the preceding argument, 
\[
\omega(Y)\subset \{0,a\}.
\]

We now show that \(\omega(Y)\) is connected. For every \(T\ge t_0\), set
\[
K_T:=\overline{Y([T,\infty))}.
\]
Since \([T,\infty)\) is connected and \(Y\) is continuous,
\(Y([T,\infty))\) is connected. Therefore \(K_T\) is connected as well.
Indeed, \(\{K_T\}_{T\ge t_0}\) is a decreasing family of nonempty compact
connected sets.

By definition of the \(\omega\)-limit set, we have
\[
\omega(Y)
=
\bigcap_{T\ge t_0} K_T.
\]
Since the intersection of a decreasing family of nonempty compact
connected subsets of \(\mathbb R\) is again nonempty and connected,
\(\omega(Y)\) is connected. Therefore,
\[
\omega(Y)=\{0\}
\qquad\text{or}\qquad
\omega(Y)=\{a\}.
\]
That is
\[
Y(t)\to0
\qquad\text{or}\qquad
Y(t)\to2(m-4).
\]

\end{proof}

Suppose $Y\to0$
as \(t\to\infty.\)
Write
\[
b:=m-4>2.
\]
Then equation \eqref{eq:Y-rig} can be rewritten as
\begin{equation}
    Y''+(b-2)Y'-2bY=F(t),\label{eq:Y-forced-linearization}
\end{equation}
where
\[
F(t)=-Y(t)^2+\rho(t),
\qquad
|\rho(t)|\le Ce^{-2t}.
\]

\begin{claim}
   For every $\gamma\in (0,2)$, we have \(|Y(t)|+|Y'(t)|
\le
C_\gamma e^{-\gamma t},\quad t\gg 1.\)
\end{claim}

\begin{proof}

Define
\[
c_+(t):=\frac{Y'(t)+bY(t)}{b+2},
\qquad
c_-(t):=\frac{2Y(t)-Y'(t)}{b+2}.
\]
Then
\[
Y=c_++c_-,
\qquad
Y'=2c_+-bc_-.
\]
A direct calculation using \eqref{eq:Y-forced-linearization} gives
\begin{equation}
    c_+'-2c_+=\frac{F}{b+2},
\qquad
c_-'+bc_-=-\frac{F}{b+2}.\label{eq:Y-dichotomy-system}
\end{equation}

Since \(Y(t),Y'(t)\to0\), we also have
\[
c_+(t),c_-(t)\to0.
\]
Therefore, the homogeneous mode \(Ce^{2t}\) is absent. Multiplying
the first equation in \eqref{eq:Y-dichotomy-system} by \(e^{-2t}\) and integrating from
\(t\) to \(+\infty\), we obtain
\begin{equation}
    c_+(t)
=
-\frac1{b+2}
\int_t^\infty e^{2(t-s)}F(s)\,ds.\label{eq:Y-growing-coordinate}
\end{equation}
Fixing \(T>0\) and integrating the second equation
in \eqref{eq:Y-dichotomy-system} from \(T\) to \(t\), we obtain
\begin{equation}
    c_-(t)
=
e^{-b(t-T)}c_-(T)
-\frac1{b+2}
\int_T^t e^{-b(t-s)}F(s)\,ds.\label{eq:Y-decaying-coordinate}
\end{equation}

We now prove exponential decay. Fix
\(
0<\gamma<2.
\)
Since \(Y(t)\to0\), after increasing \(T\) we may assume
\[
|Y(t)|\le\varepsilon
\qquad (t\ge T),
\]
where \(\varepsilon>0\) will be chosen sufficiently small. Therefore
\begin{equation}
|F(t)|
\le
|Y(t)|^2+Ce^{-2t}
\le
\varepsilon\bigl(|c_+(t)|+|c_-(t)|\bigr)+Ce^{-2t}.
\label{eq:Y-forcing-smallness}
\end{equation}

We now use the half-line exponential dichotomy directly.  Let
\[
 X_0:=C_0([T,\infty);\mathbb R^2),\qquad
 \|(u_+,u_-)\|_0:=\sup_{t\ge T}(|u_+(t)|+|u_-(t)|),
\]
and, for the fixed $0<\gamma<2$,
\[
 X_\gamma:=\left\{u\in X_0:\
 \|u\|_\gamma:=\sup_{t\ge T}e^{\gamma(t-T)}
 (|u_+(t)|+|u_-(t)|)<\infty\right\}.
\]
For $u=(u_+,u_-)$ put
\[
 \mathcal F(t,u):=-(u_+(t)+u_-(t))^2+\rho(t)
\]
and define
\begin{align*}
 (\mathcal T_Tu)_+(t)
 &:=-\frac1{b+2}\int_t^\infty e^{-2(s-t)}
          \mathcal F(s,u)\,ds,\\
 (\mathcal T_Tu)_-(t)
 &:=e^{-b(t-T)}c_-(T)
   -\frac1{b+2}\int_T^t e^{-b(t-s)}
          \mathcal F(s,u)\,ds.
\end{align*}
By \eqref{eq:Y-growing-coordinate}--\eqref{eq:Y-decaying-coordinate}, the actual pair
$c=(c_+,c_-)$ is a fixed point of $\mathcal T_T$ in $X_0$.

Choose $\delta>0$ so small that $C\delta<1/4$, where $C$ is larger
than the $L^1$ norms of the two dichotomy kernels multiplied by
$2/(b+2)$.  Since $c(t)\to0$ and $\rho(t)=O(e^{-2t})$, we may increase
$T$ so that
\[
 \|c\|_{L^\infty([T,\infty))}\le\frac\delta2,
 \qquad
 |c_-(T)|+Ce^{-2T}\le\frac\delta4.
\]
On the closed $X_0$--ball $\overline B_\delta$ one has
\[
 |\mathcal F(t,u)-\mathcal F(t,v)|
 \le 2\delta\bigl(|u_+(t)-v_+(t)|+|u_-(t)-v_-(t)|\bigr),
\]
so the two integral formulas imply
\[
 \|\mathcal T_Tu-\mathcal T_Tv\|_0
 \le C\delta\|u-v\|_0.
\]
The same kernel bounds and the preceding choice of $T$ show that
$\mathcal T_T(\overline B_\delta)\subset\overline B_\delta$.
Consequently $\mathcal T_T$ has at most one fixed point in that ball.

We next solve the same integral equation in $X_\gamma$.  If
$\|u\|_0\le\delta$, then
$|(u_++u_-)^2|\le2\delta(|u_+|+|u_-|)$.  Hence
\begin{align*}
 \sup_{t\ge T}e^{\gamma(t-T)}|(\mathcal T_Tu)_+(t)|
 &\le \frac{C\delta}{2+\gamma}\|u\|_\gamma+Ce^{-2T},\\
 \sup_{t\ge T}e^{\gamma(t-T)}|(\mathcal T_Tu)_-(t)|
 &\le |c_-(T)|+\frac{C\delta}{b-\gamma}\|u\|_\gamma+Ce^{-2T}.
\end{align*}
Here we used $\gamma<2<b$ and
\[
 e^{\gamma(t-T)}\int_t^\infty e^{-2(s-t)}e^{-\gamma(s-T)}\,ds
 =\frac1{2+\gamma},
\]
\[
 e^{\gamma(t-T)}\int_T^t e^{-b(t-s)}e^{-\gamma(s-T)}\,ds
 \le\frac1{b-\gamma}.
\]
The identical estimates for differences show that, after decreasing
$\delta$ once more if necessary, $\mathcal T_T$ is a strict contraction
on a closed $X_\gamma$--ball of radius
$C_\gamma(|c_-(T)|+e^{-2T})$ contained in $\overline B_\delta$.
Let $\widetilde c$ be its fixed point.  Both $c$ and $\widetilde c$ lie in
$\overline B_\delta\subset X_0$, where the fixed point is unique; hence
$c=\widetilde c$.  Therefore
\begin{equation}
    |c_+(t)|+|c_-(t)|
\le
C_\gamma e^{-\gamma(t-T)},
\qquad t\ge T.\label{eq:Y-coordinate-decay}
\end{equation}
Hence
\begin{equation}
|Y(t)|+|Y'(t)|
\le
C_\gamma e^{-\gamma t},
\qquad t\ge T,\,
0<\gamma<2.\label{eq:Y-zero-decay}
\end{equation}

\end{proof}

Recall that
\[
Y(t)=c_0 r^2p(r),
\qquad
t=\log r,
\]
where
\[
p(r)=\int_{-1}^1Q(r,z)\ph(z)\,dz.
\]
Since \(e^{-\gamma t}=r^{-\gamma}\), estimate \eqref{eq:Y-zero-decay} yields
\begin{equation}
    |p(r)|
=
\frac{|Y(\log r)|}{c_0r^2}
\le
C_\gamma r^{-2-\gamma}
\qquad (r\gg1).\label{eq:p-zero-decay}
\end{equation}

On the other hand, combining \eqref{eq:p-zero-decay} and the first estimate in \eqref{eq:w-minimal-tail}, we obtain
\begin{align}
|Q(r,z)|
\le
|p(r)|\varphi(z)+|w(r,z)| 
\le
C\bigl(r^{-2-\gamma}+r^{-4}\bigr)\varphi(z).
\end{align}

Choose any \(0<\gamma<2\) and set
\[
a_0:=2+\gamma\in(2,4).
\]
Hence, 
\[
|Q(r,z)|
\le
Cr^{-a_0}\varphi(z)
\qquad
(r\gg1,\ |z|<1).
\]

\begin{claim}
    $Q\in H^1_0(\Cyl)\cap L^3(\Cyl)$. 
\end{claim}
\begin{proof}
    
More generally, assume
\[
 |Q(r,z)|\le Cr^{-a}\ph(z),\qquad a>2.
\]
Coming back to \eqref{eq:p-radial-rig}, we have $g(r)\le Cr^{-2a}$,
while the unified massive inverse estimate in Step~3 gives directly
\eqref{eq:bootstrap-w-phi}, namely
$|w(r,z)|\le Cr^{-2a}\ph(z)$.  Since $p(r)\to0$, integration of
$-(r^{m-1}p')'=r^{m-1}g$ gives
\begin{equation}\label{eq:newton-bootstrap}
 p(r)=\frac1{m-2}\left[
 r^{2-m}\int_0^r s^{m-1}g(s)\dd s
 +\int_r^\infty s g(s)\dd s\right].
\end{equation}
If $2a<m$, then $p=O(r^{-(2a-2)})$; if $2a>m$, then
$p=O(r^{-(m-2)})$; at $2a=m$ one has
$p=O(r^{2-m}\log r)$ and hence
$p=O(r^{-(m-2-\delta)})$ for every $\delta>0$.  Consequently, for
arbitrarily small $\delta>0$,
\begin{equation}\label{eq:bootstrap-map}
 Q=O(r^{-a^+}\ph),\qquad
 a^+\ge\min\{2a-2,m-2\}-\delta.
\end{equation}
Starting from $a_0=2+d_0$, the exponents improve geometrically until,
after finitely many steps, one obtains $a_*>m/2$.  Standard unit-scale
interior and boundary Schauder estimates then give
$|\nabla_XQ|+|Q_z|\le Cr^{-a_*}$ on the far field.  Hence
$Q\in H^1_0(\Cyl)\cap L^3(\Cyl)$.  

\end{proof}

Testing the equation with radial
cutoffs times $Q$, and semistability with the same cutoffs, then letting
the cutoff radius tend to infinity, gives
\[
 \int_\Cyl|\nabla Q|^2=\La\int_\Cyl(Q^2+Q^3),
 \qquad
 \int_\Cyl|\nabla Q|^2\ge\La\int_\Cyl(Q^2+2Q^3),
\]
a contradiction.  Therefore
\begin{equation}\label{eq:Y-limit-rig}
 Y(t)\to2(m-4),\qquad p(r)=D_mr^{-2}+o(r^{-2}).
\end{equation}

\medskip\noindent
\emph{Step 5: the indicial rate.}
Let $b(t)=Y(t)-2(m-4)$.  Then
\[
 b''+(m-6)b'+2(m-4)b=O(b^2)+O(e^{-2t}).
\]
Set
\[
 \lambda_-:=\beta_--2,\qquad \lambda_+:=\beta_+-2.
\]
Then $0<\lambda_-<\lambda_+$ and
\[
 \lambda_-+\lambda_+=m-6,\qquad
 \lambda_-\lambda_+=2(m-4).
\]
Writing
\[
 R(t):=O(b(t)^2)+O(e^{-2t}),
\]
the equation becomes
\[
 (\partial_t+\lambda_-)(\partial_t+\lambda_+)b=R(t).
\]
Fix $0<\gamma<\min\{2,\lambda_-\}$.  Since $b(t)\to0$, choose
$T$ so large that $|b(t)|\le\delta$ for $t\ge T$, where $\delta>0$
will be fixed below.  The variation-of-constants formula on $[T,t]$ is
\begin{align}\label{eq:b-voc-rig}
 b(t)={}&A_Te^{-\lambda_-(t-T)}+B_Te^{-\lambda_+(t-T)}\notag\\
 &+\frac1{\lambda_+-\lambda_-}
 \int_T^t\bigl(e^{-\lambda_-(t-s)}-e^{-\lambda_+(t-s)}\bigr)R(s)\,ds,
\end{align}
where $A_T,B_T$ are determined linearly by $b(T),b'(T)$.  Differentiating
\eqref{eq:b-voc-rig} gives the analogous formula for $b'$.  For
$S>T$ put
\[
 M_{\gamma,S}:=
 \sup_{T\le t\le S}e^{\gamma(t-T)}\bigl(|b(t)|+|b'(t)|\bigr).
\]
Because $|R(t)|\le C\delta |b(t)|+Ce^{-2t}$ on $[T,\infty)$, the
kernel bounds
\[
 \int_T^t e^{-(\lambda_\pm-\gamma)(t-s)}\,ds
 \le\frac1{\lambda_\pm-\gamma}
\]
and $\gamma<2$ imply
\[
 M_{\gamma,S}
 \le C_\gamma\bigl(|b(T)|+|b'(T)|+e^{-2T}\bigr)
      +C_\gamma\delta M_{\gamma,S}.
\]
Choose $\delta$ so that $C_\gamma\delta\le\frac12$.  The last term is
absorbed, and the resulting bound is independent of $S$.  Letting
$S\to\infty$ gives
\[
 |b(t)|+|b'(t)|\le C_\gamma e^{-\gamma t}.
\]
Finally the differential equation for $b$ and
$R=O(b^2)+O(e^{-2t})$ give the same bound for $b''$.  Hence, for every
$0<\gamma<\min\{2,\beta_--2\}$,
\[
 |b(t)|+|b'(t)|+|b''(t)|\le C_\gamma e^{-\gamma t}.
\]
Equivalently, for every $2<\tau<\min\{4,\beta_-\}$,
\begin{equation}\label{eq:p-point-tail}
 p(r)=D_mr^{-2}+O(r^{-\tau}),\qquad
 p'(r)=-2D_mr^{-3}+O(r^{-\tau-1}).
\end{equation}
Let
\[
 a(r):=p(r)-D_mr^{-2}.
\]
Besides \eqref{eq:p-point-tail}, the first-mode equation and
\eqref{eq:w-sc-r4} give
\[
 a''(r)=O(r^{-\tau-2}).
\]
Indeed, if
\[
 e(r):=\La\int_{-1}^1Q^2\ph\dd z-c_0p^2,
\]
then $e=O(r^{-6})$ and, by \eqref{eq:w-natural-derivatives},
$e'=O(r^{-7})$.  Differentiating the equation for $a$ gives
\[
 a'''(r)=O(r^{-\tau-3}),
\]
because $\tau<4$.  The mean-value theorem on dyadic annuli now yields
\[
 a(r)\ph=O_{C^{2,\alpha}_{\rm sc}}(r^{-\tau}).
\]
Since \eqref{eq:w-sc-r4} and $\tau<4$ imply
$w=O_{C^{2,\alpha}_{\rm sc}}(r^{-\tau})$, we conclude that
\[
 Q-D_mr^{-2}\ph
 =O_{C^{2,\alpha}_{\rm sc}}(r^{-\tau}),
\]
which is \eqref{eq:tailQ-sc}.  The pointwise quotient estimate
\eqref{eq:tailQ} follows from the boundary factorization, while
\eqref{eq:p-point-tail} and the second estimate in
\eqref{eq:w-minimal-tail} give \eqref{eq:tailQr}.
\end{proof}

\begin{corollary}[Radial monotonicity]\label{cor:radial-mono}
Under the assumptions of \cref{lem:universal-tail},
\begin{equation}\label{eq:radial-mono}
 Q_r(r,z)<0\qquad(r>0,\ |z|<1).
\end{equation}
\end{corollary}

\begin{proof}
Set $G=-Q_r$.  Differentiating \eqref{eq:cylinder} for $r>0$ and using
\[
 \partial_r\!\left(\frac{m-1}{r}Q_r\right)
 =\frac{m-1}{r}Q_{rr}-\frac{m-1}{r^2}Q_r
\]
gives the exact equation
\begin{equation}\label{eq:Veq-rig}
 L_QG=-\frac{m-1}{r^2}G,
 \qquad L_Q:=-\Delta_X-\partial_{zz}-\La(1+2Q).
\end{equation}
By \eqref{eq:tailQr}, after division by $\ph$,
\[
 G(r,z)=2D_mr^{-3}\ph(z)+o(r^{-3}\ph(z)),
\]
uniformly for $|z|\le1$; hence $G>0$ for $r\ge R_1$ and $|z|<1$.
At the axis radial regularity gives $G(r,z)=O(r)$, while differentiating
$Q(r,\pm1)=0$ in $r$ gives $G(r,\pm1)=0$.  Therefore
$G_-:=\max\{-G,0\}$ is supported in $\{r<R_1\}$, has zero trace on
$z=\pm1$, and satisfies $r^{-1}G_-\in L^2$ near $r=0$.  Thus
$G_-\in H^1_0(\Cyl)$ and may be obtained as an $H^1$ limit of smooth
compactly supported admissible test functions.

Multiply \eqref{eq:Veq-rig} by $G_-$ and integrate.  On $\{G<0\}$,
$G=-G_-$ and $\nabla G=-\nabla G_-$.  Hence
\begin{align*}
 -\mathfrak q_Q[G_-]=\int_\Cyl\bigl(\nabla G\cdot\nabla G_-
   -\La(1+2Q)GG_-\bigr)
 =
 -(m-1)\int_\Cyl\frac{GG_-}{r^2}
 = (m-1)\int_\Cyl\frac{G_-^2}{r^2}.
\end{align*}
Consequently
\begin{equation}\label{eq:Vnegative-rig}
 -\mathfrak q_Q[G_-]
 =(m-1)\int_\Cyl\frac{G_-^2}{r^2}.
\end{equation}
Semistability gives $\mathfrak q_Q[G_-]\ge0$, so the left side of
\eqref{eq:Vnegative-rig} is nonpositive whereas the right side is
nonnegative.  Both vanish, and therefore $G_-=0$.  Hence $G\ge0$.

Note that equation \eqref{eq:Veq-rig} can be written as
\[
 -\Delta G+c(r,z)G=0\quad\text{in }\{r>0,|z|<1\},
 \qquad
 c(r,z)=\frac{m-1}{r^2}-\La(1+2Q),
\]
with locally bounded coefficient $c$.  The local Harnack inequality for
nonnegative solutions of Schr\"odinger equations with bounded zeroth-order
coefficient therefore applies on every ball compactly contained in
$\{r>0,|z|<1\}$.  If $G$ vanished at an interior point, Harnack's
inequality would imply that $G$ vanishes on a neighborhood of that point;
propagating this conclusion through overlapping interior balls gives
$G\equiv0$ on the connected set $\{r>0,|z|<1\}$.  This contradicts the
positive tail.  Hence $G>0$, i.e. $Q_r<0$ for every $r>0$.
\end{proof}

\subsection{The normal operator and fast Jacobi fields}

\begin{definition}[Radial-even Dirichlet Jacobi field at $Q$]\label{def:jacobi-field}
Let $Q$ be a smooth threshold-cylinder profile.  A
\emph{radial-even Dirichlet Jacobi field at $Q$} is a nonzero function
$H$ satisfying
\begin{equation}\label{eq:jacobi}
 L_QH=0,
 \qquad H(r,\pm1)=0,
\end{equation}
which is radial in $X$, even in $z$, smooth with the usual radial-axis
compatibility at $r=0$, and polynomially decaying at infinity, i.e. for
some $\rho>0$,
\[
 H=O_{C^{2,\alpha}_{\rm sc}}(r^{-\rho}).
\]
It is the Jacobi field associated with the linearization of
\eqref{eq:cylinder} at $Q$ within the radial-even Dirichlet symmetry
class.
\end{definition}
Write
\[
 H=h(r)\ph+H^\perp,
 \qquad \int_{-1}^1H^\perp(r,z)\ph(z)\,dz=0,
\]
Recall from \eqref{eq:w-eq-rig} that
\[
 \mathscr L_\perp
 =-\partial_{rr}-\frac{m-1}{r}\partial_r+H_z,
 \qquad H_z=-\partial_{zz}-\La,
\]
on $\ph^\perp$.  By \cref{lem:universal-tail}, with $\tau=\sigma$,
\begin{equation}\label{eq:Q-tail-normal}
 Q=D_mr^{-2}\ph+R_Q,
 \qquad R_Q=O_{C^{2,\alpha}_{\rm sc}}(r^{-\sigma}),
 \qquad \sigma>2.
\end{equation}
Projecting $L_QH=0$ onto $\ph^\perp$ gives the exact equation
\begin{equation}\label{eq:jacobi-perp-projected}
 \mathscr L_\perp H^\perp
 =2\La P^\perp\!\left(Q(h\ph+H^\perp)\right).
\end{equation}
We record the quantitative transverse estimate used below.
\begin{claim}
Fix $\ell\ge0$ such that
\[
 \|H\|_{\mathcal J_\ell}:=
 \|H\|_{C^{2,\alpha}(\{r<2R_1\})}
 +\sup_{R\ge R_1}R^{-\ell}
   \|H\|_{C^{2,\alpha}_{\rm sc}(A_R)}<\infty.
\]
There exist $R_1,c,C>0$, depending only on the fixed profile $Q$ and
on $m,\alpha,\ell$, such that, after setting
\[
 A_R=\{(X,z):R<|X|<2R,\ |z|<1\},
\]
and
\[
 M_h(R):=\sup_{R/8\le r\le16R}\bigl(|h(r)|+r|h'(r)|\bigr),
\]
one has
\begin{equation}\label{eq:jacobi-perp-annular}
  \|H^\perp\|_{C^{2,\alpha}_{\rm sc}(A_R)}
 \le CR^{-2}M_h(R)+C\|H\|_{\mathcal J_\ell} e^{-cR},\qquad R\ge R_1.
\end{equation}
\end{claim}
\begin{proof}
Write $v:=H^\perp$.  We separate the proof into an exterior
unit-strip inverse and a finite derivative bootstrap.

\smallskip\noindent
\emph{1. Exponentially weighted inverses and exterior localization.}
For $j=0,1,2,3$ put
\[
 T_j=\left(\mathscr L_\perp+\frac{j(m-1)}{r^2}\right)^{-1},
 \qquad
 B_Q=2\La\chi_{R_1}P^\perp(Q\,\cdot),
 \qquad
 \mathscr A_j^Q=T_j^{-1}-B_Q,
\]
where $\chi_{R_1}=0$ for $r\le R_1/2$ and $\chi_{R_1}=1$ for
$r\ge R_1$.  The unperturbed inverse is the regular (Friedrichs at
the axis), polynomially bounded realization.  To avoid confusing
polynomial growth with the positive decay weights $X_q,Y_q$ above,
for $p\in\R$ define
\[
 \|F\|_{\mathcal B_p}
   =\sup_{r\ge0}(1+r)^p\|F(r,\cdot)\|_{L^2(-1,1)}.
\]
A bound in $\mathcal B_{-\ell}$ allows growth of order $\ell$.
If $\theta$ is a bounded real $1$-Lipschitz function and $a\ge0$, also set
$\|F\|_{\mathcal B_{p,a,\theta}}
=\|e^{a\theta}F\|_{\mathcal B_p}$.

Let $G_*=G_{8\La,0}$ be the positive radial Green kernel in the proof
of the massive inverse.  For every fixed $p\in\R$ and
$0\le a<\sqrt{8\La}$, its Bessel representation gives
\begin{equation}\label{eq:jacobi-exponential-moment}
 \sup_{r\ge0}(1+r)^p
 \int_0^\infty G_*(r,s)(1+s)^{-p}
                 e^{a|r-s|}s^{m-1}\,ds\le C_{p,a}.
\end{equation}
Here is a direct way to check the weight.  For $r,s\ge1$ the Bessel
bounds give exponential decay at rate $\sqrt{8\La}$ times fixed powers
of $r/s$ and $s/r$.  Split the integral into $s<r/2$,
$r/2\le s\le2r$, and $s>2r$; on the middle part the ratios are
bounded, and on the other parts their fixed powers are absorbed by
$e^{-(\sqrt{8\La}-a)|r-s|}$.  When both radii are bounded the weights
are bounded and the integral is bounded by the resolvent of the
constant source.  The mixed compact--large regions have the same
exponential bound.  This proves \eqref{eq:jacobi-exponential-moment}
for positive and negative $p$.

The scalar kernels for all transverse modes and all $j$ are dominated
by $G_*$, since both the mass and $j(m-1)r^{-2}$ are nonnegative
increments of the reference operator.  Minkowski and Parseval,
followed by $|\theta(r)-\theta(s)|\le|r-s|$, now imply
\begin{equation}\label{eq:jacobi-exponential-conjugation}
 \|T_jF\|_{\mathcal B_{p,a,\theta}}
 \le C_{p,a}\|F\|_{\mathcal B_{p,a,\theta}}.
\end{equation}
These constants are independent of $\theta$, of its supremum norm,
and of $j\in\{0,1,2,3\}$.  Moreover
\[
 \|B_Qv\|_{\mathcal B_{p,a,\theta}}
 \le C R_1^{-2}\|v\|_{\mathcal B_{p,a,\theta}},
\]
because multiplication by $Q$ is bounded on $L^2_z$ and $P^\perp$
is an orthogonal projection.  Fix a small $a>0$.  Choose $R_1$ so
large that $C C_{p,a}R_1^{-2}<1/2$ for $p=0,-\ell$ and also for
$a=0$.  Thus the Neumann series
\begin{equation}\label{eq:jacobi-Aj-neumann}
 (\mathscr A_j^Q)^{-1}
 =\sum_{n=0}^\infty(T_jB_Q)^nT_j
\end{equation}
converges in each of these exponentially weighted spaces, uniformly
in $\theta$.  It is this weighted convergence, rather than an
unweighted operator norm, that supplies off-diagonal decay.  The
solutions obtained in the different weights agree by uniqueness in
the polynomially bounded class.

For an exterior source, use the growth norm
\[
 \|F\|_{\mathcal S_\ell}
 =\|F\|_{\mathcal B_{-\ell}}
 +\sup_{r\ge R_1/4}(1+r)^{-\ell}
                       \|F\|_{C^{0,\alpha}(\mathcal U_r)}.
\]
Let $R\ge16R_1$, $R/4\le\rho\le8R$, and choose a smooth cutoff
$\eta_R$ supported in $R/8<r<16R$, equal to one on
$R/6\le r\le12R$.  Split $F=\eta_RF+(1-\eta_R)F$.
The first part is bounded in $\mathcal B_0$ by its local supremum.
For the second part choose
\[
 \theta(r)=\min\{\operatorname{dist}(r,
                \operatorname{supp}((1-\eta_R)F)),c_1R\},
\]
where $c_1>0$ is small and fixed.  It vanishes on the source support
and is at least $c_2R$ throughout the observation strip, for fixed
$c_2>0$.  The weighted inverse therefore bounds the far contribution
there by
$C(1+R)^\ell e^{-ac_2R}\|F\|_{\mathcal S_\ell}$.
Absorb the polynomial into a slightly smaller exponential.
On translated unit strips, the equation is uniformly elliptic in
$(r,z)$; the nonlocal part is the bounded rank-one term
$\ph\int Qv\ph$.  Its $L^2$ bound first gives a $W^{2,2}$ estimate,
then $C^{0,\gamma}$ for every $\gamma<1$, and finally the boundary
Schauder estimate.  For the far contribution the local source is
zero.  For the near contribution its local $C^{0,\alpha}$ norm is
retained.  We obtain
\begin{equation}\label{eq:jacobi-unit-localized-inverse}
 \begin{split}
 \| (\mathscr A_j^Q)^{-1}F
       \|_{C^{2,\alpha}(\mathcal U_{\rho,1/2})}
 \le{}& C\sup_{R/8\le s\le16R}
                \|F\|_{C^{0,\alpha}(\mathcal U_s)}
       +Ce^{-cR}\|F\|_{\mathcal S_\ell}.
 \end{split}
\end{equation}
The proof applies to any two fixed nested annular intervals with
positive separation in the $r/R$ variable.  We use successively
smaller intervals below.  The bounded range $R_1\le R\le16R_1$
is covered directly by $\|H\|_{\mathcal J_\ell}$, after increasing
the constant in the exponential term.
Before applying the inverse to a derivative of $v=H^\perp$, multiply
that derivative by a fixed cutoff vanishing near the axis.  The
resulting commutator is supported in a fixed annulus and is controlled
by $C\|H\|_{\mathcal J_\ell}$, using local differentiated Schauder
estimates for the homogeneous Jacobi equation.  It therefore
contributes only the exponential term in
\eqref{eq:jacobi-unit-localized-inverse}.  This also removes any issue
with the axis realization of the differentiated radial operators.

\smallskip\noindent
\emph{2. Zeroth transverse estimate and the first-mode derivatives.}
For $r\ge R_1$, the transverse Jacobi equation is
\[
 \left(\mathscr L_\perp-2\La P^\perp(Q\,\cdot)\right)v
 =2\La P^\perp(Qh\ph).
\]
Replacing the exterior operator by $\mathscr A_0^Q$ produces only a
fixed compactly supported error; by
\eqref{eq:jacobi-unit-localized-inverse} this contributes
$O(C\|H\|_{\mathcal J_\ell}e^{-cR})$ on the annuli under consideration.  Since
$Q=O_{C^{2,\alpha}_{\rm sc}}(r^{-2})$, the mean-value theorem gives
\[
 \sup_{R/8\le s\le16R}
 \|P^\perp(Qh\ph)\|_{C^{0,\alpha}(\mathcal U_s)}
 \le CR^{-2}M_h(R).
\]
Hence \eqref{eq:jacobi-unit-localized-inverse} with $j=0$ yields
\begin{equation}\label{eq:jacobi-v-unit0}
 \sup_{R/4\le\rho\le8R}
 \|v\|_{C^{2,\alpha}(\mathcal U_{\rho,1/2})}
 \le CR^{-2}M_h(R)+C\|H\|_{\mathcal J_\ell}e^{-cR}.
\end{equation}

Projecting $L_QH=0$ onto $\ph$ gives
\begin{equation}\label{eq:jacobi-h-euler-local}
 -h''-\frac{m-1}{r}h'-\frac{4(m-4)}{r^2}h=F_H,
\end{equation}
where
\[
 F_H=2\La D_mr^{-2}\int_{-1}^1\ph^2v\,\dd z
 +2\La\int_{-1}^1R_Q(h\ph+v)\ph\,\dd z.
\]
Using \eqref{eq:jacobi-v-unit0},
$R_Q=O_{C^{2,\alpha}_{\rm sc}}(r^{-\sigma})$ with $\sigma>2$, and
the definition of $M_h(R)$, we obtain
\begin{equation}\label{eq:jacobi-h2-simple}
 \sup_{R/4\le r\le8R}|h''(r)|
 \le CR^{-2}M_h(R)+C\|H\|_{\mathcal J_\ell}e^{-cR}.
\end{equation}

\smallskip\noindent
\emph{3. First and second radial derivatives of the transverse field.}
Differentiating the exterior transverse equation once gives
\begin{equation}\label{eq:jacobi-vr-eq-simple}
 \left(\mathscr L_\perp+\frac{m-1}{r^2}
       -2\La P^\perp(Q\,\cdot)\right)v_r
 =2\La P^\perp\!\left(Q_r(h\ph+v)+Qh'\ph\right).
\end{equation}
After insertion of the cutoff defining $\mathscr A_1^Q$, the discrepancy
is compactly supported and hence exponentially small on $A_R$.  From
\eqref{eq:jacobi-v-unit0}, the universal-tail derivative bounds
\[
 Q_r=O_{C^{0,\alpha}_z}(r^{-3}),\qquad Q=O(r^{-2}),
\]
and $|h'|\le CR^{-1}M_h(R)$, the right-hand side of
\eqref{eq:jacobi-vr-eq-simple} is bounded in the local
$C^{0,\alpha}$ norm by $CR^{-3}M_h(R)+Ce^{-cR}$.  Therefore
\eqref{eq:jacobi-unit-localized-inverse} with $j=1$ gives
\begin{equation}\label{eq:jacobi-vr-unit}
 \sup_{R/2\le\rho\le4R}
 \|v_r\|_{C^{2,\alpha}(\mathcal U_{\rho,1/2})}
 \le CR^{-3}M_h(R)+C\|H\|_{\mathcal J_\ell}e^{-cR}.
\end{equation}

Differentiating $F_H$ and using
\eqref{eq:jacobi-v-unit0}--\eqref{eq:jacobi-vr-unit}, together with
$R_Q=O_{C^{2,\alpha}_{\rm sc}}(r^{-\sigma})$, gives
\[
 \sup_{R/2\le r\le4R}|F_H'(r)|
 \le CR^{-3}M_h(R)+C\|H\|_{\mathcal J_\ell}e^{-cR}.
\]
Differentiating \eqref{eq:jacobi-h-euler-local} therefore yields
\begin{equation}\label{eq:jacobi-h3-simple}
 \sup_{R/2\le r\le4R}|h'''(r)|
 \le CR^{-3}M_h(R)+C\|H\|_{\mathcal J_\ell}e^{-cR}.
\end{equation}

Differentiating the transverse equation twice gives
\begin{align}\label{eq:jacobi-vrr-eq-simple}
 &\left(\mathscr L_\perp+\frac{2(m-1)}{r^2}
       -2\La P^\perp(Q\,\cdot)\right)v_{rr}\notag\\
 &\quad=2\La P^\perp\!\left(
 Q_{rr}(h\ph+v)+2Q_r(h'\ph+v_r)+Qh''\ph\right)
 +\frac{2(m-1)}{r^3}v_r.
\end{align}
All terms on the right are
$O_{C^{0,\alpha}}(R^{-4}M_h(R))+O(e^{-cR})$ on the relevant unit
strips.  Hence the case $j=2$ of
\eqref{eq:jacobi-unit-localized-inverse} yields
\begin{equation}\label{eq:jacobi-vrr-unit}
 \sup_{3R/4\le\rho\le3R}
 \|v_{rr}\|_{C^{2,\alpha}(\mathcal U_{\rho,1/2})}
 \le CR^{-4}M_h(R)+C\|H\|_{\mathcal J_\ell}e^{-cR}.
\end{equation}

\smallskip\noindent
\emph{4. One more radial derivative and recovery of the scaled norm.}
The estimates already obtained also justify the H\"older norm of the
next source.  Differentiating the first-mode source twice, using the
bounds for $v,v_r,v_{rr}$, gives
$|F_H''|\le CR^{-4}M_h(R)+C\|H\|_{\mathcal J_\ell}e^{-cR}$ on
$3R/4\le r\le3R$.  The twice differentiated first-mode equation
therefore yields
\[
 |h^{(4)}|\le CR^{-4}M_h(R)+C\|H\|_{\mathcal J_\ell}e^{-cR}
 \quad(3R/4\le r\le3R).
\]
In particular $h^{(3)}$ has the required unit-strip $C^{0,\alpha}$ bound.
Differentiating \eqref{eq:jacobi-vrr-eq-simple} once more gives
\begin{align}\label{eq:jacobi-vrrr-eq-simple}
 &\left(\mathscr L_\perp+\frac{3(m-1)}{r^2}
       -2\La P^\perp(Q\,\cdot)\right)v_{rrr}\notag\\
 &\quad=2\La P^\perp\!\left(
 Q_{rrr}(h\ph+v)+3Q_{rr}(h'\ph+v_r)
 +3Q_r(h''\ph+v_{rr})+Qh'''\ph\right)\notag\\
 &\qquad+\frac{6(m-1)}{r^3}v_{rr}
          -\frac{6(m-1)}{r^4}v_r.
\end{align}
The universal-tail estimates proved in Step~3 give
\[
 Q_r=O(r^{-3}),\qquad Q_{rr}=O(r^{-4}).
\]
For the third derivative, differentiate the radial equation
\eqref{eq:p-radial-rig} once more.  Since
\[
 g''(r)=2\La\int_{-1}^1
 \bigl(Q_r^2+QQ_{rr}\bigr)\ph\,\dd z=O(r^{-6}),
\]
one obtains $p''''=O(r^{-6})$.  Together with
\eqref{eq:wrrr-rig}, this yields
\[
 Q_{rrr}=O_{C^{0,\alpha}}(r^{-5})
\]
on translated unit strips.  Together with the preceding estimates, the
right-hand side of \eqref{eq:jacobi-vrrr-eq-simple} is
$O_{C^{0,\alpha}}(R^{-5}M_h(R))+O(e^{-cR})$.  Applying
\eqref{eq:jacobi-unit-localized-inverse} with $j=3$ gives
\begin{equation}\label{eq:jacobi-vrrr-unit}
 \sup_{R\le\rho\le2R}
 \|v_{rrr}\|_{C^{2,\alpha}(\mathcal U_{\rho,1/2})}
 \le CR^{-5}M_h(R)+C\|H\|_{\mathcal J_\ell}e^{-cR}.
\end{equation}

The estimates
\eqref{eq:jacobi-v-unit0}, \eqref{eq:jacobi-vr-unit},
\eqref{eq:jacobi-vrr-unit}, and \eqref{eq:jacobi-vrrr-unit} imply, on
$A_R$,
\[
 |v|+R|v_r|+R^2|v_{rr}|
 \le CR^{-2}M_h(R)+C\|H\|_{\mathcal J_\ell}e^{-cR}.
\]
Since $v$ is radial in $X$, the identities
\[
 \partial_{X_i}v=v_r\frac{X_i}{r},\qquad
 \partial_{X_iX_j}v
 =v_{rr}\frac{X_iX_j}{r^2}
 +\frac{v_r}{r}\left(\delta_{ij}-\frac{X_iX_j}{r^2}\right)
\]
control the scaled first and second $X$ derivatives.  The
$C^{2,\alpha}$ unit-strip estimates control all $z$ derivatives and
mixed derivatives, while \eqref{eq:jacobi-vrrr-unit} and the mean-value
theorem control the radial Hölder seminorm of the second $X$ derivatives
on distances up to order $R$.  Consequently
\[
 \|v\|_{C^{2,\alpha}_{\rm sc}(A_R)}
 \le CR^{-2}M_h(R)+C\|H\|_{\mathcal J_\ell}e^{-cR},
\]
which is \eqref{eq:jacobi-perp-annular}.

In the abbreviated asymptotic notation used below this is
\begin{equation}\label{eq:jacobi-perp}
 H^\perp
 =O_{C^{2,\alpha}_{\rm sc}}
   \bigl(r^{-2}(|h|+r|h'|)\bigr)+O(e^{-cr}).
\end{equation}
\end{proof}

The first-mode equation is governed by the Euler operator
\begin{equation}\label{eq:euler}
 Eh:=-h''-\frac{m-1}{r}h'-\frac{4(m-4)}{r^2}h,
\end{equation}
whose indicial solutions are $r^{-\beta_-}$ and $r^{-\beta_+}$, with
\[
 W(r^{-\beta_-},r^{-\beta_+})
 =-(\beta_+-\beta_-)r^{-(m-1)}.
\]
The next lemma records the complete weight bootstrap and the slow/fast
alternative.  It will be used later for the branch tangent without repeating
this argument.

\begin{lemma}[Jacobi indicial bootstrap and slow/fast alternative]
\label{lem:jacobi-slow-fast}
Let $Q$ be a smooth semistable profile satisfying
\eqref{eq:Q-tail-normal}, and let
\[
 H=h(r)\ph+H^\perp
\]
be a radial-even Dirichlet Jacobi field which is regular at the axis and
polynomially decaying at infinity.  Then there exists a unique slow
coefficient $c_-(H)$ given by
\begin{equation}\label{eq:slow-coefficient-limit}
 c_-(H)
 =\lim_{r\to\infty}
 \frac{r^{\beta_-}}{\beta_+-\beta_-}
 \bigl(rh'(r)+\beta_+h(r)\bigr).
\end{equation}
Exactly one of the following alternatives occurs:
\begin{equation}\label{eq:jacobi-slow-general}
 H=c_-(H)r^{-\beta_-}\ph
   +o_{C^{2,\alpha}_{\rm sc}}(r^{-\beta_-})
 \qquad\text{if }c_-(H)\ne0,
\end{equation}
or
\begin{equation}\label{eq:jacobi-fast-general}
 H=c_+(H)r^{-\beta_+}\ph
   +o_{C^{2,\alpha}_{\rm sc}}(r^{-\beta_+})
 \qquad\text{if }c_-(H)=0,
\end{equation}
with faster decay if also $c_+(H)=0$.  The same expansions hold after
division by $\ph$, uniformly up to $z=\pm1$.
\end{lemma}

\begin{definition}[Slow and fast radial-even Jacobi fields]\label{def:slow-fast-jacobi}
For a radial-even Dirichlet Jacobi field covered by
\cref{lem:jacobi-slow-fast}, the number $c_-(H)$ in
\eqref{eq:slow-coefficient-limit} is its \emph{slow coefficient}.  We call
$H$ \emph{slow} if $c_-(H)\ne0$ and \emph{fast} if $c_-(H)=0$.  Thus a slow
field has leading first-mode decay $r^{-\beta_-}\ph$, whereas a fast field
has decay at least $r^{-\beta_+}\ph$ in the first mode, with the transverse
complement decaying still faster as specified by
\cref{lem:jacobi-slow-fast}.
\end{definition}

\begin{proof}
Projecting $L_QH=0$ onto the first transverse mode and using
\eqref{eq:Q-tail-normal} gives
\begin{equation}\label{eq:jacobi-first-general}
 Eh=F_H,
\end{equation}
where
\begin{align}\label{eq:jacobi-first-source-general}
 F_H(r)
 :={}&2\La D_mr^{-2}\int_{-1}^1\ph^2H^\perp\,\dd z \\
 &+2\La\int_{-1}^1R_Q(h\ph+H^\perp)\ph\,\dd z.\notag
\end{align}
Here we used $2\La D_m\int_{-1}^1\ph^3=4(m-4)$.  The transverse
component is controlled by the annular estimate
\eqref{eq:jacobi-perp-annular}.

\medskip\noindent
\emph{Step 1: bootstrap up to the slow indicial root.}
By polynomial decay there exists $\rho_0>0$ such that
\begin{equation}\label{eq:jacobi-initial-rho}
 H=O_{C^{2,\alpha}_{\rm sc}}(r^{-\rho_0}).
\end{equation}
Suppose, more generally, that for some $\rho>0$,
\begin{equation}\label{eq:jacobi-rho-hyp}
 H=O_{C^{2,\alpha}_{\rm sc}}(r^{-\rho}).
\end{equation}
Then $M_h(R)=O(R^{-\rho})$, and
\eqref{eq:jacobi-perp-annular} yields
\begin{equation}\label{eq:jacobi-perp-rho-gain}
 H^\perp=O_{C^{2,\alpha}_{\rm sc}}(r^{-\rho-2}).
\end{equation}
Substituting this and $R_Q=O_{C^{2,\alpha}_{\rm sc}}(r^{-\sigma})$
into \eqref{eq:jacobi-first-source-general} gives
\begin{equation}\label{eq:jacobi-source-rho-gain}
 F_H=O_{C^{0,\alpha}_{\rm sc}}(r^{-\rho-\sigma}),
\end{equation}
because the first term in \eqref{eq:jacobi-first-source-general} is
$O(r^{-\rho-4})$ and $2<\sigma<4$.

We now make the Euler improvement explicit.  If $Eh=f$ on $[R,\infty)$,
variation of constants with both homogeneous coefficients retained gives
\begin{align}\label{eq:euler-preindicial-voc}
 h(r)={}&A r^{-\beta_-}+B r^{-\beta_+}
 -\frac{r^{-\beta_-}}{\beta_+-\beta_-}
   \int_R^r s^{\beta_-+1}f(s)\,\dd s \\
 &+\frac{r^{-\beta_+}}{\beta_+-\beta_-}
   \int_R^r s^{\beta_++1}f(s)\,\dd s.\notag
\end{align}
Consequently, if $f=O(r^{-q})$, then for every
\begin{equation}\label{eq:euler-preindicial-improve}
 0<\rho'<\min\{\beta_-,q-2\}
\end{equation}
one has $h=O_{C^{2,\alpha}_{\rm sc}}(r^{-\rho'})$; at the borderline $q-2=\beta_-$ one has
$r^{-\beta_-}\log r=O(r^{-\rho'})$ for every $\rho'<\beta_-$, so the
same conclusion holds with any such $\rho'$.  Applying this with $q=\rho+\sigma$ and combining with
\eqref{eq:jacobi-perp-rho-gain}, we obtain
\begin{equation}\label{eq:jacobi-weight-step}
 H=O_{C^{2,\alpha}_{\rm sc}}(r^{-\rho'})
 \quad\text{for every}\quad
 \rho'<\min\{\beta_-,\rho+\sigma-2\}.
\end{equation}
Thus, as long as the current weight is below $\beta_-$, it increases by
an arbitrarily small amount less than the fixed positive number
$\sigma-2$.  Starting from \eqref{eq:jacobi-initial-rho}, a finite number
of applications of \eqref{eq:jacobi-weight-step} reaches every prescribed
weight $\rho<\beta_-$. 
This is the bootstrap up to the slow indicial root.

Choose now $\rho<\beta_-$ so close to $\beta_-$ that
\begin{equation}\label{eq:jacobi-mu-above-slow}
 \beta_-<\mu<\min\{\beta_+,\rho+\sigma-2\}
\end{equation}
for some $\mu$.  Then \eqref{eq:jacobi-source-rho-gain} implies
$F_H\in C^{0,\alpha}_{\mu+2}$.  Rewriting
\eqref{eq:euler-preindicial-voc} by replacing the first integral with a
tail integral and absorbing the resulting constant into $A$ gives
\begin{align}\label{eq:euler-between-roots-general}
 h(r)=c_-r^{-\beta_-}+c_+r^{-\beta_+}
 +\frac1{\beta_+-\beta_-}\Bigg[
 &r^{-\beta_-}\int_r^\infty s^{\beta_-+1}F_H(s)\,\dd s\\
 &+r^{-\beta_+}\int_R^r s^{\beta_++1}F_H(s)\,\dd s
 \Bigg].\notag
\end{align}
The bracketed term is
$O_{C^{2,\alpha}_{\rm sc}}(r^{-\mu})$.  Therefore $c_-$ is uniquely
determined.  
With $t=\log r$ and
$y(t)=h(e^t)$, equation \eqref{eq:jacobi-first-general} becomes
\[
 y''+(m-2)y'+4(m-4)y=q(t),
 \qquad q(t):=-e^{2t}F_H(e^t),
\]
and
\[
 a_-(t):=
 \frac{e^{\beta_-t}}{\beta_+-\beta_-}
 \bigl(y'(t)+\beta_+y(t)\bigr)
\]
satisfies
\[
 a_-'(t)=\frac{e^{\beta_-t}}{\beta_+-\beta_-}q(t).
\]
Because \eqref{eq:jacobi-mu-above-slow} gives
$e^{\beta_-t}q(t)\in L^1(T,\infty)$, $a_-(t)$ converges.  Returning to
$r=e^t$ gives precisely \eqref{eq:slow-coefficient-limit}.

If $c_-(H)\ne0$, formula \eqref{eq:euler-between-roots-general} and
\eqref{eq:jacobi-perp-rho-gain} yield
\eqref{eq:jacobi-slow-general}.

\medskip\noindent
\emph{Step 2: bootstrap from a vanishing slow coefficient to the fast root.}
Assume $c_-(H)=0$.  Formula \eqref{eq:euler-between-roots-general} first
gives $H=O_{C^{2,\alpha}_{\rm sc}}(r^{-\mu})$ for the weight $\mu$ in
\eqref{eq:jacobi-mu-above-slow}.  Repeating
\eqref{eq:jacobi-perp-rho-gain}--\eqref{eq:jacobi-source-rho-gain} and
using \eqref{eq:euler-between-roots-general} with the already determined
slow coefficient equal to zero improves the exponent according to
\[
 \mu<\mu'<\min\{\beta_+,\mu+\sigma-2\}.
\]
Since $\sigma-2>0$, finitely many iterations place the weight arbitrarily
close to $\beta_+$ from below.  Choose such a weight, still denoted by
$\mu$, so that
\begin{equation}\label{eq:jacobi-fast-integrability-general}
 \mu+\sigma>\beta_++2,
 \qquad
 \mu+2>\beta_+.
\end{equation}
Then $F_H=O(r^{-q})$ with $q:=\mu+\sigma>\beta_++2$, hence
\[
 I_+:=\int_R^\infty s^{\beta_++1}F_H(s)\,\dd s
\]
converges absolutely.  In \eqref{eq:euler-between-roots-general} write
\[
 \int_R^r s^{\beta_++1}F_H(s)\,\dd s
 =I_+-\int_r^\infty s^{\beta_++1}F_H(s)\,\dd s.
\]
Absorbing $I_+/(\beta_+-\beta_-)$ into the coefficient of
$r^{-\beta_+}$ and estimating both remaining tail integrals by
$O(r^{2-q})=o(r^{-\beta_+})$ gives
\[
 h(r)=c_+(H)r^{-\beta_+}+o(r^{-\beta_+}).
\]
Moreover \eqref{eq:jacobi-perp-rho-gain} and the second inequality in
\eqref{eq:jacobi-fast-integrability-general} imply
$H^\perp=o_{C^{2,\alpha}_{\rm sc}}(r^{-\beta_+})$, proving
\eqref{eq:jacobi-fast-general}.  If $c_+(H)=0$, the same formula gives
strictly faster decay.

Finally, all remainders have zero Dirichlet trace at $z=\pm1$.  The
boundary Schauder estimate and the factorization
$\ph(z)\asymp1-|z|$ therefore give the same slow and fast expansions
after division by $\ph$, uniformly up to the transverse boundary.
\end{proof}

For continuation we use a graph realization of the massive transverse
block.  The scaled H\"older spaces remain useful for profiles and
homogeneous Jacobi fields; the domain of the operator on arbitrary
sources is specified separately.

Fix
\begin{equation}\label{eq:mu-choice}
 \beta_-<\mu<\min\{\beta_+,\beta_-+\sigma-2\},
 \qquad \psi_-=\chi_\infty r^{-\beta_-}\ph.
\end{equation}
Let $\mathcal Y_\mu^\perp=P^\perp C^{0,\alpha}_\mu(\Cyl)$, with the
radial and even symmetries.
\begin{definition}[Massive transverse graph realization]\label{def:massive-graph-domain}
The \emph{massive transverse graph domain} is
\begin{equation}\label{eq:massive-graph-domain}
 \begin{split}
 \mathcal G_\mu^\perp=\{w:
 &\ w\in C^{2,\alpha}_{\mathrm{loc}}(\overline\Cyl),\quad
 w(r,\pm1)=0,\quad P_1w=0,\\
 &\ w\in C^{0,\alpha}_\mu(\Cyl),\quad
 \mathscr L_\perp w\in\mathcal Y_\mu^\perp\},\\
 \|w\|_{\mathcal G_\mu^\perp}
 &=\|w\|_{C^{0,\alpha}_\mu}
       +\|\mathscr L_\perp w\|_{C^{0,\alpha}_\mu}.
 \end{split}
\end{equation}
All functions in this space have the radial-axis compatibility and the
symmetries imposed above.  The operator $\mathscr L_\perp$ is
$-\Delta_X+H_z$ on $\ph^\perp$.  The adjective \emph{graph} refers to the
fact that both $w$ and $\mathscr L_\perp w$ are controlled in the same
weighted topology; no unproved gain of scaled $X$ derivatives for
arbitrary sources is built into the definition.
\end{definition}

\begin{lemma}[The massive graph realization]\label{lem:massive-graph}
The space $\mathcal G_\mu^\perp$ is Banach, and
\[
 \mathscr L_\perp:\mathcal G_\mu^\perp\longrightarrow
                         \mathcal Y_\mu^\perp
\]
is an isomorphism.  On every fixed compact cylinder, its graph norm
controls the usual $C^{2,\alpha}$ norm on a smaller compact cylinder.
It also controls the physical boundary factor:
\begin{equation}\label{eq:massive-graph-boundary-factor}
 \sup_{r\ge0}(1+r)^\mu
 \left(\|w(r,\cdot)\|_{C^1([-1,1])}
       +\sup_{|z|<1}\frac{|w(r,z)|}{\ph(z)}\right)
 \le C\|w\|_{\mathcal G_\mu^\perp}.
\end{equation}
If a source has scaled derivatives in $X$ through order two with their
corresponding weights, then the solution has the same scaled
$X$-regularity.  In particular
\begin{equation}\label{eq:massive-same-order}
 \|\mathscr L_\perp^{-1}F\|_{C^{2,\alpha}_\mu}
 \le C\|F\|_{C^{2,\alpha}_\mu}
 \qquad(F\in P^\perp C^{2,\alpha}_\mu).
\end{equation}
\end{lemma}

\begin{proof}
This is the case $\rho=1$ of the uniform product-resolvent estimate in
\cref{prop:massive-product-resolvent}.  The compact-cylinder
$C^{2,\alpha}$ estimate and the boundary quotient estimate
\eqref{eq:massive-graph-boundary-factor} are included there.  The
same-order estimate \eqref{eq:massive-same-order} follows by commuting
$X$-derivatives with the product resolvent.  See also
Lockhart--McOwen~\cite{LM} for the general weighted elliptic framework on
noncompact manifolds; the appendix records the concrete polynomial-H\"older
version used here.
\end{proof}

\begin{definition}[Normalized Jacobi domain, target, and bordered operator]\label{def:normalized-jacobi-spaces}
Define the normalized domain and target by
\begin{equation}\label{eq:normalized-graph-spaces}
 \begin{split}
 \mathscr D_\mu
 &=\spanop\{\psi_-\}
   \oplus\bigl(C^{2,\alpha}_{\mu,\mathrm{rad}}(\R^m)\ph\bigr)
   \oplus\mathcal G_\mu^\perp,\\
 \mathscr Y_\mu
 &=P_1C^{0,\alpha}_{\mu+2}\oplus\mathcal Y_\mu^\perp.
 \end{split}
\end{equation}
The direct sums carry their natural sum norms.  Point evaluation at the
axis is continuous by the local estimate in
\cref{lem:massive-graph}.  For a profile $Q$, the \emph{normalized Jacobi
operator} is the bordered map
\[
 \mathcal N_QH:=(L_QH,H(0,0)):
 \mathscr D_\mu\longrightarrow\mathscr Y_\mu\times\R.
\]
The scalar component fixes the center-height normalization and removes
the branch tangent from the normalized kernel.
\end{definition}

Homogeneous fields in this graph domain have the stronger tail
regularity used in the Jacobi analysis above.  Indeed their first mode
satisfies $h=O_{C^{2,\alpha}_{\mathrm{sc}}}(r^{-\beta_-})$, and their
complement is initially bounded by $Cr^{-\mu}$.  Local Schauder
estimates first give a polynomial bound in the norm $\mathcal J_\ell$
used in \eqref{eq:jacobi-perp-annular}, for some finite $\ell$.
That estimate then yields
\[
 \|H^\perp\|_{C^{2,\alpha}_{\mathrm{sc}}(A_R)}
 \le CR^{-\beta_--2}+Ce^{-cR}.
\]
Thus the indicial bootstrap and the slow/fast alternative apply to all
homogeneous fields in $\mathscr D_\mu$.

\medskip\noindent\emph{Averaged coefficients.}
The preceding annular estimate and indicial argument also apply to
$L_{\overline Q}$ when $\overline Q=(Q_1+Q_2)/2$ is the average of
two smooth profiles with the same threshold tail.  Indeed every
coefficient estimate and every differentiated tail bound used above
is preserved by taking this average.  Projection of
$L_{\overline Q}H=0$ gives exactly
\eqref{eq:jacobi-first-general}--\eqref{eq:jacobi-first-source-general}
with $Q$ replaced by $\overline Q$.  Neither the nonlinear equation
for $\overline Q$ nor its semistability was used in those linear
arguments.  In particular, for every $\mu$ satisfying
\eqref{eq:mu-choice}, a polynomially decaying such solution has the
representation
\begin{equation}\label{eq:averaged-coefficient-domain}
 H=c_-\chi_\infty r^{-\beta_-}\ph+H_f,
 \quad P_1H_f\in C^{2,\alpha}_\mu\ph,
 \quad P^\perp H_f\in\mathcal G_\mu^\perp.
\end{equation}
For the last assertion the annular estimate gives
$H^\perp=O_{C^{2,\alpha}_{\rm sc}}(r^{-\beta_--2})$;
its projected equation gives
$\mathscr L_\perp H^\perp=O_{C^{0,\alpha}_{\rm sc}}(r^{-\beta_--2})$.
Both belong to weight $\mu$, since $\mu<\beta_-+2$.
The first-mode remainder lies in weight $\mu$ by the explicit
between-roots formula.  These statements assert membership for each
fixed pair; no uniform inverse is inferred from them.

\begin{proposition}[Normalized Fredholm index]\label{prop:fredholm}
If $Q\in\mathscr X_\sigma$ is a smooth semistable profile, then
\begin{equation}\label{eq:NQ}
 \mathcal N_Q:\mathscr D_\mu\to\mathscr Y_\mu\times\R,
 \qquad
 \mathcal N_QH=(L_QH,H(0,0)),
\end{equation}
is Fredholm of index zero.
\end{proposition}

\begin{proof}
Write
\[
 H=h(r)\ph+H^\perp,
 \qquad
 F=f(r)\ph+F^\perp,
 \qquad
 \int_{-1}^1H^\perp(r,z)\ph(z)\dd z=0.
\]
By \cref{lem:universal-tail}, applied with $\tau=\sigma$,
\begin{equation}\label{eq:Q-normal-tail-rig}
 Q=D_mr^{-2}\ph+R_Q,
 \qquad
 R_Q=O_{C^{2,\alpha}_{\rm sc}}(r^{-\sigma}).
\end{equation}
Since $2\La D_mB_2=4(m-4)$, the first-mode projection of the leading
potential equals the Euler potential.  More precisely, if
\[
 E:=-\partial_{rr}-\frac{m-1}{r}\partial_r
       -\frac{4(m-4)}{r^2},
 \qquad
 \mathscr L_\perp:=-\partial_{rr}-\frac{m-1}{r}\partial_r+H_z,
\]
and
\begin{equation}\label{eq:B-upper-rig}
 (BH^\perp)(r)
 :=-2\La D_mr^{-2}
   \int_{-1}^1\ph(z)^2H^\perp(r,z)\dd z,
\end{equation}
then the exact operator can be written on the exterior as
\begin{equation}\label{eq:normal-splitting-rig}
 L_Q
 \binom{h}{H^\perp}
 =
 \begin{pmatrix}E&B\\0&\mathscr L_\perp\end{pmatrix}
 \binom{h}{H^\perp}
 +\mathcal K_Q\binom{h}{H^\perp}.
\end{equation}
The upper-right operator $B$ is part of the tail normal operator; it is
not regarded as a small error.  Every remaining term gains a strictly
positive power of $r$ in the chosen domain/target weights.  Indeed:
\begin{enumerate}[label=(\roman*)]
\item multiplication by $R_Q$ gains $\sigma$ powers;
\item the leading coupling from $h\ph$ into $\ph^\perp$ is
      $O(r^{-2}h)$, and $\mu<\beta_-+\sigma-2<\beta_-+2$;
\item multiplication of $H^\perp$ by $Q$ in the transverse equation
      gains two powers;
\item the difference between the exact first-mode coefficient and
      $4(m-4)r^{-2}$ is $O(r^{-\sigma})$.
\end{enumerate}

To turn the tail normal expression into a global operator on the
weighted spaces, fix $R_0>2$ so that the tail expansion above holds for
$r\ge R_0$.  Choose $\chi_0\in C^\infty([0,\infty))$ with
$\chi_0=0$ on $[0,R_0/2]$ and $\chi_0=1$ on $[R_0,\infty)$.
Define
\begin{align*}
 \widetilde V(r)
 &:=4(m-4)\frac{\chi_0(r)}{r^2},\\
 \widetilde E
 &:=-\partial_{rr}-\frac{m-1}{r}\partial_r-\widetilde V(r),\\
 (\widetilde B H^\perp)(r)
 &:=-2\La D_m\frac{\chi_0(r)}{r^2}
   \int_{-1}^1\ph(z)^2H^\perp(r,z)\dd z.
\end{align*}
Both $\widetilde E$ and $\widetilde B$ are regular at the axis and agree
with $E$ and $B$, respectively, for $r\ge R_0$.  Set
\[
 \widetilde{\mathcal T}
 :=\begin{pmatrix}\widetilde E&\widetilde B\\0&\mathscr L_\perp\end{pmatrix},
 \qquad
 \mathcal K_Q:=L_Q-\widetilde{\mathcal T}.
\]
Thus $\mathcal K_Q$ is now a globally defined bounded operator
$\mathscr D_\mu\to\mathscr Y_\mu$.

We justify that $\mathcal K_Q$ is compact.  Let $\chi_R(r)$ be a smooth
cutoff which equals one for $r\le R$ and vanishes for $r\ge2R$, where
$R\ge2R_0$, and set $\mathcal K_{Q,R}:=\chi_R\mathcal K_Q$.  On a fixed cylinder, the first-mode norm and the local graph
estimate in \cref{lem:massive-graph} control the ordinary
$C^{2,\alpha}$ norm on a slightly smaller cylinder.  The embedding
of this bounded set into $C^{0,\alpha}$ is compact.  Since
$\mathcal K_{Q,R}$ is a zeroth-order operator with smooth compactly
supported coefficients, it follows that
$\mathcal K_{Q,R}:\mathscr D_\mu\to\mathscr Y_\mu$ is compact.
On the tail, (i)--(iv) give a strict gain in the chosen weights.  Hence
there is $\delta>0$ such that
\[
 \|(1-\chi_R)\mathcal K_Q\|_{\mathscr D_\mu\to\mathscr Y_\mu}
 \le CR^{-\delta}\longrightarrow0.
\]
Consequently $\mathcal K_Q$ is the operator-norm limit of the compact
operators $\mathcal K_{Q,R}$ and is compact.

Note that the choice of $\chi_0$ doesn't influence the index.  If
$(\widetilde E_j,\widetilde B_j)$, $j=1,2$, are two choices with different $\chi_0$, then
\[
 \widetilde E_1-\widetilde E_2=a_0(r),
 \qquad
 (\widetilde B_1-\widetilde B_2)H^\perp
 =a_1(r)\int_{-1}^1\ph^2H^\perp\dd z,
\]
with $a_0,a_1$ smooth and compactly supported.  By the same compact local
H\"older embedding, these differences are compact operators.  

We now compute the index of the upper triangular normal operator in
\eqref{eq:normal-splitting-rig}.

\medskip\noindent
\emph{The massive block.}
By \cref{lem:massive-graph}, for each
$F^\perp\in\mathcal Y_\mu^\perp$ there is a unique solution
$H^\perp\in\mathcal G_\mu^\perp$ of
$\mathscr L_\perp H^\perp=F^\perp$, and
\begin{equation}\label{eq:M-isom-rig}
 \|H^\perp\|_{\mathcal G_\mu^\perp}
 \le C\|F^\perp\|_{C^{0,\alpha}_\mu}.
\end{equation}
Thus this graph realization is an isomorphism and has index zero.

\medskip\noindent
\emph{The Euler block and its index.}
The exterior Euler expression is
\[
 E_\infty h:=-h''-\frac{m-1}{r}h'-\frac{4(m-4)}{r^2}h,
 \qquad r\ge R_0,
\]
where $R_0>2$ is fixed.  Its homogeneous solutions are
$y_-=r^{-\beta_-}$ and $y_+=r^{-\beta_+}$, with
\[
 W(y_-,y_+)=-(\beta_+-\beta_-)r^{-(m-1)}.
\]
For $f\in C^{0,\alpha}_{\mu+2}(r>R_0)$ define
\begin{align}\label{eq:euler-VOC-rig}
 (\mathcal G_Ef)(r)
 :=\frac1{\beta_+-\beta_-}\Bigg[
 &r^{-\beta_-}\int_r^\infty s^{\beta_-+1}f(s)\dd s\\
 &+r^{-\beta_+}\int_{R_0}^r s^{\beta_++1}f(s)\dd s\Bigg].\notag
\end{align}
Since $\beta_-<\mu<\beta_+$, both integrals converge in the indicated
direction, direct differentiation gives $E_\infty\mathcal G_Ef=f$, and
\begin{equation}\label{eq:euler-weight-est-rig}
 \norm{\mathcal G_Ef}_{C^{2,\alpha}_\mu(r>R_0)}
 \le C\norm{f}_{C^{0,\alpha}_{\mu+2}(r>R_0)}.
\end{equation}
Thus every exterior solution belonging to the admissible weighted class
has the unique form
\begin{equation}\label{eq:euler-general-rig}
 h=c_-r^{-\beta_-}+c_+r^{-\beta_+}+\mathcal G_Ef.
\end{equation}
The slow coefficient is represented in the global domain by the added
vector $\chi_\infty r^{-\beta_-}$, while
$r^{-\beta_+}\in C^{2,\alpha}_\mu$ because $\beta_+>\mu$.

The exterior expression $E_\infty$ is only a tail model.  We use the
smooth global extension $\widetilde E$ fixed above; it agrees with
$E_\infty$ for $r\ge R_0$ and is a smooth radial elliptic operator at the
axis.  Introduce the scalar spaces
\[
 \mathscr D^E_\mu
 :=\spanop\{\chi_\infty r^{-\beta_-}\}
      \oplus C^{2,\alpha}_{\mu,\mathrm{rad}}(\R^m),
 \qquad
 \mathscr Y^E_\mu
 :=C^{0,\alpha}_{\mu+2,\mathrm{rad}}(\R^m).
\]

By the compact-extension argument above, the index of this scalar block is independent of the chosen smooth core extension.  It therefore suffices to compute the index of the fixed operator $\widetilde E$.

We next prove that
\[
 \widetilde E:\mathscr D^E_\mu\longrightarrow\mathscr Y^E_\mu
\]
is onto and has one-dimensional kernel.  Let
$f\in\mathscr Y^E_\mu$ be arbitrary.  For each $a\in\R$, solve on
$[0,R_0]$ the regular radial initial-value problem
\begin{equation}\label{eq:euler-core-ivp}
 \widetilde Eh=f,
 \qquad h(0)=a,
 \qquad h'(0)=0.
\end{equation}
The condition $h'(0)=0$ is exactly the radial smoothness condition at the
axis.  Standard regular ODE theory (equivalently, the integral identity
\[
 h'(r)=-r^{1-m}\int_0^r
 s^{m-1}\bigl(f(s)+\widetilde V(s)h(s)\bigr)\dd s
\]
near $r=0$) gives a unique smooth radial solution on $[0,R_0]$.  Denote
its Cauchy data at $R_0$ by
\[
 A_a:=h(R_0),\qquad B_a:=h'(R_0).
\]
On $[R_0,\infty)$ every admissible solution is of the form
\eqref{eq:euler-general-rig}.  The constants $(c_-,c_+)$ are chosen so
that this exterior solution matches $(A_a,B_a)$ at $R_0$.  They solve
\[
 \begin{pmatrix}
 R_0^{-\beta_-}&R_0^{-\beta_+}\\
 -\beta_-R_0^{-\beta_--1}&-\beta_+R_0^{-\beta_+-1}
 \end{pmatrix}
 \binom{c_-}{c_+}
 =
 \binom{A_a-(\mathcal G_Ef)(R_0)}
       {B_a-(\mathcal G_Ef)'(R_0)}.
\]
The determinant is
\begin{equation}\label{eq:euler-matching-det}
 -(\beta_+-\beta_-)R_0^{-(m-1)}\ne0,
\end{equation}
so the matching coefficients exist and are unique.  The resulting global
function belongs to $\mathscr D^E_\mu$: on the exterior the slow term is
represented by $c_-\chi_\infty r^{-\beta_-}$, while the fast term and the
particular solution belong to $C^{2,\alpha}_\mu$. Taking, for
instance, $a=0$ gives a global solution for every $f$, hence
\[
 \operatorname{coker}\widetilde E=0.
\]

For $f=0$, every regular solution of \eqref{eq:euler-core-ivp} is uniquely determined by
the single number $a=h(0)$. Hence
\[
 \dim\ker\widetilde E=1.
\]
Therefore $\widetilde E$ is Fredholm and
\begin{equation}\label{eq:E-index-rig}
 \operatorname{ind}\widetilde E
 =\dim\ker\widetilde E-\dim\operatorname{coker}\widetilde E
 =1.
\end{equation}

\medskip\noindent
\emph{Index of the upper triangular normal operator.}
The global normal operator is the already defined
\[
 \widetilde{\mathcal T}
 =\begin{pmatrix}\widetilde E&\widetilde B\\0&\mathscr L_\perp\end{pmatrix}.
\]
The operator $\widetilde B:\mathcal G_\mu^\perp\to
C^{0,\alpha}_{\mu+2}$ is bounded: on the tail it is multiplication by
$r^{-2}$ followed by a fixed transverse functional, while on the core it
has smooth compactly supported coefficients.  Since $\mathscr L_\perp$
is an isomorphism, the target-space map
\[
 \mathcal P(f,F^\perp)
 :=\bigl(f-\widetilde B\mathscr L_\perp^{-1}F^\perp,F^\perp\bigr)
\]
is a bounded automorphism with inverse
\[
 \mathcal P^{-1}(f,F^\perp)
 =\bigl(f+\widetilde B\mathscr L_\perp^{-1}F^\perp,F^\perp\bigr).
\]
A direct calculation gives
\[
 \mathcal P\widetilde{\mathcal T}(h,H^\perp)
 =\bigl(\widetilde Eh,\mathscr L_\perp H^\perp\bigr).
\]
Thus $\widetilde{\mathcal T}$ is Fredholm-equivalent to
$\widetilde E\oplus\mathscr L_\perp$, and therefore
\[
 \operatorname{ind}\widetilde{\mathcal T}
 =\operatorname{ind}\widetilde E
  +\operatorname{ind}\mathscr L_\perp
 =1+0=1.
\]

Since $L_Q=\widetilde{\mathcal T}+\mathcal K_Q$ and $\mathcal K_Q$ is
compact, Atkinson's theorem gives that $L_Q$ is Fredholm with the same
index.  Hence
\begin{equation}\label{eq:LQ-index-rig}
 \operatorname{ind}(L_Q:\mathscr D_\mu\to\mathscr Y_\mu)=1.
\end{equation}
Finally $(L_Q,H\mapsto H(0,0))$ is a finite-rank perturbation of
$(L_Q,0):\mathscr D_\mu\to\mathscr Y_\mu\times\R$.  The latter has index
$\operatorname{ind}L_Q-1=0$, and finite-rank perturbations preserve the
index.  Therefore
\[
 \operatorname{ind}\mathcal N_Q=0.
\]
This proves the proposition.
\end{proof}

\subsection{The small ordered branch}

Let $c_0=\La B_2$.
\begin{definition}[Regular Lane--Emden profile]\label{def:lane-emden-profile}
The \emph{regular Lane--Emden profile} $V=V_m$ is the unique maximal
radial solution of the initial-value problem
\begin{equation}\label{eq:LE}
 -V''-\frac{m-1}{r}V'=c_0V^2,
 \qquad V(0)=1,
 \qquad V'(0)=0.
\end{equation}
Equivalently, $V(|Y|)$ solves $-\Delta_YV=c_0V^2$ in its radial domain
of existence and is smooth at $Y=0$.  The adjective \emph{regular}
distinguishes it from the singular scale-invariant profile
$2(m-4)/(c_0r^2)$ used in the separation argument below.
\end{definition}
For $m>6$ the solution remains positive: otherwise its first zero
would yield a positive Dirichlet solution of
$-\Delta V=c_0V^2$ in a ball, contrary to Pohozaev's variational
identity \cite{Pohozaev}, since the exponent $2$ is critical or
supercritical when $m\ge6$ and strictly supercritical for $m>6$.
Once positivity is known, global continuation is immediate. Indeed,
\[
 (r^{m-1}V')'=-c_0r^{m-1}V^2<0,
\]
so $V'(r)<0$ for $r>0$ and hence $0<V(r)\le1$ throughout its maximal
interval of existence. Moreover the integrated radial equation gives
\[
 -V'(r)=c_0r^{1-m}\int_0^r s^{m-1}V(s)^2\,\dd s
 \le \frac{c_0}{m}r.
\]
Thus both $V$ and $V'$ stay bounded on every finite radial interval,
and the standard ODE continuation theorem gives a global solution on
$[0,\infty)$. The
existence and qualitative analysis of regular positive radial
Lane--Emden states in the critical and supercritical regimes is
classical; see, for example, Joseph--Lundgren \cite{JL} and
Gui--Ni--Wang \cite{GNW}.  For broader classification and Liouville
theory for Lane--Emden solutions on unbounded domains, including stable
and finite-Morse-index classes, see Farina \cite{Farina2007}; for stable
entire solutions of general convex semilinear equations, see
Dupaigne--Farina \cite{DupaigneFarina}.  We retain the preceding argument
because it matches exactly the normalization used here.

The following lemma is the form of the Joseph--Lundgren separation
principle needed here; see \cite{JL}.  We include the proof because we
also need the sign of the scaling field and the first nonzero tail
coefficient in the weighted Lyapunov--Schmidt construction.

\begin{lemma}[Lane--Emden separation]\label{lem:LE}
Assume
\begin{equation}\label{eq:JL-condition}
 \frac{(m-2)^2}{4}>4(m-4).
\end{equation}
Then
\begin{equation}\label{eq:LE-sep}
 0<V(r)<\frac{2(m-4)}{c_0r^2},
 \qquad
 Z(r):=2V(r)+rV'(r)>0.
\end{equation}
Moreover
\begin{equation}\label{eq:LE-tail}
 V(r)=D_mr^{-2}+c_{LE}r^{-\kappa}+o(r^{-\kappa}),
 \qquad \kappa\in\{\beta_-,\beta_+\},\quad c_{LE}<0
\end{equation}
for the first nonzero correction coefficient.
\end{lemma}

\begin{proof}
The first inequality follows from the same first-contact/Hardy argument as \cref{lem:scalar-separation}, applied to $W=2(m-4)/(c_0r^2)$. Put $t=\log r$ and $Y=c_0r^2V$. Then
\begin{equation}\label{eq:LE-autonomous}
 Y''+(m-6)Y'-2(m-4)Y+Y^2=0,
 \qquad
 Y'=c_0r^2Z.
\end{equation}
For $t\ll0$, $Y'>0$. If $Y'$ had a first zero, then $0<Y<2(m-4)$ there and \eqref{eq:LE-autonomous} would give
\[
 Y''=Y\bigl(2(m-4)-Y\bigr)>0,
\]
contradicting first contact with zero from above. Hence $Z>0$ and $Y'>0$ on $\R$.  By \eqref{eq:LE-sep},
\[
 0<Y(t)=c_0r^2V(r)<2(m-4),
\]
so $Y$ is increasing and converges to some $L\in(0,2(m-4)]$.
To identify $L$, repeating the arguments of \cref{clm:Y-energy-limit,clm:Y-omega-limit} with the forcing $\rho\equiv0$, we get
\[
 L\bigl(L-2(m-4)\bigr)=0.
\]
Since $L>0$, necessarily $L=2(m-4)$, and therefore
\[
 Y(t)\uparrow2(m-4).
\]

Set $\eta(t):=Y(t)-2(m-4)$.  Then
\[
 \eta''+(m-6)\eta'+2(m-4)\eta+\eta^2=0.
\]
The linearized characteristic exponents are $-(\beta_- -2)$ and $-(\beta_+-2)$.  Applying the scalar variation-of-constants argument in \cref{lem:jacobi-slow-fast} (with the transverse remainder absent) gives either
\[
 \eta(t)=a_-e^{-(\beta_--2)t}+o\bigl(e^{-(\beta_--2)t}\bigr),
 \qquad a_-\ne0,
\]
or, if the slow coefficient vanishes,
\[
 \eta(t)=a_+e^{-(\beta_+-2)t}+o\bigl(e^{-(\beta_+-2)t}\bigr),
 \qquad a_+\ne0.
\]
If both asymptotic coefficients vanished, the same variation-of-constants
formula has no homogeneous term.  Starting from the preliminary decay
$|\eta|+|\eta_t|=O(e^{-\gamma t})$ with any
$0<\gamma<\beta_--2$, its right-hand side is
$O(e^{-2\gamma t})$; applying the zero-coefficient formula again doubles
the admissible decay exponent.  Iteration therefore gives
\[
 |\eta(t)|+|\eta_t(t)|=O(e^{-Mt})
 \qquad\text{for every }M>0.
\]
Write $X=(\eta,\eta_t)$.  On a sufficiently small neighborhood of the
origin the first-order system $X'=F(X)$ satisfies $|F(X)|\le C|X|$.
Backward Gronwall consequently yields, for $t\ge T$,
\[
 |X(T)|\le e^{C(t-T)}|X(t)|.
\]
Choosing $M>C$ in the preceding superexponential estimate and letting
$t\to\infty$ gives $X(T)=0$.  Uniqueness for the ODE then yields
$\eta\equiv0$, which is impossible.  

Since $Y=c_0r^2V$ and $c_0D_m=2(m-4)$, this yields exactly \eqref{eq:LE-tail} with $\kappa\in\{\beta_-,\beta_+\}$.  Finally $Z=2V+rV'$ is positive and $(2-\kappa)c_{LE}$ is the leading coefficient of $Z$, so $c_{LE}<0$.
\end{proof}

For $m\ge16$, condition \eqref{eq:JL-condition} holds.  We now
perform a weighted Lyapunov--Schmidt reduction.  The abstract
simple-eigenvalue mechanism is classical; compare
Crandall--Rabinowitz \cite{CR71}.  Here the domain is unbounded and the
parameter enters a singular transverse scaling, so the weighted
resolvent, normalization and differentiability estimates are proved
below rather than imported from the abstract theorem.  For small
$s>0$, put $Y=sX$ and seek
\begin{equation}\label{eq:LS-ansatz}
 Q_s(X,z)=s^2\bigl(v_s(Y)\ph(z)+s^2w_s(Y,z)\bigr),
 \qquad \int_{-1}^1w_s(Y,z)\ph(z)\dd z=0.
\end{equation}

Projection gives
\begin{align}
 -\Delta_Yv_s&=\La\int_{-1}^1(v_s\ph+s^2w_s)^2\ph\dd z,\label{eq:LS1}\\
 (-s^2\Delta_Y+H_z)w_s&=\La P^\perp(v_s\ph+s^2w_s)^2.\label{eq:LS2}
\end{align}
At $s=0$, $v_0=V$, we have
\[
 w_0=\La H_z^{-1}P^\perp(V^2\ph^2).
\]
Set $\rho=s^2$.  We solve \eqref{eq:LS1}--\eqref{eq:LS2} by a
parameter-dependent transverse resolvent, rather than by treating
$H_z^{-1}$ as if it generated derivatives in the $Y$ variables.

For $\gamma>0$ and an integer $k\ge0$, let
\[
 \|f\|_{\mathcal F_\gamma^k}
   :=\sum_{|\nu|\le k}
           \|\partial_Y^\nu f\|_{C^{0,\alpha}_{\gamma+|\nu|}}.
\]
For $k\ge2$ define the mixed weighted norm
\begin{equation}\label{eq:mixed-high-norms}
 \begin{split}
 \|w\|_{\mathcal W_\gamma^{(k)}}
 := {}&\|w\|_{\mathcal F_\gamma^k}
 +\sum_{|\nu|\le k-1}
       \|\partial_z\partial_Y^\nu w\|_{C^{0,\alpha}_{\gamma+|\nu|}}\\
 &+\sum_{|\nu|\le k-2}
       \|\partial_{zz}\partial_Y^\nu w\|_{C^{0,\alpha}_{\gamma+|\nu|}}.
 \end{split}
\end{equation}
Here and below the weight is in $Y$ and $z$ is unscaled.  For all
$\gamma>0$ and $k\ge2$ the solution space is
\[
 \mathcal W_{\gamma,\perp}^{(k)}
 =\{w\in\mathcal W_\gamma^{(k)}:
 w(Y,\pm1)=0,\ P_1w=0,\ w\text{ is radial in }Y
                       \text{ and even in }z\}.
\]
Radial axis compatibility is understood in Cartesian coordinates.
We abbreviate $\mathcal W_\sigma=\mathcal W_\sigma^{(2)}$ and
$\mathcal W_{\sigma,\perp}=\mathcal W_{\sigma,\perp}^{(2)}$.
The source spaces impose no boundary trace condition.

\begin{claim}\label{clm:transverse-rho-resolvent}
Fix $\gamma>0$, an integer $k\ge2$, and $\rho_0>0$.
For every radial-even source $f\in\mathcal F_\gamma^k$ satisfying
$P_1f=0$, and every $0\le\rho\le\rho_0$, the Dirichlet problem
\[
 (H_z-\rho\Delta_Y)w=f,\qquad w(Y,\pm1)=0,\qquad P_1w=0
\]
has a unique solution in $\mathcal W_{\gamma,\perp}^{(k)}$, and
\[
 \|w\|_{\mathcal W_\gamma^{(k)}}\le C_k\|f\|_{\mathcal F_\gamma^k}.
\]
In particular, when $k=2$ and $\gamma=\sigma$,
\begin{equation}\label{eq:transverse-rho-resolvent}
 \|w\|_{\mathcal W_\sigma}
 \le C\sum_{|\nu|\le2}
       \|\partial_Y^\nu f\|_{C^{0,\alpha}_{\sigma+|\nu|}}.
\end{equation}
All constants are independent of $\rho\in[0,\rho_0]$.
\end{claim}

\begin{proof}
Apply \cref{prop:massive-product-resolvent} with the parameter
$0\le\rho\le\rho_0$.  The orthogonality condition $P_1f=0$ is necessary
and ensures that the product resolvent acts on the full stated equation,
not merely on its projected part.  The proposition gives the uniform
mixed estimate \eqref{eq:transverse-rho-resolvent} and uniqueness.
\end{proof}

\begin{proposition}[Weighted Lyapunov--Schmidt construction at the threshold]
\label{prop:weighted-LS-construction}
Assume $m\ge16$ and choose $2<\sigma<\min\{4,\beta_-\}$.  There exists
$s_0>0$ such that for $0<s<s_0$, with $\rho=s^2$, the projected system
\eqref{eq:LS1}--\eqref{eq:LS2} has a unique radial-even solution in a
fixed weighted neighborhood of $(V,w_0)$ satisfying $v_s(0)=1$.  It obeys
\begin{equation}\label{eq:LS-proposition-expansion}
 v_s=V+O_{C^{2,\alpha}_\sigma}(s^2),\qquad
 w_s=w_0+O_{\mathcal W_\sigma}(s^2),
\end{equation}
is $C^1$ in the parameter $\rho=s^2$, and the corresponding cylinder
solution
\[
 Q_s(X,z)=s^2\bigl(v_s(sX)\ph(z)+s^2w_s(sX,z)\bigr)
\]
is positive and satisfies the coercive estimate
\begin{equation}\label{eq:small-coerc-prop}
 \mathfrak q_{Q_s}[\zeta]
 \ge c_0'\int_{\R^m}\frac{\zeta_1^2}{r^2}
      +c_1'\|\zeta^\perp\|_{H^1(\Cyl)}^2,
 \qquad 0<s<s_0,
\end{equation}
for the decomposition $\zeta=\zeta_1\ph+\zeta^\perp$.  Moreover
$A(s)=Q_s(0,0)$ satisfies $A'(s)=2s+O(s^3)>0$.
\end{proposition}

\begin{proof}
\emph{Step 1: the transverse fixed point.}
Write
\[
 v=V+\xi,\qquad w=w_0+\omega,
 \qquad \xi(0)=0,
\]
and let $\mathcal R_\rho$ denote the transverse resolvent from
\cref{clm:transverse-rho-resolvent}.  More precisely,
\[
 \mathcal R_\rho=(H_z-\rho\Delta_Y)^{-1}
\]
on the radial-even Dirichlet $P_1$-orthogonal realization specified there;
at $\rho=0$ this reduces to the pointwise inverse $H_z^{-1}$ on
$\varphi^\perp$.  For a fixed $\xi$, define
$\mathcal T_{\rho,\xi}(\omega)$ by
\begin{align}
 \mathcal T_{\rho,\xi}(\omega)
 :=\mathcal R_\rho\Big[{}&\rho\Delta_Yw_0
 +\La P^\perp\bigl(2V\xi\ph^2+\xi^2\ph^2\bigr)\notag\\
 &+2\La\rho P^\perp\bigl((V+\xi)\ph(w_0+\omega)\bigr)
 +\La\rho^2P^\perp(w_0+\omega)^2\Big].
 \label{eq:LS-omega-fixed-point}
\end{align}
Equation $\omega=\mathcal T_{\rho,\xi}(\omega)$ is precisely the fixed-point
form of the transverse equation \eqref{eq:LS2}.  Indeed, substituting
$v=V+\xi$ and $w=w_0+\omega$ into \eqref{eq:LS2}, using
\[
 H_zw_0=\La P^\perp(V^2\ph^2),
\]
and expanding the square gives
\begin{align*}
 (H_z-\rho\Delta_Y)\omega={}&\rho\Delta_Yw_0
 +\La P^\perp(2V\xi\ph^2+\xi^2\ph^2)\\
 &+2\La\rho P^\perp((V+\xi)\ph(w_0+\omega))
 +\La\rho^2P^\perp(w_0+\omega)^2.
\end{align*}
 Conversely, every fixed point of
\eqref{eq:LS-omega-fixed-point} satisfies \eqref{eq:LS2}; hence the two
formulations are equivalent for fixed $\xi$.

The product estimates in the weighted spaces and
\eqref{eq:transverse-rho-resolvent} imply, for
$\|\xi\|_{C^{2,\alpha}_\sigma}\le c_*$ and
$\|\omega\|_{\mathcal W_\sigma}\le C_*$,
\begin{align}
 \|\mathcal T_{\rho,\xi}(\omega)\|_{\mathcal W_\sigma}
 &\le C\bigl(\rho+\|\xi\|_{C^{2,\alpha}_\sigma}
       +\rho\|\omega\|_{\mathcal W_\sigma}\bigr),
 \label{eq:LS-omega-bound}\\
 \|\mathcal T_{\rho,\xi}(\omega_1)
   -\mathcal T_{\rho,\xi}(\omega_2)\|_{\mathcal W_\sigma}
 &\le C\rho\|\omega_1-\omega_2\|_{\mathcal W_\sigma}.
 \label{eq:LS-omega-contraction}
\end{align}
We now choose the fixed-point ball explicitly.  First fix $C_*>0$.
Choose $c_*>0$ so small that $Cc_*\le C_*/4$, and then decrease
$\rho_0>0$ so that
\[
 C\rho_0(1+C_*)\le C_*/4,
 \qquad C\rho_0\le\frac12.
\]
Then \eqref{eq:LS-omega-bound} shows that, whenever
$\|\xi\|_{C^{2,\alpha}_\sigma}\le c_*$,
\[
 \mathcal T_{\rho,\xi}:
 \overline B_{C_*}^{\mathcal W_\sigma}
 \longrightarrow
 \overline B_{C_*}^{\mathcal W_\sigma},
\]
while \eqref{eq:LS-omega-contraction} makes this map a contraction,
uniformly for $0\le\rho\le\rho_0$.  Hence the Banach fixed-point theorem
gives, for every admissible $\xi$, a unique solution $\omega$ in this
ball of
\[
 \omega=\mathcal T_{\rho,\xi}(\omega).
\]
We denote it by
\[
 \Omega_\rho(\xi):=\omega,
 \qquad
 \Omega_\rho(\xi)
 =\mathcal T_{\rho,\xi}\bigl(\Omega_\rho(\xi)\bigr).
\]
Thus $\Omega_\rho$ is the nonlinear solution map
$\xi\mapsto\omega$ obtained from the fixed-point equation.
The estimates above yield
\begin{equation}\label{eq:LS-omega-Lipschitz}
 \|\Omega_\rho(\xi)\|_{\mathcal W_\sigma}
 \le C(\rho+\|\xi\|_{C^{2,\alpha}_\sigma}),
 \qquad
 \|\Omega_\rho(\xi_1)-\Omega_\rho(\xi_2)\|_{\mathcal W_\sigma}
 \le C\|\xi_1-\xi_2\|_{C^{2,\alpha}_\sigma}.
\end{equation}

\medskip\noindent
\emph{Step 2: the normalized Lane--Emden equation.}
Let $\mathcal S$ be the normalized inverse of $L_{LE}$ supplied by
\cref{prop:normalized-voc}, with $2<\sigma<\beta_-$.  Subtracting the
Lane--Emden equation for $V$ from \eqref{eq:LS1} gives
\begin{align}
 L_{LE}\xi={}&c_0\xi^2
 +2\La\rho(V+\xi)
     \int_{-1}^1\ph^2(w_0+\Omega_\rho(\xi))\,dz\notag\\
 &+\La\rho^2\int_{-1}^1
      (w_0+\Omega_\rho(\xi))^2\ph\,dz.
 \label{eq:LS-xi-equation}
\end{align}
Define the right-hand side after applying $\mathcal S$ to be
$\mathcal K_\rho(\xi)$.  The weighted algebra estimate and
\eqref{eq:LS-omega-Lipschitz} give, on the ball
$\|\xi\|_{C^{2,\alpha}_\sigma}\le K\rho$,
\[
 \|\mathcal K_\rho(\xi)\|_{C^{2,\alpha}_\sigma}
 \le C(\rho+\|\xi\|_{C^{2,\alpha}_\sigma}^2)
 \le K\rho
\]
for $K$ fixed large and $\rho_0$ small, and
\[
 \|\mathcal K_\rho(\xi_1)-\mathcal K_\rho(\xi_2)\|_{C^{2,\alpha}_\sigma}
 \le C(\rho+K\rho)
 \|\xi_1-\xi_2\|_{C^{2,\alpha}_\sigma}.
\]
Hence $\mathcal K_\rho$ is a contraction.  Its unique fixed point
$\xi_\rho$ and the corresponding
$\omega_\rho=\Omega_\rho(\xi_\rho)$ solve
\eqref{eq:LS1}--\eqref{eq:LS2} with $v(0)=1$, and
\begin{equation}\label{eq:LS-expansion}
 v_s=V+O_{C^{2,\alpha}_\sigma}(s^2),
 \qquad
 w_s=w_0+O_{\mathcal W_\sigma}(s^2).
\end{equation}

It remains to justify the uniqueness in the fixed weighted
neighborhood asserted in the proposition.  Let $(\xi,\omega)$ be any
solution of \eqref{eq:LS1}--\eqref{eq:LS2} satisfying
\[
 \|\xi\|_{C^{2,\alpha}_\sigma}\le c_*,
 \qquad \|\omega\|_{\mathcal W_\sigma}\le C_*.
\]
Applying \eqref{eq:LS-omega-bound} to this solution and absorbing the term
$C\rho\|\omega\|_{\mathcal W_\sigma}$ gives
\[
 \|\omega\|_{\mathcal W_\sigma}
 \le C_1\bigl(\rho+\|\xi\|_{C^{2,\alpha}_\sigma}\bigr).
\]
The normalized scalar equation \eqref{eq:LS-xi-equation}, the boundedness
of $\mathcal S$, and the weighted algebra estimates then imply
\[
 \|\xi\|_{C^{2,\alpha}_\sigma}
 \le C_2\left(\rho+\|\xi\|_{C^{2,\alpha}_\sigma}^2\right).
\]
After decreasing the fixed radius $c_*$, if necessary, so that
$C_2c_*\le1/2$, the quadratic term is absorbed and
\[
 \|\xi\|_{C^{2,\alpha}_\sigma}\le2C_2\rho,
 \qquad
 \|\omega\|_{\mathcal W_\sigma}\le C_3\rho.
\]
Enlarging the constant $K$ in the contraction ball, if necessary, so that
$K>2C_2$, and then decreasing $\rho_0$ once more, every solution in the
fixed weighted neighborhood lies in the $K\rho$ contraction ball.  It
therefore coincides with $(\xi_\rho,\omega_\rho)$.  This proves the
uniqueness claimed in the statement.

\par\medskip
\noindent\emph{Step 3: parameter regularity.}
The only delicate regularity point is differentiation at $\rho=0$, where
$H_z-\rho\Delta_Y$ loses uniform ellipticity in the $Y$ variables.
\Cref{prop:LS-parameter-regularity} proves the required higher-order
bounds and difference-quotient estimates.  In particular,
\begin{equation}\label{eq:LS-six-derivative-bound}
 \|\xi_\rho\|_{C^{6,\alpha}_\sigma}
 +\|\omega_\rho\|_{\mathcal W_\sigma^{(6)}}\le C\rho,
 \qquad 0\le\rho\le\rho_0,
\end{equation}
and the map $\rho\mapsto(\xi_\rho,\omega_\rho)$ is $C^1$ into the
two-derivative weighted spaces with uniformly bounded derivative.
Consequently $A(s)=Q_s(0,0)$ satisfies
\begin{equation}\label{eq:LS-center-derivative-verified}
 \begin{split}
 A(s)&=s^2+s^4w_{s^2}(0,0),\\
 A'(s)&=2s+4s^3w_{s^2}(0,0)
             +2s^5\partial_\rho w_{s^2}(0,0)=2s+O(s^3).
 \end{split}
\end{equation}
The same appendix argument justifies differentiation of the weighted
affine tail.  Elliptic bootstrap gives smoothness in $(X,z)$ for $s>0$.

\medskip\noindent
\emph{Step 4: positivity and coercivity.}
The constructed solutions are positive.  The Lane--Emden tail and
positivity on compact sets give $V(Y)\ge c(1+|Y|)^{-2}$.
The estimates for $v_s-V$ and $\partial_z w_s$, together with the zero
Dirichlet trace, imply
\[
 v_s(Y)\ph(z)+s^2w_s(Y,z)
 \ge\ph(z)\bigl[V(Y)-Cs^2(1+|Y|)^{-\sigma}\bigr]>0
\]
for $s>0$ sufficiently small, because $\sigma>2$.

We now establish a strict stability for small $s$.  Indeed,
$2c_0V(r)<4(m-4)r^{-2}$, while the sharp Hardy inequality in $\R^m$ gives
$\int|\nabla\xi|^2\ge (m-2)^2/4\int r^{-2}\xi^2$.  Hence
\begin{equation}\label{eq:LE-coerc}
 \int_{\R^m}(|\nabla\xi|^2-2c_0V\xi^2)
 \ge\delta_{LE}\int_{\R^m}\frac{\xi^2}{r^2},
 \qquad
 \delta_{LE}:=\frac{(m-2)^2}{4}-4(m-4)>0.
\end{equation}
Write
\[
 \zeta(X,z)=\zeta_1(X)\ph(z)+\zeta^\perp(X,z),
 \qquad
 \int_{-1}^1\zeta^\perp(X,z)\ph(z)\dd z=0.
\]
Expanding $\mathfrak q_{Q_s}[\zeta]$ gives
\begin{align}
 \mathfrak q_{Q_s}[\zeta]
 =&
 \int_{\R^m}|\nabla_X\zeta_1|^2\,\dd X-2\La\int_{\Cyl}
 Q_s\,\zeta_1^2\ph^2\nonumber\\&+
 \int_{\Cyl}
 \Bigl(
   |\nabla_X\zeta^\perp|^2
   +|\partial_z\zeta^\perp|^2
   -\La|\zeta^\perp|^2
 \Bigr)
 \nonumber\\
 &
 -4\La\int_{\Cyl}
 Q_s\,\zeta_1\ph\,\zeta^\perp
 -2\La\int_{\Cyl}
 Q_s\,|\zeta^\perp|^2 .\label{eq:q-exp}
\end{align}
Using \eqref{eq:LS-ansatz} and \eqref{eq:LS-expansion},
the first line in the right hand side of \eqref{eq:q-exp} becomes
\begin{align*}
 &\int_{\R^m}
 \Bigl(
   |\nabla_X\zeta_1|^2
   -2c_0s^2V(sX)\zeta_1^2
 \Bigr)\dd X-2c_0s^2
 \int_{\R^m}
 \bigl(v_s(sX)-V(sX)\bigr)\zeta_1^2\,\dd X
 \\
 &-2\La s^4
 \int_{\R^m}\zeta_1(X)^2
 \left(
   \int_{-1}^1w_s(sX,z)\ph(z)^2\,\dd z
 \right)\dd X.
\end{align*}
Therefore, 
\begin{align}
 \mathfrak q_{Q_s}[\zeta]
 ={}&
 \int_{\R^m}
 \Bigl(
   |\nabla_X\zeta_1|^2
   -2c_0s^2V(sX)\zeta_1^2
 \Bigr)\dd X
 \nonumber\\
 &+
 \int_{\Cyl}
 \Bigl(
   |\nabla_X\zeta^\perp|^2
   +|\partial_z\zeta^\perp|^2
   -\La|\zeta^\perp|^2
 \Bigr)
 \nonumber\\
 &+I_{1}+I_{2}+I_{3},\label{eq:q-exp2}
\end{align}
where
\begin{align*}
 I_1
 :={}&
 -2c_0s^2
 \int_{\R^m}
 \bigl(v_s(sX)-V(sX)\bigr)\zeta_1^2\,\dd X
 \\
 &\quad
 -2\La s^4
 \int_{\R^m}\zeta_1(X)^2
 \left(
   \int_{-1}^1
   w_s(sX,z)\ph(z)^2\,\dd z
 \right)\dd X,
 \\
 I_2
 :={}&
 -4\La
 \int_{\Cyl}
 Q_s(X,z)\,
 \zeta_1(X)\ph(z)\zeta^\perp(X,z),
 \\
 I_3
 :={}&
 -2\La
 \int_{\Cyl}
 Q_s(X,z)|\zeta^\perp(X,z)|^2 .
\end{align*}

Since the second Dirichlet eigenvalue of $-\partial_{zz}$ on $\ph^\perp$ is $4\La$, equivalently $H_z=-\partial_{zz}-\La\ge3\La$ on $\ph^\perp$, we have
\[
    \int_{-1}^1\bigl(|\partial_z\zeta^\perp|^2-\La|\zeta^\perp|^2\bigr)\dd z
 \ge3\La\int_{-1}^1|\zeta^\perp|^2\dd z.
\]
Therefore, there exists a fixed constant $c_\perp>0$ such that
\begin{equation}
\int_\Cyl\Bigl(|\nabla_X\zeta^\perp|^2+|\partial_z\zeta^\perp|^2-\La|\zeta^\perp|^2\Bigr)
 \ge c_\perp\norm{\zeta^\perp}_{H^1(\Cyl)}^2.\label{eq:perp-spectral-coercivity}
\end{equation}

After the change of variables $Y=sX$, \eqref{eq:LE-coerc} therefore gives the scale-invariant estimate
\begin{equation}
    \int_{\R^m}\Bigl(|\nabla_X\zeta_1|^2-2c_0s^2V(sX)\zeta_1^2\Bigr)\dd X
 \ge\delta_{LE}\int_{\R^m}\frac{\zeta_1^2}{r^2}\dd X.\label{eq:LE-first-mode-coercivity}
\end{equation}

It remains to bound the terms $I_1, I_2, I_3$. Note that
\(
 0<V(Y)\le C\min\{1,|Y|^{-2}\},
\)
and since $H_z^{-1}$ acts only in the $z$ variable, the definition of $w_0$ gives $w_0=O((1+|Y|)^{-4})$.  Hence \eqref{eq:LS-expansion}, together with $\sigma<4$, implies uniformly for small $s$ that
\[
 |v_s(Y)-V(Y)|\le Cs^2(1+|Y|)^{-\sigma},
 \qquad
 |w_s(Y,z)|\le C(1+|Y|)^{-\sigma}.
\]
Since $\sigma>2$, with $t=sr$ we have
\[
 s^4(1+sr)^{-\sigma}
 \le Cs^2r^{-2},
 \qquad
 r|Q_s(X,z)|\le Cs,
 \qquad
 \norm{Q_s}_{L^\infty(\Cyl)}\le Cs^2.
\]
Therefore
\[
 |I_1|
 \le Cs^2\int_{\R^m}\frac{\zeta_1^2}{r^2}\dd X,\quad |I_3|
 \le Cs^2\norm{\zeta^\perp}_{L^2(\Cyl)}^2.
\]
Note that $I_2$ is estimated by
\[
 \begin{split}
 |I_2|
 &\le C\int_\Cyl |Q_s|\,|\zeta_1|\,|\zeta^\perp|\\
 &\le Cs
 \left(\int_{\R^m}\frac{\zeta_1^2}{r^2}\dd X\right)^{1/2}
 \norm{\zeta^\perp}_{L^2(\Cyl)}.
 \end{split}
\]
Young's inequality implies
\[
 |I_2|
 \le \frac{\delta_{LE}}4\int_{\R^m}\frac{\zeta_1^2}{r^2}\dd X
 +Cs^2\norm{\zeta^\perp}_{L^2(\Cyl)}^2.
\]

Combining these estimates with \eqref{eq:perp-spectral-coercivity}, \eqref{eq:LE-first-mode-coercivity}, \eqref{eq:q-exp2}, and then choosing $s_0>0$ small enough, gives positive constants $c_0',c_1'$ independent of $s$ such that
\begin{equation}\label{eq:small-coerc}
 \mathfrak q_{Q_s}[\zeta]
 \ge c_0'\int_{\R^m}\frac{\zeta_1^2}{r^2}+c_1'\norm{\zeta^\perp}_{H^1(\Cyl)}^2
 \qquad(0<s<s_0).
\end{equation}
\end{proof}

\begin{proposition}[Small ordered semistable branch]\label{prop:small-branch}
For $m\ge16$, there is $A_0>0$ and a $C^1$ family $A\mapsto Q_A\in\mathscr X_\sigma$ for $0<A<A_0$, such that
\[
 Q_A(0,0)=A,
 \qquad
 \partial_AQ_A>0,
 \qquad
 \mathfrak q_{Q_A}\ge0.
\]
Moreover the normalized operator $\mathcal N_{Q_A}$ is an isomorphism.
\end{proposition}

\begin{proof}
Let $\Psi_s=\partial_sQ_s$.  Differentiating the cylinder equation and the
Dirichlet condition gives
\begin{equation}\label{eq:Psi-jacobi-rig}
 L_{Q_s}\Psi_s=0,\qquad \Psi_s(r,\pm1)=0.
\end{equation}
From \eqref{eq:LS-expansion},
\begin{equation}\label{eq:Psi-center-rig}
 \Psi_s(0,0)=2s+O(s^3)>0
\end{equation}
for $s>0$ small.  The coefficient $D_m$ of the threshold term
$D_mr^{-2}\ph$ is independent of $s$, hence differentiating the affine
tail removes this term and, by the Lyapunov--Schmidt construction,
\begin{equation}\label{eq:Psi-initial-weight-rig}
 \Psi_s=O_{C^{2,\alpha}_{\rm sc}}(r^{-\sigma}),
 \qquad 2<\sigma<\beta_-.
\end{equation}
Since \eqref{eq:Psi-initial-weight-rig} is a polynomial decay estimate,
the Jacobi indicial bootstrap proved above applies to the Jacobi field $\Psi_s$.
Consequently there is a unique slow coefficient $c_-(s)$ and exactly one
of the alternatives
\begin{equation}\label{eq:Psi-slow-rig}
 \Psi_s
 =c_-(s)r^{-\beta_-}\ph
  +o_{C^{2,\alpha}_{\rm sc}}(r^{-\beta_-})
 \qquad(c_-(s)\ne0),
\end{equation}
or
\begin{equation}\label{eq:Psi-fast-rig}
 \Psi_s
 =c_+(s)r^{-\beta_+}\ph
  +o_{C^{2,\alpha}_{\rm sc}}(r^{-\beta_+})
 \qquad(c_-(s)=0),
\end{equation}
with faster decay if $c_+(s)=0$.  By the last assertion of the Jacobi slow/fast analysis above, both expansions remain valid after division
by $\ph$, uniformly up to $z=\pm1$.

We claim $c_-(s)\ne0$.  Suppose $c_-(s)=0$ and choose a radial cutoff
$\eta_R$ equal to one for $r<R$, zero for $r>2R$, and
$|\eta_R'|\le C/R$.  Testing \eqref{eq:Psi-jacobi-rig} by
$\eta_R^2\Psi_s$ and expanding gives the exact identity
\begin{equation}\label{eq:Psi-cutoff-rig}
 \mathfrak q_{Q_s}[\eta_R\Psi_s]
 =\int_\Cyl \Psi_s^2|\nabla\eta_R|^2.
\end{equation}
Because $\beta_+>(m-2)/2$, the right side is
$O(R^{m-2-2\beta_+})\to0$.  Applying \eqref{eq:small-coerc} to the left
side and then monotone convergence yields
\[
 \int_{\R^m}\frac{\psi_1^2}{r^2}=0,
 \qquad \norm{\Psi_s^\perp}_{H^1(\Cyl)}=0,
\]
where $\Psi_s=\psi_1\ph+\Psi_s^\perp$.  Hence $\Psi_s\equiv0$, contrary
to \eqref{eq:Psi-center-rig}.  Therefore $c_-(s)\ne0$.

If $c_-(s)<0$, the quotient form of \eqref{eq:Psi-slow-rig} implies
$\Psi_s<0$ for all $r\ge R_s$, uniformly in $|z|<1$.  Hence
$(\Psi_s)_+$ has compact radial support and belongs to $H^1_0(\Cyl)$.
Testing \eqref{eq:Psi-jacobi-rig} with $(\Psi_s)_+$ gives
$\mathfrak q_{Q_s}[(\Psi_s)_+]=0$.  Estimate
\eqref{eq:small-coerc} forces $(\Psi_s)_+=0$, contradicting
\eqref{eq:Psi-center-rig}.  Thus $c_-(s)>0$.  The negative part
$(\Psi_s)_-$ is now compactly supported; the same test gives
$(\Psi_s)_-=0$.  Hence $\Psi_s\ge0$, and the strong maximum principle implies
\begin{equation}\label{eq:Psi-positive-rig}
 \Psi_s>0\qquad\hbox{in }\Cyl.
\end{equation}
Hopf's lemma gives $-\partial_\nu\Psi_s>0$ on $z=\pm1$.

At the center,
\[
 A(s):=Q_s(0,0)=s^2(1+s^2w_s(0,0))=s^2(1+O(s^2)),
\]
so $A'(s)=2s+O(s^3)>0$.  The one-dimensional inverse function theorem
gives a $C^1$ function $s=s(A)$ for $0<A<A_0:=A(s_0)$, and so $Q_A:=Q_{s(A)}$ is a $C^1$ family in $(0,A_0)$.  Its normalized tangent
\begin{equation}\label{eq:JA-rig}
 J_A:=\partial_AQ_A=\frac{\Psi_s}{A'(s)}
\end{equation}
satisfies
\[
 L_{Q_A}J_A=0,\qquad J_A(0,0)=1,\qquad J_A>0.
\]
For $\zeta\in C_c^1(\Cyl)$ whose support is initially separated from
$z=\pm1$, direct expansion gives the ground-state identity
\begin{equation}\label{eq:groundstate-rig}
 \mathfrak q_{Q_A}[\zeta]
 =\int_\Cyl J_A^2\left|\nabla\left(\frac\zeta{J_A}\right)\right|^2\ge0.
\end{equation}
For a general zero-trace compactly supported $\zeta$, take a boundary
cutoff $\theta_\delta$ and apply \eqref{eq:groundstate-rig} to
$\theta_\delta\zeta$.  Hopf's lemma gives $c(1-|z|)\le J_A\le
C(1-|z|)$ on the compact $X$-projection of the support, and the
one-dimensional Hardy inequality implies
$\theta_\delta\zeta\to\zeta$ in $H^1$.  Letting $\delta\downarrow0$
proves semistability for all admissible tests.

It remains to prove that $\mathcal N_{Q_A}$ is an isomorphism.  The slow
coefficient $c_-(J_A)$ is nonzero by the preceding argument.  If
$H\in\ker\mathcal N_{Q_A}$, set
\[
 \widehat H
 :=H-\frac{c_-(H)}{c_-(J_A)}J_A.
\]
Then $L_{Q_A}\widehat H=0$ and $c_-(\widehat H)=0$, so, by the Jacobi slow/fast alternative proved above,
$\widehat H$ has the fast asymptotic \eqref{eq:jacobi-fast-general}.  Repeating
\eqref{eq:Psi-cutoff-rig} with $\widehat H$ and using again
\eqref{eq:small-coerc} gives $\widehat H=0$.  Thus $H$ is a multiple of
$J_A$; since $H(0,0)=0$ and $J_A(0,0)=1$, this multiple is zero.
Therefore $\ker\mathcal N_{Q_A}=\{0\}$.  By
\cref{prop:fredholm}, $\operatorname{ind}\mathcal N_{Q_A}=0$, so its
cokernel is also zero and the bounded inverse theorem gives a bounded
inverse.  This proves all assertions.
\end{proof}

\subsection{Ordered continuation, large-\texorpdfstring{$m$}{m} compactness, and a uniform Hardy margin}

Continuation of positive elliptic branches from a locally regular
solution, together with the use of the first linearized eigenvalue to
detect loss of regularity, is classical in nonlinear eigenvalue
problems; see Crandall--Rabinowitz \cite{CR75} and Mignot--Puel
\cite{MP}.  The present cylinder problem is not a direct application
of those bounded-domain results: the continuation parameter is the
center height $A$ and the domain is unbounded.  We therefore prove the
required Fredholm, compactness and no-fast-kernel statements explicitly.

\medskip\noindent\textbf{Construction and uniqueness of the maximal ordered family.}
We make the continuation procedure precise.  Let $Q_0$ be any profile
reached from the small branch in Proposition \ref{prop:small-branch} for which the normalized operator
$\mathcal N_{Q_0}$ is an isomorphism, and put $A_0=Q_0(0,0)$.  In a
neighborhood of $(0,A_0)$ in $\mathscr D_\mu\times\R$ consider
\[
 \mathcal F_{Q_0}(H,A)
 :=\left(
 -\Delta(Q_0+H)-\La\bigl((Q_0+H)+(Q_0+H)^2\bigr),
 Q_0(0,0)+H(0,0)-A
 \right).
\]
This map takes values in $\mathscr Y_\mu\times\R$ and is smooth.
Indeed $\|H\|_{C^{0,\alpha}_{\beta_-}}\le C\|H\|_{\mathscr D_\mu}$
and $\mu+2<2\beta_-$, so multiplication satisfies
\[
 \|H_1H_2\|_{\mathscr Y_\mu}
 \le C\|H_1\|_{\mathscr D_\mu}\|H_2\|_{\mathscr D_\mu}.
\]
Its derivative with respect to $H$ at $(0,A_0)$ is exactly
\[
 D_H\mathcal F_{Q_0}(0,A_0)=\mathcal N_{Q_0}.
\]
Hence the Banach-space implicit-function theorem gives an interval
$I_0\ni A_0$ and a unique $C^1$ graph $A\mapsto Q_A$ on $I_0$ such that
$Q_{A_0}=Q_0$, $Q_A(0,0)=A$, and $Q_A$ solves the cylinder equation.
These nearby profiles remain positive.  On the tail the leading term
$D_mr^{-2}\ph$ dominates the perturbation, whose quotient by $\ph$
is $O(r^{-\beta_-})$ by \eqref{eq:massive-graph-boundary-factor}.
On a fixed compact cylinder, positivity and the Hopf boundary bound
for $Q_0$ persist under a sufficiently small local $C^1$ perturbation.
By differentiation,
\begin{equation}\label{eq:normalized-tangent-cont}
 L_{Q_A}J_A=0,
 \qquad J_A(0,0)=1,
 \qquad
 J_A:=\partial_AQ_A=\mathcal N_{Q_A}^{-1}(0,1).
\end{equation}
\begin{definition}[Positive normalized chart]\label{def:positive-normalized-chart}
A $C^1$ family $A\mapsto Q_A$ of threshold-cylinder profiles on an
interval $I$ is a \emph{positive normalized chart} if
\[
 Q_A(0,0)=A,\qquad \mathcal N_{Q_A}\text{ is an isomorphism},\qquad
 J_A:=\partial_AQ_A>0
\]
for every $A\in I$.  Thus the center height is the local parameter and
the branch tangent is the positive normalized Jacobi field satisfying
$J_A(0,0)=1$.
\end{definition}
After shrinking the small-amplitude chart supplied by
\cref{prop:small-branch}, it is a positive normalized chart.

Let $\mathscr I$ be the collection of numbers $b>0$ for which the
small-amplitude graph extends to $(0,b)$ by a finite chain of overlapping
positive normalized charts. Such an extension is unique by the implicit-function theorem.    

Define
\[
 A^*:=\sup\mathscr I\in(0,\infty].
\]
For every $A<A^*$ one may choose $b\in\mathscr I$ with $A<b$ and define
$Q_A$ by the unique extension on $(0,b)$.  This gives the maximal connected
ordered family issuing from the small branch,
\begin{equation}\label{eq:Bm}
 \cB_m=\{Q_A:0<A<A^*\},
 \qquad Q_A(0,0)=A,
 \qquad \partial_AQ_A=J_A>0.
\end{equation}
\begin{definition}[Maximal ordered branch and regular endpoint]\label{def:ordered-branch}
The family
\[
 \cB_m=\{Q_A:0<A<A^*\},\qquad
 Q_A(0,0)=A,\qquad \partial_AQ_A>0,
\]
obtained by chaining overlapping positive normalized charts is the
\emph{maximal ordered branch issuing from the small branch}.  Its finite
endpoint $A^*<\infty$ is \emph{regular} if there exists
$Q_*\in\mathscr X_\sigma$ such that
\[
 Q_A\longrightarrow Q_*\quad\text{in }\mathscr X_\sigma
 \quad\text{as }A\uparrow A^*.
\]
Otherwise a finite endpoint is called \emph{nonregular}.  When
$A^*=\infty$, the monotone limit is called the \emph{infinite-height
endpoint}; it is \emph{singular} if it is not locally bounded.
\end{definition}
In particular, arguing as in \eqref{eq:groundstate-rig}, every profile
$Q_A\in\cB_m$ is semistable.  The following continuation criterion is
the precise form needed later.

\begin{proposition}[Only a fast Jacobi field can stop a regular ordered endpoint]\label{prop:stopping}
Let $m\ge16$. If $A^*<\infty$ is a regular endpoint of \eqref{eq:Bm} and the ordered continuation fails at $A^*$, then $L_{Q_*}$ possesses a nonzero radial-even fast Dirichlet Jacobi field.
\end{proposition}

\begin{proof}
Assume that $L_{Q_*}$ has no nonzero radial-even fast Dirichlet Jacobi
field.  We show that $Q_*$ is contained in a positive normalized chart,
which contradicts maximality of \eqref{eq:Bm}.

Choose $A_j\uparrow A^*$ and set $J_j=\partial_AQ_{A_j}$.  Then
\begin{equation}\label{eq:Jj-rig}
 L_{Q_{A_j}}J_j=0,\qquad J_j(0,0)=1,\qquad J_j>0.
\end{equation}
Regular convergence implies that $Q_*$ is a smooth solution in
$\mathscr X_\sigma$.  Moreover, for every compactly supported
zero-trace test function $\zeta$,
\[
 \mathfrak q_{Q_*}[\zeta]
 =\lim_{j\to\infty}\mathfrak q_{Q_{A_j}}[\zeta]\ge0,
\]
so $Q_*$ is semistable and \cref{prop:fredholm} applies to
$\mathcal N_{Q_*}$.

\smallskip\noindent
\emph{Step 1: compactness of the normalized tangents.}
Fix $K\Subset\overline\Cyl$.  Since the coefficients of $L_{Q_{A_j}}$ converge
locally in $C^{0,\alpha}$ and $J_j(0,0)=1$, Harnack chains give uniform
upper and lower bounds for $J_j$ on compact subsets away from
$z=\pm1$.  Boundary Harnack and boundary Schauder estimates give the
corresponding bounds up to $z=\pm1$.  Hence
\begin{equation}\label{eq:Jj-local-rig}
 \norm{J_j}_{C^{2,\alpha}(K)}\le C_K.
\end{equation}

It remains to control the tail uniformly.  Fix
\[
 \beta_-<\mu''<\mu'<\mu
\]
and choose $R$ in the normal-operator region.  On $\{r\ge R\}$ write
\[
 J_j=j_j(r)\ph+J_j^\perp
\]
and set
\[
 X_{\mu'}^R
 :=\spanop\{r^{-\beta_-}\ph\}\oplus
   C^{2,\alpha}_{\mu',D}(\{r\ge R\}),
 \qquad
 \|c r^{-\beta_-}\ph+H_f\|_{X_{\mu'}^R}
 :=|c|+\|H_f\|_{C^{2,\alpha}_{\mu'}}.
\]
From the normal splitting \eqref{eq:normal-splitting-rig},
\[
 L_{Q_{A_j}}=L_{\rm nor}+\mathcal K_j,
 \qquad
 L_{\rm nor}
 =\begin{pmatrix}E&B\\0&\mathscr L_\perp\end{pmatrix}
 \quad (r\ge R).
\]
Regular convergence in $\mathscr X_\sigma$ and the strict gains in the
weights used in the proof of \cref{prop:fredholm} imply, uniformly in
$j$,
\begin{equation}\label{eq:Kj-small-rig}
 \|P^\perp\mathcal K_jH\|_{C^{2,\alpha}_{\mu'}(r\ge R)}
 +\|P_1\mathcal K_jH\|_{C^{0,\alpha}_{\mu'+2}(r\ge R)}
 \le CR^{-\vartheta}\|H\|_{X_{\mu'}^R}
\end{equation}
for some $\vartheta>0$.

Projecting the Jacobi equation \eqref{eq:Jj-rig} onto $\ph^\perp$ gives
\[
 \mathscr L_\perp J_j^\perp=-P^\perp\mathcal K_jJ_j.
\]
The same-order estimate \eqref{eq:massive-same-order},
localized to the exterior half-cylinder by its exponentially decaying
Green kernel, yields
\begin{equation}\label{eq:Jperp-tail-direct-rig}
 \|J_j^\perp\|_{C^{2,\alpha}_{\mu'}(r\ge R)}
 \le C_R\|J_j\|_{C^{2,\alpha}(\{R<r<2R\})}
     +CR^{-\vartheta}\|J_j\|_{X_{\mu'}^R}.
\end{equation}

Define 
\[
 k_{1,j}(r):=
 \int_{-1}^1(\mathcal K_jJ_j)(r,z)\ph(z)\dd z.
\]
The first-mode equation is therefore
\[
 Ej_j=f_j,
 \qquad
 f_j:=-BJ_j^\perp-k_{1,j}.
\]
Since $B$ contains the factor $r^{-2}$, \eqref{eq:Kj-small-rig} and
\eqref{eq:Jperp-tail-direct-rig} give
\begin{equation}\label{eq:firstmode-source-tail-rig}
 \|f_j\|_{C^{0,\alpha}_{\mu'+2}(r\ge R)}
 \le C_R\|J_j\|_{C^{2,\alpha}(\{R<r<2R\})}
     +CR^{-\vartheta}\|J_j\|_{X_{\mu'}^R}.
\end{equation}
Because $\beta_-<\mu'<\beta_+$, the augmented Euler tail estimate
\cref{lem:euler-tail-augmented}, applied with the Cauchy data
$j_j(R),j_j'(R)$, gives a unique coefficient $c_-(J_j)$ such that
\begin{align*}
 |c_-(J_j)|
 +\|j_j-c_-(J_j)r^{-\beta_-}\|_{C^{2,\alpha}_{\mu'}(r\ge R)}
 \le{}& C_R\|J_j\|_{C^{2,\alpha}(\{R<r<2R\})}\\
 &+CR^{-\vartheta}\|J_j\|_{X_{\mu'}^R}.
\end{align*}
Combining this estimate with \eqref{eq:Jperp-tail-direct-rig} yields
\[
 \|J_j\|_{X_{\mu'}^R}
 \le C_R\|J_j\|_{C^{2,\alpha}(\{R<r<2R\})}
     +CR^{-\vartheta}\|J_j\|_{X_{\mu'}^R}.
\]
Choose $R$ so large that $CR^{-\vartheta}\le\frac12$ and absorb the
last term.  We obtain
\begin{equation}\label{eq:tangent-tail-uniform-rig}
 |c_-(J_j)|+
 \norm{J_j-c_-(J_j)r^{-\beta_-}\ph}_{C^{2,\alpha}_{\mu'}(r\ge R)}
 \le
 2C_R\norm{J_j}_{C^{2,\alpha}(\{R<r<2R\})}.
\end{equation}
The coefficient $c_-(J_j)$ is unique because $\beta_-<\mu'$, hence
$r^{-\beta_-}\ph\notin C^{2,\alpha}_{\mu'}$.  By
\eqref{eq:Jj-local-rig}, the right-hand side is uniformly bounded.

After passing to a subsequence,
$c_-(J_j)\to c_*$, and weighted Arzel\`a--Ascoli gives, for every
$\alpha'<\alpha$,
\begin{equation}\label{eq:Jstar-conv-rig}
 J_j\to J_*\quad\hbox{locally in }C^{2,\alpha'},
 \qquad
 J_j-c_-(J_j)r^{-\beta_-}\ph
 \to J_*-c_*r^{-\beta_-}\ph
 \quad\hbox{in }C^{2,\alpha'}_{\mu''}(r\ge R).
\end{equation}
Passing to the limit in \eqref{eq:Jj-rig} yields
\begin{equation}\label{eq:Jstar-rig}
 L_{Q_*}J_*=0,\qquad J_*(0,0)=1,\qquad J_*\ge0.
\end{equation}
Thus the strong maximum principle and Hopf's lemma give $J_*>0$ in the
open cylinder and $-\partial_\nu J_*>0$ on $z=\pm1$.  The Jacobi
indicial bootstrap proved above upgrades the tail in
\eqref{eq:Jstar-conv-rig} to the original domain $\mathscr D_\mu$ and
identifies $c_-(J_*)=c_*$.  If $c_*=0$, then $J_*$ is a nonzero fast
Jacobi field, contrary to the hypothesis.  Hence $c_*\ne0$; the quotient
tail expansion of the positive field $J_*$ gives in fact $c_*>0$.

\smallskip\noindent
\emph{Step 2: invertibility at the endpoint.}
Let $H\in\ker\mathcal N_{Q_*}$.  Since $J_*\in\mathscr D_\mu$, set
\[
 \widetilde H
 :=H-\frac{c_-(H)}{c_*}J_*.
\]
Then $L_{Q_*}\widetilde H=0$ and $c_-(\widetilde H)=0$.  Hence
$\widetilde H$ is fast and, by assumption, $\widetilde H=0$.
Therefore $H=cJ_*$.  Since $H(0,0)=0$ and $J_*(0,0)=1$, one has $c=0$.
Thus $\ker\mathcal N_{Q_*}=\{0\}$.  By semistability of $Q_*$ and
\cref{prop:fredholm},
$\operatorname{ind}\mathcal N_{Q_*}=0$; hence $\mathcal N_{Q_*}$ is an
isomorphism.

\smallskip\noindent
\emph{Step 3: continuation and preservation of positivity.}
Apply the Banach-space implicit function theorem to
\[
 \mathcal F(Q,A)
 =\bigl(-\Delta Q-\La(Q+Q^2),\,Q(0,0)-A\bigr)
\]
at $(Q_*,A^*)$.  It gives a unique $C^1$ graph
$A\mapsto\widehat Q_A$ for $|A-A^*|<\delta$.  Its normalized tangent
\[
 \widehat J_A:=\partial_A\widehat Q_A
 =\mathcal N_{\widehat Q_A}^{-1}(0,1)
\]
depends continuously on $A$ in $\mathscr D_\mu$, and
$\widehat J_{A^*}=J_*$.  In particular $c_-(\widehat J_A)\to c_*>0$.
Estimate \eqref{eq:massive-graph-boundary-factor}, the scalar part of
the graph norm, and $\mu>\beta_-$ give, uniformly for $A$ near $A^*$,
\[
 \frac{\widehat J_A(r,z)}{\ph(z)}
 =c_-(\widehat J_A)r^{-\beta_-}+O(r^{-\mu}),
 \qquad r\to\infty.
\]
Hence $\widehat J_A>0$ for all $r\ge R_1$, with $R_1$ independent of
$A$ near $A^*$.  On the compact cylinder $\{r\le R_1\}$, positivity of
$J_*$ together with Hopf's lemma gives
$J_*/\ph\ge c_{R_1}>0$; continuity of
$A\mapsto\widehat J_A$ in the local $C^1$ topology preserves this lower
bound after decreasing $\delta$.  Thus $\widehat J_A>0$ throughout the
cylinder for $|A-A^*|<\delta$.

It remains to show that the old branch enters this implicit-function
chart in the stated domain, rather than merely in a weaker weight.
Put $H_A=Q_A-Q_*$.  Regular convergence gives
\[
 \eta_A:=\|H_A\|_{C^{2,\alpha}_\sigma}\longrightarrow0.
\]
The exact difference equation is
\[
 L_{\overline Q_A}H_A=0,
 \qquad \overline Q_A=(Q_A+Q_*)/2.
\]
By \eqref{eq:averaged-coefficient-domain}, each $H_A$ belongs to
$\mathscr D_\mu$.  This is only a membership assertion; the following
estimate proves convergence in its norm.  The endpoint operator,
whose invertibility was proved in Step~2, satisfies
\[
 L_{Q_*}H_A=\La H_A^2,
 \qquad H_A(0,0)=A-A^*.
\]
Thus
\begin{equation}\label{eq:endpoint-graph-absorption}
 \|H_A\|_{\mathscr D_\mu}
 \le C\left(\|H_A^2\|_{\mathscr Y_\mu}+|A-A^*|\right).
\end{equation}
The definition of $\mathscr D_\mu$ gives its continuous embedding into
$C^{0,\alpha}_{\beta_-}$.  The weighted product inequality and
$\mu+2<\sigma+\beta_-$ give
\begin{equation}\label{eq:endpoint-mixed-product}
 \|H_A^2\|_{\mathscr Y_\mu}
 \le C\|H_A\|_{C^{2,\alpha}_\sigma}
           \|H_A\|_{\mathscr D_\mu}.
\end{equation}
Indeed on each annulus the product has weight $\sigma+\beta_-$;
projection onto $\ph$ requires only weight $\mu+2$, and projection
onto $\ph^\perp$ only weight $\mu$.  The compact part follows from
the local graph estimate.  Since $C\eta_A<1/2$ for $A$ close to
$A^*$, \eqref{eq:endpoint-graph-absorption} absorbs the product and yields
\begin{equation}\label{eq:endpoint-strong-graph-convergence}
 \|Q_A-Q_*\|_{\mathscr D_\mu}
 \le 2C|A-A^*|\longrightarrow0.
\end{equation}
Local uniqueness in the implicit-function theorem now identifies the
old and new graphs.  No upgrade of a weak convergence statement is
being inferred solely from the indicial membership argument.
The chart is positive and extends \eqref{eq:Bm} to $A>A^*$, a
contradiction.  Hence failure of ordered continuation at a regular
endpoint can occur only if $L_{Q_*}$ has a nonzero radial-even fast
Dirichlet Jacobi field.
\end{proof}

Define
\begin{equation}\label{eq:Sm}
 S_m:=\sup\{r^2Q(r,0):\ Q\in\cB_m,\ r>0\}\in(0,\infty].
\end{equation}
The next theorem is the central large-dimensional estimate.

\begin{theorem}[Large-dimensional compact-annulus estimate]\label{thm:Sm}
As $m\to\infty$,
\begin{equation}\label{eq:Smom2}
 \frac{S_m}{m^2}\longrightarrow0.
\end{equation}
In particular $S_m<\infty$ for all sufficiently large $m$.
\end{theorem}

\begin{proof}
Assume that \eqref{eq:Smom2} fails.  Then there exist $\kappa>0$,
integers $m_j\to\infty$, and profiles $Q_j\in\cB_{m_j}$ such that
\begin{equation}\label{eq:Sm-contr-assump-rig}
 \sup_{r>0}r^2Q_j(r,0)>2\kappa m_j^2.
\end{equation}
For each fixed profile, \cref{lem:universal-tail} gives
$r^2Q_j(r,0)\to D_{m_j}=O(m_j)$ as $r\to\infty$, while radial regularity
gives $r^2Q_j(r,0)\to0$ as $r\downarrow0$.  Hence, for all large $j$,
there is a first radius $r_j>0$ satisfying
\begin{equation}\label{eq:firstcontactSm-rig}
 r_j^2Q_j(r_j,0)=\kappa m_j^2.
\end{equation}
Set
\[
 M_j:=Q_j(r_j,0),\qquad
 p_j(r):=\int_{-1}^1Q_j(r,z)\ph(z)\dd z.
\]
By the choice of $r_j$ and \cref{cor:radial-mono},
\begin{align}
 r^2Q_j(r,0)&<\kappa m_j^2 &&(0<r<r_j),\label{eq:firstleft-rig}\\
 Q_j(r,0)&\le M_j &&(r\ge r_j).\label{eq:firstright-rig}
\end{align}
Furthermore, \cref{lem:scalar-separation} gives
\begin{equation}\label{eq:pMjzero-rig}
 0<\frac{p_j(r_j)}{M_j}
 <\frac{D_{m_j}}{r_j^2M_j}
 =\frac{D_{m_j}}{\kappa m_j^2}
 \longrightarrow0.
\end{equation}
After passing to a subsequence, exactly one of the alternatives
$M_j\to0$, $M_j\to M_\infty\in(0,\infty)$, or $M_j\to\infty$ holds.

\medskip\noindent
\emph{Case 1: $M_j\to0$.}
Define, for $x>-r_j$,
\[
 U_j(x,z):=\frac{Q_j(r_j+x,z)}{M_j}.
\]
For $x\le0$, \eqref{eq:firstleft-rig} and
\cref{lem:transverse} imply
\[
 U_j(x,z)
 \le\frac{r_j^2}{(r_j+x)^2}\ph(z),
\]
whereas for $x\ge0$, \eqref{eq:firstright-rig} gives
$U_j(x,z)\le\ph(z)$.  Since
\begin{equation}\label{eq:case1-drift-rig}
 \frac{m_j}{r_j}
 =\frac{m_j\sqrt{M_j}}{\sqrt{\kappa}\,m_j}
 =\frac{\sqrt{M_j}}{\sqrt\kappa}\to0,
\end{equation}
we have $r_j\to\infty$, and on every compact subset of
$\R\times[-1,1]$,
\[
 0\le U_j\le1+o(1).
\]
The equation is
\begin{equation}\label{eq:case1-eq-rig}
 -(U_j)_{xx}-\frac{m_j-1}{r_j+x}(U_j)_x-(U_j)_{zz}
 =\La(U_j+M_jU_j^2).
\end{equation}
On a fixed compact set, the drift coefficient in
\eqref{eq:case1-eq-rig} tends uniformly to zero by
\eqref{eq:case1-drift-rig}; the right side is uniformly bounded.  Interior
and boundary $W^{2,p}$ estimates, followed by Schauder estimates, give
precompactness in $C^{2,\alpha'}_{\rm loc}$ for every $\alpha'<\alpha$.
Thus a subsequence converges to $U$ satisfying
\begin{equation}\label{eq:case1-limit-rig}
 -U_{xx}-U_{zz}=\La U\quad\hbox{in }\R\times(-1,1),
 \qquad U(x,\pm1)=0,
 \qquad 0\le U\le1,
 \qquad U(0,0)=1.
\end{equation}
Let $\{\ph_k\}_{k\ge1}$ be an orthonormal Dirichlet eigenbasis of
$-\partial_{zz}$, with $\ph_1=\ph$ and eigenvalues
$\lambda_1=\La<\lambda_2\le\cdots$.  Put
$c_k(x)=\int_{-1}^1U(x,z)\ph_k(z)\dd z$.  From
\eqref{eq:case1-limit-rig},
\[
 -c_1''=0,
 \qquad
 -c_k''+(\lambda_k-\La)c_k=0\quad(k\ge2).
\]
All $c_k$ are bounded on $\R$.  Hence $c_1$ is constant, and for $k\ge2$
the representation
$c_k=A_ke^{\sqrt{\lambda_k-\La}\,x}
+B_ke^{-\sqrt{\lambda_k-\La}\,x}$ forces $A_k=B_k=0$.
Thus $U=c\ph$.  Since $U(0,0)=1$ and $\ph(0)=1$, $c=1$.  Therefore
\[
 \frac{p_j(r_j)}{M_j}
 =\int_{-1}^1U_j(0,z)\ph(z)\dd z
 \longrightarrow\int_{-1}^1\ph^2\dd z=1,
\]
contradicting \eqref{eq:pMjzero-rig}.

\medskip\noindent
\emph{Case 2: $M_j\to M_\infty\in(0,\infty)$.}
Use the same $U_j$.  From $r_j^2M_j=\kappa m_j^2$,
\begin{equation}\label{eq:case2-drift-rig}
 \frac{m_j-1}{r_j}
 =\frac{m_j-1}{m_j}\sqrt{\frac{M_j}{\kappa}}
 \longrightarrow b:=\sqrt{\frac{M_\infty}{\kappa}}.
\end{equation}
In particular $r_j\to\infty$.  The preceding pointwise bounds again give
local uniform boundedness, and compactness applied to
\eqref{eq:case1-eq-rig} yields
\begin{equation}\label{eq:case2-limit-rig}
 -U_{xx}-U_{zz}-bU_x
 =\La(U+M_\infty U^2),
 \quad U(x,\pm1)=0,
 \quad U(0,0)=1,
 \quad U\ge0.
\end{equation}
The strong maximum principle gives $U>0$ in the open strip.  Hence
\[
 c_*:=\int_{-1}^1U(0,z)\ph(z)\dd z>0.
\]
Local $C^0$ convergence at $x=0$ implies
$p_j(r_j)/M_j\to c_*>0$, contradicting
\eqref{eq:pMjzero-rig}.

\medskip\noindent
\emph{Case 3: $M_j\to\infty$.}
Put
\begin{equation}\label{eq:case3scale-rig}
 U_j(x,y)
 :=M_j^{-1}Q_j\!
 \left(r_j+M_j^{-1/2}x,\,M_j^{-1/2}y\right).
\end{equation}
The rescaled domain is
\[
 x>-r_j\sqrt{M_j},\qquad |y|<\sqrt{M_j},
\]
and
\begin{equation}\label{eq:Rjcase3-rig}
 r_j\sqrt{M_j}=\sqrt\kappa\,m_j\to\infty.
\end{equation}
For $x\le0$, \eqref{eq:firstleft-rig} gives
\[
 U_j(x,y)
 \le\left(\frac{r_j\sqrt{M_j}}
 {r_j\sqrt{M_j}+x}\right)^2,
\]
and for $x\ge0$, \eqref{eq:firstright-rig} gives $U_j\le1$.
Also \cref{cor:radial-mono} gives $(U_j)_x\le0$.  Thus on each compact
subset of $\R^2$,
\begin{equation}\label{eq:case3-local-bound-rig}
 0\le U_j\le1+o(1),\qquad (U_j)_x\le0.
\end{equation}
The rescaled equation is
\begin{equation}\label{eq:case3-eq-rig}
 -(U_j)_{xx}-(U_j)_{yy}
 -\frac{m_j-1}{r_j\sqrt{M_j}+x}(U_j)_x
 =\La(M_j^{-1}U_j+U_j^2).
\end{equation}
By \eqref{eq:Rjcase3-rig},
\[
 \frac{m_j-1}{r_j\sqrt{M_j}+x}
 \longrightarrow \kappa^{-1/2}
\]
uniformly on compact sets.  Elliptic compactness therefore gives a limit
$U\in C^2(\R^2)$ satisfying
\begin{equation}\label{eq:case3-limit-rig}
 -U_{xx}-U_{yy}-\kappa^{-1/2}U_x=\La U^2,
 \qquad 0\le U\le1,
 \qquad U(0,0)=1,
 \qquad U_x\le0.
\end{equation}
For each $y$, the monotonicity in $x$ defines
\[
 U_-(y):=\lim_{x\to-\infty}U(x,y),\qquad
 U_+(y):=\lim_{x\to+\infty}U(x,y),
\]
with $0\le U_+\le U_-\le1$.  Let $x_n\to+\infty$.  The translates
$U_n(x,y)=U(x+x_n,y)$ satisfy the same equation
\eqref{eq:case3-limit-rig} and are uniformly bounded in
$C^{2,\alpha}$ on compact sets.  Any convergent subsequence has pointwise
limit $U_+(y)$ by monotonicity, hence the limit is independent of $x$.
Passing to the equation gives
\begin{equation}\label{eq:Uplus-rig}
 -U_+''=\La U_+^2\quad\hbox{on }\R.
\end{equation}
The same argument at $-\infty$ gives
$-U_-''=\La U_-^2$.  If $v\ge0$ is a bounded $C^2$ solution of
$-v''=\La v^2$ on $\R$, then $v''\le0$, so $v'$ is nonincreasing.  If
$v'(y_0)>0$, then $v(y)\to-\infty$ as $y\to-\infty$; if
$v'(y_0)<0$, then $v(y)\to-\infty$ as $y\to+\infty$.  Hence $v'\equiv0$,
and the equation gives $v\equiv0$.  Therefore
$U_-=U_+=0$.  Since $U_x\le0$,
\[
 0=U_-(y)\ge U(x,y)\ge U_+(y)=0,
\]
so $U\equiv0$, contradicting $U(0,0)=1$.

All three alternatives are impossible.  Therefore
$S_m/m^2\to0$.  In particular, if $S_m=\infty$ for infinitely many
$m\to\infty$, the same construction with any fixed $\kappa>0$ would give
the preceding contradiction; hence $S_m<\infty$ for all sufficiently
large $m$.
\end{proof}

Choose once and for all $m_0\ge17$ so large that for every $m\ge m_0$,
\begin{equation}\label{eq:gammam}
 \gamma_m:=\frac{(m-2)^2}{4}-2\La S_m>0.
\end{equation}
From now on $m\ge m_0$ is fixed.

\begin{proposition}[Uniform Hardy coercivity]\label{prop:hardy-coerc}
For every $Q\in\cB_m$ and every test function $\zeta\in H_0^1(\Cyl)$,
\begin{equation}\label{eq:uniform-coerc}
 \mathfrak q_Q[\zeta]
 \ge\gamma_m\int_\Cyl\frac{\zeta^2}{r^2}.
\end{equation}
\end{proposition}

\begin{proof}
By \cref{lem:transverse} and \eqref{eq:Sm},
$0<Q(r,z)\le S_mr^{-2}$. For a.e. fixed $X$, the first Dirichlet eigenvalue in $z$ gives
\[
 \int_{-1}^1(|\zeta_z|^2-\La\zeta^2)\dd z\ge0.
\]
For a.e. fixed $z$, the sharp Hardy inequality in $\R^m$ gives
\[
 \int_{\R^m}|\nabla_X\zeta|^2\dd X
 \ge\frac{(m-2)^2}{4}\int_{\R^m}\frac{\zeta^2}{r^2}\dd X.
\]
Combining these inequalities with the potential bound proves \eqref{eq:uniform-coerc}.
\end{proof}

\begin{corollary}[No fast Jacobi fields]\label{cor:no-fast}
For $m\ge m_0$, $L_Q$ has no nonzero Dirichlet Jacobi field in the fast class. In particular it has no nonzero $H^1_0(\Cyl)$ Jacobi field. The radial-even Jacobi space on the regular ordered branch is exactly $\spanop\{\partial_AQ_A\}$.
\end{corollary}

\begin{proof}
If $J\in H^1_0(\Cyl)$ satisfies $L_QJ=0$, let $\eta_R(r)$ equal one on $r<R$, vanish on $r>2R$, and satisfy $|\nabla\eta_R|\le C/R$. Testing by $\eta_R^2J$ yields the exact identity
\[
 \mathfrak q_Q[\eta_RJ]=\int_\Cyl J^2|\nabla\eta_R|^2.
\]
By \eqref{eq:uniform-coerc},
\[
 \gamma_m\int\frac{\eta_R^2J^2}{r^2}
 \le \frac{C}{R^2}\int_{R<r<2R}J^2\longrightarrow0.
\]
Monotone convergence gives $J=0$. A fast regular radial-even field belongs to $H^1_0$ because $m\ge17$ and $\beta_+>m/2$, so it is zero.

The branch tangent $\partial_A Q_A$ has nonzero slow coefficient; otherwise it would be fast by \cref{lem:jacobi-slow-fast}, and hence $\partial_A Q_A=0$, contradicting $(\partial_AQ_A)(0,0)=1$. For any other radial-even Jacobi field, subtract its slow coefficient times $\partial_A Q_A$. The difference is fast and therefore zero.
\end{proof}

\begin{corollary}[Global ordered branch]\label{cor:global-branch}
For $m\ge m_0$, the family \eqref{eq:Bm} extends to all $A>0$. Moreover
\begin{equation}\label{eq:global-properties}
 \partial_AQ_A>0,
 \qquad (Q_A)_r<0,
 \qquad \partial_z(Q_A/\ph)<0,
 \qquad \mathfrak q_{Q_A}\ge0.
\end{equation}
\end{corollary}

\begin{proof}
Suppose $A^*<\infty$.  Along the maximal ordered branch,
\eqref{eq:Bm}, \cref{cor:radial-mono}, and \cref{lem:transverse} give
\[
 \partial_AQ_A>0,\qquad
 0<Q_A(r,z)\le Q_A(r,0)\le Q_A(0,0)=A\le A^*.
\]
The uniform $L^\infty$ bound and the equation first give, on every
compact cylinder subdomain (with the usual Dirichlet boundary charts when
it meets $z=\pm1$), a uniform $W^{2,p}$ bound for every finite $p$.
Choosing $p$ large and bootstrapping the equation gives a uniform
$C^{2,\alpha_1}$ bound for some $\alpha_1\in(\alpha,1)$.  Monotonicity in
$A$ defines a pointwise limit $Q_*$.  The compact embedding
$C^{2,\alpha_1}\Subset C^{2,\alpha}$, together with uniqueness of the
pointwise limit, therefore yields
\begin{equation}\label{eq:QA-local-endpoint-rig}
 Q_A\longrightarrow Q_*
 \qquad\hbox{in }C^{2,\alpha}_{\rm loc}(\Cyl)
 \quad\hbox{as }A\uparrow A^*.
\end{equation}

It remains only to upgrade this local convergence to the
space $\mathscr X_\sigma$.  Choose once and for all
\[
 \sigma<\tau<\min\{4,\beta_-\}.
\]
Since $0<Q_A\le A^*$, the family has a uniform
$L^\infty(\Cyl)$ bound.
Fix $A_0\in(0,A^*)$.  Order gives $Q_A\ge Q_{A_0}$ for
$A_0\le A<A^*$.  Apply the ordered lower-barrier argument
\eqref{eq:ordered-tail-sandwich} with $M=A^*$ and the fixed profile
$Q_{A_0}$, followed by the scaled ODE estimate in
\cref{rem:tail-uniformity-correct}.  These bounds require no prior
weighted compactness of the half-open branch and give
\begin{equation}\label{eq:QA-uniform-tau-rig}
 \sup_{A_0\le A<A^*}\sup_{R\ge R_0}
 R^\tau
 \bigl\|Q_A-D_mr^{-2}\ph\bigr\|_{C^{2,\alpha}_{\rm sc}(A_R)}
 \le C_{A^*}.
\end{equation}
Passing to the local limit in each fixed annulus shows that the same
bound holds for $Q_*$.  Consequently, for every $R\ge R_0$,
\begin{equation}\label{eq:QA-tail-sigma-convergence-rig}
 \sup_{A_0\le A<A^*}
 \bigl\|Q_A-Q_*\bigr\|_{C^{2,\alpha}_\sigma(\{r\ge R\})}
 \le C_{A^*}R^{\sigma-\tau}.
\end{equation}
Since $\tau>\sigma$, the right-hand side tends to zero as $R\to\infty$;
on the fixed region $\{r\le2R\}$ the convergence follows from
\eqref{eq:QA-local-endpoint-rig}.  Hence
\[
 Q_A\longrightarrow Q_*
 \qquad\hbox{in }\mathscr X_\sigma,
\]
so $A^*$ is a regular endpoint.  By \cref{prop:stopping}, failure to
continue at $Q_*$ would produce a nonzero radial-even fast Jacobi field,
contradicting \cref{cor:no-fast}.  Therefore $A^*=\infty$.
\end{proof}

\subsection{The singular cylinder endpoint}

The next estimate uses the standard semistability strategy of
testing the equation and the stability inequality with related
nonlinear cutoffs.  This general philosophy is central in the
regularity theory of semistable solutions; compare Cabr\'e--Capella
\cite{CC}, Cabr\'e \cite{Cabre2010,CabreSurvey}, Dupaigne
\cite{DupaigneBook}, and Cabr\'e--Figalli--Ros-Oton--Serra \cite{CFRS}.
The particular cubic cancellation below is tailored to the present
quadratic nonlinearity and is proved directly.

\begin{lemma}[Local stable-energy estimate]\label{lem:local-energy}
Let $Q$ be any positive semistable solution of \eqref{eq:cylinder}. For every compact $K\Subset\overline\Cyl$ there is $C_K$, independent of $Q(0,0)$, such that
\begin{equation}\label{eq:local-L3H1}
 \int_KQ^3+\int_K|\nabla Q|^2\le C_K.
\end{equation}
\end{lemma}

\begin{proof}
Choose compactly supported smooth cutoffs $\chi,\eta$ on
$\overline\Cyl$, with $0\le\chi,\eta\le1$, $\chi=1$ on a neighbourhood
of $K$, and $\eta=1$ on a neighbourhood of $\operatorname{supp}\chi$.
The cutoffs need not vanish at $z=\pm1$, because $Q$ has zero trace.
Test semistability with $Q\eta^3$ and the equation with $Q\eta^6$.
Cancellation of the mixed gradient terms gives, equivalently,
\begin{equation}\label{eq:stable-energy-id}
 \La\int_\Cyl Q^3\eta^6
 \le 9\int_\Cyl Q^2\eta^4|\nabla\eta|^2.
\end{equation}
Young's inequality, with its coefficient chosen after the factor $9$,
gives
\[
 9Q^2\eta^4|\nabla\eta|^2
 \le\frac{\La}{2}Q^3\eta^6+C_{\La}|\nabla\eta|^6.
\]
Absorption yields
\[
 \int_{\operatorname{supp}\chi}Q^3
 \le\int_\Cyl Q^3\eta^6
 \le C_{\La}\int_\Cyl|\nabla\eta|^6.
\]
Next test the equation with $Q\chi^2$.  Using
$2Q\chi|\nabla Q||\nabla\chi|
\le\tfrac12\chi^2|\nabla Q|^2+2Q^2|\nabla\chi|^2$, one obtains
\[
 \frac12\int_\Cyl\chi^2|\nabla Q|^2
 \le\La\int_\Cyl\chi^2(Q^2+Q^3)
       +2\int_\Cyl Q^2|\nabla\chi|^2.
\]
The established $L^3$ bound on $\operatorname{supp}\chi$, H\"older's
inequality, and the finite measure of this fixed support bound the
right-hand side independently of $Q(0,0)$.  This proves both terms in
\eqref{eq:local-L3H1}, also when $K$ meets the transverse boundary.
\end{proof}

\begin{theorem}[Singular semistable cylinder endpoint]\label{thm:cylinder-singular}
As $A\to\infty$, $Q_A$ increases pointwise to an unbounded weak solution
\[
 Q_\infty\in H^1_{\rm loc}(\Cyl)\cap L^3_{\rm loc}(\Cyl)
\]
of \eqref{eq:cylinder}. It is semistable and satisfies
\begin{equation}\label{eq:Qinf-bound}
 0<Q_\infty(r,z)\le Q_\infty(r,0)\ph(z)\le\frac{S_m}{r^2}\ph(z)
 \qquad(r>0).
\end{equation}
\end{theorem}

\begin{proof}
Monotonicity in $A$ defines the pointwise limit. By \cref{lem:local-energy}, $Q_A$ is uniformly bounded in $H^1$ and $L^3$ on compact sets. Uniform integrability and monotone convergence give strong $L^p_{\rm loc}$ convergence for every $p<3$; in particular $Q_A^2\to Q_\infty^2$ in $L^1_{\rm loc}$. Hence both the weak equation and the semistability inequality pass to the limit.

The bound \eqref{eq:Qinf-bound} is the monotone limit of the corresponding bounds for $Q_A$. 

The limit $Q_\infty$ is unbounded. Given $M>0$, choose $A>2M$. Since $Q_A(0,0)=A$, continuity gives a set of positive measure on which $Q_A>M$. Since $Q_\infty\ge Q_A$, the essential supremum of $Q_\infty$ exceeds $M$. As $M$ is arbitrary, $Q_\infty\notin L^\infty_{\rm loc}$.
\end{proof}

\section{Constant-boundary branches and a finite-domain annulus estimate}
\label{sec:finite-branch}

We first extract the finite-domain comparison consequence of the
cylinder construction.  We then establish an estimate for all bounded
solutions with the symmetries and monotonicities of the minimal branch.
The dimension is allowed to tend to infinity only in this second step.

\subsection{The boundary parameter and the minimal branch}

Let $S_m$ be as in \eqref{eq:Sm}, and put
\begin{equation}\label{eq:finite-Bm}
 C_\ph=\max_{0\le t\le1}
 \frac{\ph(\sqrt{1-t^2})}{t^2},
 \qquad \overline B_m=C_\ph S_m.
\end{equation}
At $t=0$ the quotient is defined by its limit $\pi/4$.  Thus
$C_\ph<\infty$ and $\overline B_m=o(m^2)$ by \cref{thm:Sm}.  At a point of
$\Gamma_\eps$ with $r>0$, \eqref{eq:intro-cylinder-input} gives
\begin{equation}\label{eq:finite-cylinder-trace}
 0\le Q_A(r,z)
 \le\frac{S_m}{r^2}\ph\bigl(\sqrt{1-\eps^2r^2}\bigr)
 \le \overline B_m\eps^2.
\end{equation}
The same bound holds at the two tips by continuity.  Only this pointwise
trace bound will be used.

\begin{lemma}[Spectral comparison with the cylinder]\label{lem:cb-comparison}
For all sufficiently large $m$, set
\[
 \gamma_m^{\mathrm c}
 =\frac{(m-2)^2}{4}-2\La S_m>0.
\]
For every $A>0$, $\eps>0$ and $\zeta\in H^1_0(D_\eps)$,
\begin{equation}\label{eq:cb-cylinder-gap}
 \mathfrak q_{Q_A,D_\eps}[\zeta]
 \ge2\eps\sqrt{\La\gamma_m^{\mathrm c}}
       \int_{D_\eps}\zeta^2.
\end{equation}
If a nonnegative classical solution $q$ of
$-\Delta q=\La(q+q^2)$ in $D_\eps$ satisfies
$q|_{\Gamma_\eps}\ge Q_A|_{\Gamma_\eps}$, then $q\ge Q_A$.
Consequently the constant-boundary problem has no nonnegative bounded
classical solution when $\delta\ge \overline B_m\eps^2$.
\end{lemma}

\begin{proof}
It suffices to prove the form estimate for smooth compactly supported
functions and then use density.  At fixed $X$, the section is
$(-h(r),h(r))$, with $h(r)=\sqrt{1-\eps^2r^2}$.  The one-dimensional
Dirichlet inequality gives
\begin{align}
 \int_{D_\eps}(|\zeta_z|^2-\La\zeta^2)
 &\ge\La\int_{D_\eps}
        \frac{\eps^2r^2}{1-\eps^2r^2}\zeta^2
 \ge\La\eps^2\int_{D_\eps}r^2\zeta^2.
 \label{eq:cb-sectional-gap}
\end{align}
At fixed $z$, extend $\zeta$ by zero in the $X$ variables and apply the
$m$-dimensional Hardy inequality.  Since $Q_A\le S_m r^{-2}$,
\[
 \mathfrak q_{Q_A,D_\eps}[\zeta]
 \ge\int_{D_\eps}
       \left(\frac{\gamma_m^{\mathrm c}}{r^2}
                       +\La\eps^2r^2\right)\zeta^2.
\]
The arithmetic--geometric mean inequality proves
\eqref{eq:cb-cylinder-gap}.

Now let $v=(Q_A-q)_+$.  The boundary inequality gives
$v\in H^1_0(D_\eps)$.  Testing the difference of the two classical
equations by $v$ yields the exact identity
\[
 \mathfrak q_{Q_A,D_\eps}[v]= -\La\int_{D_\eps}v^3\le0.
\]
Estimate \eqref{eq:cb-cylinder-gap} implies $v=0$.  If
$q|_{\Gamma_\eps}=\delta\ge \overline B_m\eps^2$, then
\eqref{eq:finite-cylinder-trace} permits this comparison for every
$A>0$.  Therefore $q(0,0)\ge Q_A(0,0)=A$ for all $A$, which is
impossible for a bounded classical solution.
\end{proof}

For the rest of this subsection fix such an $m$ and a number $\eps>0$.
\begin{definition}[Constant-boundary admissible parameter, minimal solution, and endpoint]\label{def:cb-minimal-branch}
A number $\delta>0$ is \emph{admissible} if the constant-boundary problem
\[
 -\Delta q=\La(q+q^2)\quad\text{in }D_\eps,
 \qquad q=\delta\quad\text{on }\Gamma_\eps
\]
admits a nonnegative bounded classical solution.  Its endpoint parameter is
\begin{equation}\label{eq:cb-delta-star}
 \delta_\eps^*
 :=\sup\{\delta>0:\ \delta\text{ is admissible}\}.
\end{equation}
For an admissible $\delta$, a \emph{minimal solution} $q_\delta$ is the
pointwise smallest nonnegative classical solution at that boundary value.
The family of such $q_\delta$ on the admissible interval is the
\emph{constant-boundary minimal branch}.  Its monotone pointwise limit as
$\delta\uparrow\delta_\eps^*$ is the \emph{constant-boundary endpoint}.
\end{definition}
For a solution $q$, the maximum principle gives $q\ge\delta$.  Writing
$u=q-\delta$ reduces the problem to
\begin{equation}\label{eq:cb-zero-equation}
 \begin{cases}
 -\Delta u=\La\bigl[u^2+(1+2\delta)u+\delta(1+\delta)\bigr]
       &\text{in }D_\eps,\\
 u=0&\text{on }\Gamma_\eps.
 \end{cases}
\end{equation}

The subsolution--supersolution iteration used below is classical;
see Sattinger \cite{Sattinger}.  The organization into a monotone
minimal branch and the relation between minimality and nonnegativity
of the first linearized eigenvalue are standard features of the
positive increasing convex theory; compare Mignot--Puel \cite{MP} and
Cabr\'e--Capella \cite{CC}.  Since our parameter is a nonzero constant
boundary value rather than the multiplier in the equation, we give the
short argument in full.

\begin{proposition}[Minimal branch and monotonicity]
\label{prop:cb-minimal}
The endpoint in \eqref{eq:cb-delta-star} satisfies
\begin{equation}\label{eq:cb-delta-range}
 0<\delta_\eps^*\le \overline B_m\eps^2.
\end{equation}
For every $0<\delta<\delta_\eps^*$ there is a smallest nonnegative
classical solution $q_\delta$.  Both $q_\delta$ and
$u_\delta=q_\delta-\delta$ are strictly increasing with $\delta$ in the
interior.  The first Dirichlet eigenvalue of
\[
 L_{q_\delta}=-\Delta-\La(1+2q_\delta)
\]
is positive.  Moreover $u_\delta$ is $O(m)$-invariant in $X$, even in
$z$, and satisfies
\begin{equation}\label{eq:cb-monotonicities}
 (u_\delta)_r\le0,
 \qquad (u_\delta)_z\le0\quad\text{for }z>0.
\end{equation}
\end{proposition}

\begin{proof}
The first Dirichlet eigenvalue of $-\Delta$ on $D_\eps$ is strictly
larger than $\La$.  Indeed, the sectional Dirichlet inequality gives
$\int|\zeta_z|^2\ge\La\int\zeta^2$, and equality in the full Rayleigh
quotient would force the first eigenfunction to be independent of $X$.
Its zero boundary trace then forces it to vanish.  Hence the classical
Dirichlet problem
\[
 (-\Delta-\La)\mathcal R_\eps=1\quad\text{in }D_\eps,
 \qquad \mathcal R_\eps|_{\Gamma_\eps}=0
\]
has a bounded positive solution.  For sufficiently small $\delta>0$,
$\overline u=2\La\delta\mathcal R_\eps$ is a supersolution of
\eqref{eq:cb-zero-equation}, since its residual is
\[
 2\La\delta-
 \La\bigl[\overline u^2+2\delta\overline u+\delta(1+\delta)\bigr]
 =\La\delta+O_{m,\eps}(\delta^2)>0.
\]
The zero function is a subsolution.  Monotone Poisson iteration
between these two functions gives a classical solution; this is the
standard monotone method of \cite{Sattinger}, specialized here to the
shifted zero-boundary equation \eqref{eq:cb-zero-equation}.

Whenever a solution exists at a larger boundary value, it is a
supersolution for every smaller positive boundary value.  Iteration
starting from the constant $\delta$ is bounded by any such
supersolution and converges to the minimal solution $q_\delta$.  It
also proves comparison between minimal solutions at different
parameters.  In particular the set of admissible positive parameters
is an interval and $\delta_\eps^*>0$.  The upper bound follows from
\cref{lem:cb-comparison}.

We next prove semistability.  Suppose that the first eigenvalue $\lambda_1(L_{q_\delta})$
of $L_{q_\delta}$ were negative, and let $\phi>0$ be its Dirichlet
eigenfunction.  The strong maximum principle and Hopf's lemma imply
$u_\delta>0$, $-\partial_\nu u_\delta>0$, and
$u_\delta\ge c\phi$ for some $c>0$.  For small $t>0$,
$q_\delta-t\phi\ge\delta$ and
\[
 -\Delta(q_\delta-t\phi)
   -\La\bigl[(q_\delta-t\phi)+(q_\delta-t\phi)^2\bigr]
 =-t\lambda_1(L_{q_\delta})\phi-\La t^2\phi^2\ge0.
\]
Iteration below this smaller supersolution contradicts minimality.
Thus $\lambda_1(L_{q_\delta})\ge0$.

Choose $\delta<\delta'<\delta_\eps^*$ and put
$w=q_{\delta'}-q_\delta\ge0$.  Subtraction gives
\[
 L_{q_\delta}w=\La w^2,
 \qquad w|_{\Gamma_\eps}=\delta'-\delta.
\]
For the positive first eigenfunction $\phi$ of $L_{q_\delta}$, Green's
identity gives
\begin{equation}\label{eq:cb-strict-stability}
 \lambda_1(L_{q_\delta})\int_{D_\eps}w\phi
 =\La\int_{D_\eps}w^2\phi
     +(\delta'-\delta)\int_{\Gamma_\eps}(-\partial_\nu\phi)>0.
\end{equation}
It follows that $\lambda_1(L_{q_\delta})>0$.  Also $-\Delta w\ge0$ and its boundary value
is positive, so $w\ge\delta'-\delta>0$.  Since $q_{\delta'}>q_\delta$
and $t+t^2$ is strictly increasing for $t\ge0$,
$-\Delta(w-\delta'+\delta)>0$; hence $u_{\delta'}>u_\delta$ in the
interior.

Every orthogonal transformation in $X$, and reflection in $z$, maps a
minimal solution to a minimal solution.  Uniqueness of the smallest
solution therefore gives the asserted symmetries.  In the half-domain
$D_\eps\cap\{X_i>0\}$, the derivative
$v=\partial_{X_i}u_\delta$ satisfies $L_{q_\delta}v=0$.  It vanishes
on the symmetry plane, while on the physical boundary
\[
 v=(\partial_\nu u_\delta)\nu_i\le0,
\]
because the outward normal component $\nu_i$ has the sign of $X_i$.
The zero extension of $v_+$ belongs to $H^1_0(D_\eps)$.  Testing the
derivative equation by $v_+$ and using $\lambda_1(L_{q_\delta})>0$ gives $v_+=0$.
Radial symmetry now implies $(u_\delta)_r\le0$.  The same argument in
$D_\eps\cap\{z>0\}$ gives the second inequality in
\eqref{eq:cb-monotonicities}.
\end{proof}

\subsection{The large-dimensional annulus estimate}

In this subsection $\overline B_m\ge0$ is any sequence satisfying $\overline B_m/m^2\to0$;
it need not initially be the sequence in \eqref{eq:finite-Bm}.
\begin{definition}[Finite-domain annulus class and annulus height]\label{def:annulus-class}
The \emph{finite-domain annulus class} $\mathcal E_m(\overline B_m)$ consists of
triples $(\eps,\delta,u)$ such that
\[
 0<\eps\le m^{-2},\qquad0\le\delta\le \overline B_m\eps^2,
\]
$u\in C^2(\overline{D_\eps})$ is nonnegative, $O(m)$-invariant in $X$,
even in $z$, solves \eqref{eq:cb-zero-equation}, and satisfies
\[
 u_r\le0,\qquad u_z\le0\quad(z>0).
\]
The zero solution at $\delta=0$ is allowed.  The associated
\emph{annulus height} is the extended nonnegative number
\begin{equation}\label{eq:finite-Tm}
 T_m=T_m(\overline B_m)
 :=\sup\{r^2u(r,0):
       (\eps,\delta,u)\in\mathcal E_m(\overline B_m),\ 0<r<\eps^{-1}\}.
\end{equation}
No stability condition is included in this definition.
\end{definition}

\begin{theorem}[Finite-domain annulus estimate]\label{thm:finite-annulus}
For every nonnegative sequence $\overline B_m=o(m^2)$,
\begin{equation}\label{eq:finite-Tm-small}
 \frac{T_m(\overline B_m)}{m^2}\longrightarrow0
 \qquad(m\to\infty).
\end{equation}
In particular $T_m(\overline B_m)<\infty$ for every sufficiently large $m$.
\end{theorem}

\begin{proof}
Suppose \eqref{eq:finite-Tm-small} fails.  There are $\kappa>0$,
integers $m_j\to\infty$ and triples
$(\eps_j,\delta_j,u_j)\in\mathcal E_{m_j}(\overline B_{m_j})$ for which
\[
 \sup_{0<r<\eps_j^{-1}}r^2u_j(r,0)>2\kappa m_j^2.
\]
For each fixed solution, $r^2u_j(r,0)$ is continuous and vanishes both
at $r=0$ and at $r=\eps_j^{-1}$.  Let $r_j$ be its first contact with
$\kappa m_j^2$, and write
\begin{equation}\label{eq:finite-first-contact}
 M_j=u_j(r_j,0)>0,
 \qquad r_j^2M_j=\kappa m_j^2.
\end{equation}
First contact and the two monotonicities imply the global bound, within
the physical domain,
\begin{equation}\label{eq:finite-first-bounds}
 0\le\frac{u_j(r,z)}{M_j}\le
 \begin{cases}
 (r_j/r)^2,&0<r\le r_j,\\
 1,&r\ge r_j.
 \end{cases}
\end{equation}
The parameter assumptions and $r_j<\eps_j^{-1}$ also give
\begin{equation}\label{eq:finite-small-parameters}
 \delta_j\to0,
 \qquad\frac{\delta_j}{M_j}
 \le\frac{\overline B_{m_j}}{\kappa m_j^2}\to0,
 \qquad\eps_jm_j\le m_j^{-1}\to0.
\end{equation}
Passing to a subsequence, $M_j$ tends to a finite positive number, to
infinity, or to zero.  All local estimates below are for the
uniformly elliptic equations in the two variables $(r,z)$, away from
the radial axis.  Their drift coefficients are explicitly controlled;
no estimate in a space of increasing dimension is invoked.

\medskip\noindent
\emph{Step 1: a finite positive limiting amplitude is impossible.}
Suppose $M_j\to M_\infty\in(0,\infty)$.  In the translated physical
variables define
\[
 U_j(x,z)=M_j^{-1}u_j(r_j+x,z)
 \quad\hbox{on}\quad
 \Omega_j^{(1)}:=\bigl\{(x,z):x>-r_j,\ \eps_j^2(r_j+x)^2+z^2<1\bigr\}.
\]
Thus $\Omega_j^{(1)}$ is exactly the image of the radial half-plane part of
$D_{\eps_j}$ under the translation $r=r_j+x$.  Then $r_j\to\infty$,
$\eps_jr_j\to0$, and
\[
 \frac{m_j-1}{r_j+x}\longrightarrow
 b:=\sqrt{M_\infty/\kappa}
\]
on every fixed $x$ interval.  The domains converge locally, including
their boundaries, to $\R\times(-1,1)$.  The equation reads
\begin{equation}\label{eq:finite-physical-window}
 -(U_j)_{xx}-(U_j)_{zz}
 -\frac{m_j-1}{r_j+x}(U_j)_x
 =\La\left[M_jU_j^2+(1+2\delta_j)U_j
                +\frac{\delta_j(1+\delta_j)}{M_j}\right].
\end{equation}
For every fixed $R>0$, \eqref{eq:finite-first-bounds} gives
\[
 \sup_{\substack{|x|\le R\\(x,z)\in\Omega_j^{(1)}}}U_j(x,z)
 \le \left(\frac{r_j}{r_j-R}\right)^2=1+o_R(1).
\]
Together with \eqref{eq:finite-small-parameters}, local and Dirichlet
boundary $W^{2,p}$ and Schauder estimates therefore give a subsequential
limit $U$ with
\[
 \begin{gathered}
 -U_{xx}-U_{zz}-bU_x=\La(U+M_\infty U^2),\\
 U(x,\pm1)=0,\qquad0\le U\le1,\qquad
 U(0,0)=1,\qquad U_x\le0.
 \end{gathered}
\]
The pointwise limit $U_-(z)=\lim_{x\to-\infty}U(x,z)$ exists.
Since $U_x\le0$, $U_-(0)\ge U(0,0)=1$, while $U\le1$; hence
$U_-(0)=1$.  To identify its equation, translate $U$ by any sequence tending to
$-\infty$.  The same compactness estimates apply; monotone convergence
shows that every local limit is $U_-(z)$, independent of $x$.
Consequently
\[
 -U_-''=\La(U_-+M_\infty U_-^2),\qquad
 U_-(\pm1)=0,\qquad U_-(0)=1.
\]
Multiplication by $\ph$ and two integrations by parts cancel the
linear term and yield $M_\infty\int_{-1}^1U_-^2\ph=0$, a contradiction.

\medskip\noindent
\emph{Step 2: an infinite limiting amplitude is impossible.}
Suppose $M_j\to\infty$ and set
\[
 U_j(x,y)=M_j^{-1}u_j
       (r_j+M_j^{-1/2}x,M_j^{-1/2}y),
 \qquad (x,y)\in\Omega_j^{(2)},
\]
where
\[
 \Omega_j^{(2)}:=\left\{(x,y):
 x>-r_j\sqrt{M_j},\quad
 \eps_j^2\bigl(r_j+M_j^{-1/2}x\bigr)^2+M_j^{-1}y^2<1\right\}.
\]
This is the image of $D_{\eps_j}$ under the coefficient-scale recentering
$r=r_j+M_j^{-1/2}x$, $z=M_j^{-1/2}y$.  The identities
\[
 r_j\sqrt{M_j}=\sqrt\kappa\,m_j\to\infty,
 \qquad
 (\eps_jr_j)^2=\frac{\kappa\eps_j^2m_j^2}{M_j}\to0
\]
show that the rescaled domains exhaust $\R^2$.  On every fixed
compact set, \eqref{eq:finite-first-bounds} also yields
\[
 U_j(x,y)\le
 \left(1-\frac{R}{r_j\sqrt{M_j}}\right)^{-2}=1+o_R(1)
 \qquad (|x|+|y|\le R),
\]
for all sufficiently large $j$.  The drift coefficient converges to
$\kappa^{-1/2}$; the linear and constant terms tend to zero.  Compactness
gives
\[
 -U_{xx}-U_{yy}-\kappa^{-1/2}U_x=\La U^2,
 \qquad0\le U\le1,\quad U(0,0)=1,\quad U_x\le0.
\]
As in Step~1, the negative-end limit $U_-(y)$ is a bounded nonnegative
solution of $-U_-''=\La U_-^2$ on $\R$.  Since $U_x\le0$,
$U(0,0)=1$, and $U\le1$, one has $U_-(0)=1$.
A nonnegative concave function on the whole line is constant; its
equation then makes it zero.  This is a contradiction.

\medskip\noindent
\emph{Step 3: at low amplitude the cross-section must approach the full strip.}
It remains to consider $M_j\to0$.  We retain the translated physical
variables and the domain $\Omega_j^{(1)}$ from Step~1, namely
\[
 U_j(x,z)=M_j^{-1}u_j(r_j+x,z),\qquad (x,z)\in\Omega_j^{(1)},
\]
so that \eqref{eq:finite-physical-window} remains the exact equation in
these variables.  Define
\[
 t_j=\eps_jr_j\in(0,1),\qquad
 h_j^0=\sqrt{1-t_j^2}.
\]
Since $m_j/r_j=\sqrt{M_j/\kappa}\to0$, we have $r_j\to\infty$ and
the drift tends to zero on fixed windows.  Pass to a subsequence with
$h_j^0\to h\in[0,1]$.

If $h=0$, the part of the domain with $|x|<1$ lies in
$|z|<\bar h_j$, where
\[
 \bar h_j=\sqrt{1-\eps_j^2(r_j-1)^2}\longrightarrow0.
\]
For large $j$, \eqref{eq:finite-first-bounds} bounds $U_j$ there by
$2$, and the right-hand side of \eqref{eq:finite-physical-window} is
bounded above by a fixed $C$.  Let $k_j=\pi/(6\bar h_j)$, choose $A\ge C/2$,
and set
\[
 W_j(x,z)=A(\bar h_j^2-z^2)
   +4\frac{\cosh(k_jx)}{\cosh k_j}\cos(2k_jz).
\]
The cosine is at least $1/2$ for $|z|\le \bar h_j$.  With
$b_j(x)=(m_j-1)/(r_j+x)$, the operator
$-\partial_{xx}-\partial_{zz}-b_j\partial_x$ applied to the second
term is that term multiplied by
$3k_j^2-b_j(x)k_j\tanh(k_jx)$, which is nonnegative for large $j$.
Applied to the first term it equals $2A$.  Thus this operator applied
to $W_j$ is at least $C$.  On the physical boundary $U_j=0\le W_j$,
and on any artificial side $x=\pm1$ one has $U_j\le2\le W_j$.
The maximum principle on the intersection of the window with the
physical domain yields
\[
 1=U_j(0,0)\le A\bar h_j^2+\frac{4}{\cosh k_j}\longrightarrow0,
\]
a contradiction.

Suppose now $h>0$.  On fixed windows the boundary graphs converge
smoothly to $z=\pm h$.  Compactness gives a bounded nonnegative
solution
\[
 -U_{xx}-U_{zz}=\La U
   \quad\hbox{in }\R\times(-h,h),\qquad
 U(x,\pm h)=0,\qquad U(0,0)=1.
\]
Its projection $p$ on the positive first Dirichlet eigenfunction of
$(-h,h)$ is bounded and satisfies
\[
 -p''+\left(\frac{\La}{h^2}-\La\right)p=0.
\]
For $h<1$ every bounded solution of this equation on $\R$ is zero.
Nonnegativity of $U$ then forces $U=0$, a contradiction.  Hence $h=1$.
Expanding in the full Dirichlet eigenbasis of $(-1,1)$, the first
coefficient solves $-p''=0$ and is constant, while each higher
coefficient solves an equation with positive mass and is zero.
Thus the only such limit is $U(x,z)=\ph(z)$.  In particular,
\begin{equation}\label{eq:finite-t-small}
 t_j\longrightarrow0.
\end{equation}

\medskip\noindent
\emph{Step 4: Helly compactness and pointwise first-mode reduction.}
Let $I=(1/2,1)$ and set
\[
 \mathcal R:=I\times(-1,1).
\]
For $s\in I$ define
\[
 h_j(s)=\sqrt{1-t_j^2s^2},\qquad
 \Phi_j(s,y)=(r_js,h_j(s)y),
\]
and
\begin{equation}\label{eq:finite-flattening}
 \begin{gathered}
 v_j(s)=\frac{u_j(r_js,0)}{M_j},\qquad s\in I,\\
 V_j(s,y)=\frac{u_j(\Phi_j(s,y))}{M_j}
 =\frac{u_j(r_js,h_j(s)y)}{M_j},
 \qquad (s,y)\in\mathcal R.
 \end{gathered}
\end{equation}
The map $\Phi_j$ is well defined because
$\eps_j^2(r_js)^2+h_j(s)^2y^2<1$ for $|y|<1$.  The first-contact
bound and radial monotonicity imply
\begin{equation}\label{eq:finite-v-order}
 1\le v_j(s)\le s^{-2}\le4,
 \qquad v_j\text{ is nonincreasing on }I.
\end{equation}
By Helly's selection theorem for uniformly bounded monotone functions
(see, for example, \cite{RoydenFitzpatrick}), after passing to a
subsequence there is a nonincreasing function $v:I\to[1,4]$ such that
$v_j(s)\to v(s)$ at every continuity point of $v$.

\begin{claim*}[Pointwise first-mode reduction]
For every continuity point $s\in I$ of $v$,
\begin{equation}\label{eq:finite-first-mode}
 V_j(s,\cdot)\longrightarrow v(s)\ph
 \quad\text{in }C^0([-1,1]).
\end{equation}
\end{claim*}
\begin{proof}
Fix such an $s$ and put $\bar r_j=r_js$.  Consider the translated
profiles
\[
 \widetilde U_{j,s}(x,z)=M_j^{-1}u_j(\bar r_j+x,z)
\]
on
\[
 \widetilde\Omega_{j,s}
 =\bigl\{(x,z):x>-\bar r_j,\ 
   \eps_j^2(\bar r_j+x)^2+z^2<1\bigr\}.
\]
Fix $R>0$.  Since $s<1$ and $r_j\to\infty$, one has
$\bar r_j+x<r_j$ for $|x|\le R$ and all sufficiently large $j$.
Hence the first-contact estimate gives
\[
 0\le \widetilde U_{j,s}(x,z)
 \le \left(\frac{r_j}{\bar r_j+x}\right)^2
 \le \left(s-\frac{R}{r_j}\right)^{-2}
\]
on every fixed translated window.  Moreover
\[
 \frac{m_j-1}{\bar r_j+x}
 =\frac{(m_j-1)/r_j}{s+x/r_j}\longrightarrow0
\]
uniformly there, since $m_j/r_j=\sqrt{M_j/\kappa}\to0$, and
\[
 \eps_j^2(\bar r_j+x)^2
 =t_j^2\left(s+\frac{x}{r_j}\right)^2\longrightarrow0
\]
by \eqref{eq:finite-t-small}.  Thus the domains
$\widetilde\Omega_{j,s}$ converge locally, including the Dirichlet
boundary, to the full strip $\R\times(-1,1)$.  Since
$M_j\to0$, $\delta_j/M_j\to0$, and $\delta_j\to0$, standard interior
and boundary compactness yields, from any subsequence, a further
subsequence converging locally up to the boundary to a bounded
nonnegative solution
\[
 -\widetilde U_{xx}-\widetilde U_{zz}=\La\widetilde U
 \quad\text{in }\R\times(-1,1),\qquad
 \widetilde U(x,\pm1)=0.
\]
At $x=0$,
\[
 \widetilde U_{j,s}(0,0)=v_j(s)\longrightarrow v(s),
\]
so $\widetilde U(0,0)=v(s)$.  Expanding in the Dirichlet
eigenfunctions of $(-1,1)$, the first coefficient solves $-a_1''=0$
and is constant by boundedness on $\R$, while every higher coefficient
satisfies
\[
 -a_k''+(\lambda_k-\La)a_k=0,
 \qquad \lambda_k>\La,
\]
and hence vanishes.  Since $\ph(0)=1$, necessarily
\[
 \widetilde U(x,z)=v(s)\ph(z).
\]
The limit is unique, and therefore the whole sequence converges to this
profile locally up to the strip boundary.  Finally
$h_j(s)=\sqrt{1-t_j^2s^2}\to1$.  Uniform boundary Schauder estimates
(or, equivalently, a one-dimensional flattening of the interval
$(-h_j(s),h_j(s))$ at $x=0$) then give
\[
 \widetilde U_{j,s}(0,h_j(s)\,\cdot)
 \longrightarrow v(s)\ph
 \quad\text{in }C^0([-1,1]).
\]
The left-hand side is exactly $V_j(s,\cdot)$, proving
\eqref{eq:finite-first-mode}.
\end{proof}

Since a monotone function has at most countably many discontinuities,
the Helly convergence holds almost everywhere.  The bound $|v_j|\le4$
and dominated convergence on compact subintervals therefore yield
\begin{equation}\label{eq:finite-v-limit}
 \begin{gathered}
 v_j\to v\quad\text{almost everywhere and in }L^p_{\rm loc}(I)
       \quad(1\le p<\infty),\\
 1\le v\le4,
 \qquad v\text{ is nonincreasing}.
 \end{gathered}
\end{equation}
No estimate of an $s$-derivative is used in this step.

\medskip\noindent
\emph{Step 5: exact projection on the varying cross-section.}
Define
\[
 \begin{gathered}
 \beta_j=\frac{h_j'}{h_j},\qquad
 p_j(s)=\int_{-1}^1V_j(s,y)\ph(y)\,dy,
 \qquad \mathcal F_j(s)=\int_{-1}^1V_j(s,y)^2\ph(y)\,dy,\\
 G_j(s)=\int_{-1}^1y(V_j)_y\ph\,dy,
 \qquad \mathcal H_j(s)=\int_{-1}^1y^2(V_j)_{yy}\ph\,dy.
 \end{gathered}
\]
The moments are defined on $I$.  By evenness and monotonicity in $z$,
\[
 0\le V_j(s,y)\le V_j(s,0)=v_j(s)\le4.
\]
Hence \eqref{eq:finite-first-mode}, which holds at every continuity
point of $v$, together with dominated convergence gives, on every
compact subinterval of $I$,
\begin{equation}\label{eq:finite-projection-limits}
 p_j\longrightarrow v,
 \qquad \mathcal F_j\longrightarrow B_2v^2
 \quad\hbox{in }L^1_{\rm loc}(I),
 \qquad B_2=\int_{-1}^1\ph^3\,dy=\frac8{3\pi}.
\end{equation}
Since $V_j(s,\pm1)=\ph(\pm1)=0$, integration by parts
gives
\begin{equation}\label{eq:finite-GH}
 G_j=-\int_{-1}^1V_j(\ph+y\ph')\,dy,
 \qquad
 \mathcal H_j=\int_{-1}^1V_j(y^2\ph)''\,dy.
\end{equation}
In particular $p_j,\mathcal F_j,G_j,\mathcal H_j$ are uniformly bounded on compact
subintervals of $I$.

We next write explicitly the equation satisfied by $V_j$ in the
flattened variables.  At fixed physical $z$, one has
\[
 y=\frac{z}{h_j(s)},\qquad
 \frac{dy}{ds}\Big|_z=-\beta_j(s)y,
\]
and hence, with
\[
 D_j:=\partial_s-\beta_jy\partial_y,
\]
\[
 \partial_r=r_j^{-1}D_j.
\]
At fixed $r$ (equivalently, fixed $s$), one has instead
\[
 \partial_z=h_j^{-1}\partial_y.
\]
Moreover
\[
 D_j^2
 =\partial_{ss}-2\beta_jy\partial_{sy}
       +\beta_j^2y^2\partial_{yy}
       +(\beta_j^2-\beta_j')y\partial_y.
\]
Since
\[
 u_j(r_js,h_j(s)y)=M_jV_j(s,y),
\]
substitution into \eqref{eq:cb-zero-equation}, followed by division by
$M_j$, gives the exact equation on $\mathcal R=I\times(-1,1)$:
\[
\begin{split}
 &-\frac1{r_j^2}\Bigg[
 (V_j)_{ss}+\frac{m_j-1}{s}(V_j)_s
 -2\beta_jy(V_j)_{sy}+\beta_j^2y^2(V_j)_{yy}\\
 &\hspace{28mm}
 +\left(\beta_j^2-\beta_j'
       -\frac{m_j-1}{s}\beta_j\right)y(V_j)_y
 \Bigg]
 -\frac1{h_j^2}(V_j)_{yy}\\
 &\qquad=\La\left[
 M_jV_j^2+(1+2\delta_j)V_j
 +\frac{\delta_j(1+\delta_j)}{M_j}\right],
 \qquad V_j(s,\pm1)=0.
\end{split}
\]
We now project this equation onto the first transverse mode $\ph$.
The pure $s$-derivatives give
$p_j''+(m_j-1)s^{-1}p_j'$.  Furthermore,
\[
 \int_{-1}^1y(V_j)_{sy}\ph\,dy=G_j',\qquad
 \int_{-1}^1y^2(V_j)_{yy}\ph\,dy=\mathcal H_j,
\]
and the remaining $y(V_j)_y$ term gives the corresponding multiple of
$G_j$.  Finally, using $V_j(s,\pm1)=\ph(\pm1)=0$ and
$-\ph''=\La\ph$, two integrations by parts yield
\[
 -\int_{-1}^1(V_j)_{yy}\ph\,dy
 =\La\int_{-1}^1V_j\ph\,dy=\La p_j.
\]
Thus the projection gives the following exact classical identity on
$I$:
\begin{equation}\label{eq:finite-exact-projection}
 \begin{split}
 -r_j^{-2}\left[p_j''+\frac{m_j-1}{s}p_j'+E_j\right]
       +\frac{\La}{h_j^2}p_j
 =\La\left[(1+2\delta_j)p_j+M_j\mathcal F_j
        +\frac{\delta_j(1+\delta_j)}{M_j}\mathsf B_1\right],
 \end{split}
\end{equation}
where
\begin{equation}\label{eq:finite-Ej}
 \begin{gathered}
 E_j=-2\beta_jG_j'+\beta_j^2\mathcal H_j+
 \left(\beta_j^2-\beta_j'-\frac{m_j-1}{s}\beta_j\right)G_j,
 \qquad \mathsf B_1=\int_{-1}^1\ph\,dy=\frac4\pi.
 \end{gathered}
\end{equation}
For a fixed $I'\Subset I$, direct differentiation of $h_j$ gives
\[
 \|\beta_j\|_{L^\infty(I')}
       +\|\beta_j'\|_{L^\infty(I')}\le C_{I'}t_j^2.
\]
For every $\psi\in C_c^\infty(I)$, move the $s$ derivatives in
\eqref{eq:finite-exact-projection} and \eqref{eq:finite-Ej} onto
$\psi$.  The uniform bounds following \eqref{eq:finite-GH} then give
\begin{equation}\label{eq:finite-distribution-errors}
 \begin{split}
 \left|\left\langle p_j''+\frac{m_j-1}{s}p_j',\psi\right\rangle\right|
       &\le C_\psi m_j,\\
 |\langle E_j,\psi\rangle|&\le C_\psi m_jt_j^2.
 \end{split}
\end{equation}
For example, the only derivative of $G_j$ appears as
$\langle-2\beta_jG_j',\psi\rangle
=2\int G_j(\beta_j\psi)'$, so it is bounded by $C_\psi t_j^2$.
The bounds in \eqref{eq:finite-distribution-errors} therefore do not
presuppose any control of $G_j'$ or $p_j'$.

Using $r_j^2M_j=\kappa m_j^2$, divide
\eqref{eq:finite-exact-projection} by $\La M_j$ and set
\begin{equation}\label{eq:finite-alpha-d}
 a_j=\frac{t_j^2}{M_j}>0,
 \qquad d_j=\frac{\delta_j(1+\delta_j)}{M_j^2}\ge0.
\end{equation}
The result is
\begin{equation}\label{eq:finite-reduced-projection}
 a_j\frac{s^2}{1-t_j^2s^2}p_j
   =\mathcal F_j+2\frac{\delta_j}{M_j}p_j+\mathsf B_1d_j+R_j,
 \qquad R_j\longrightarrow0\quad\hbox{in }\mathcal D'(I).
\end{equation}
Indeed \eqref{eq:finite-distribution-errors} bounds each pairing of
$R_j$ by $C_\psi(1+t_j^2)/m_j$.  The constant forcing has the
additional exact ratio
\begin{equation}\label{eq:finite-forcing-ratio}
 \frac{d_j}{a_j}
 =\frac{\delta_j(1+\delta_j)}{\kappa\eps_j^2m_j^2}
 \le\frac{\overline B_{m_j}(1+\delta_j)}{\kappa m_j^2}
 \longrightarrow0.
\end{equation}
This estimate retains the possible interaction between the geometric
cross-sectional term and the small boundary forcing.

\medskip\noindent
\emph{Step 6: contradiction with radial monotonicity.}
If $a_j\to\infty$, divide
\eqref{eq:finite-reduced-projection} by $a_j$.
Equations \eqref{eq:finite-small-parameters},
\eqref{eq:finite-t-small}, \eqref{eq:finite-forcing-ratio} and
\eqref{eq:finite-projection-limits} then give $s^2v=0$ in
$\mathcal D'(I)$, contrary to $v\ge1$.

Otherwise a subsequence has $a_j\to a_\infty\in[0,\infty)$.
Equation \eqref{eq:finite-forcing-ratio} gives $d_j\to0$, so
\eqref{eq:finite-reduced-projection} becomes
\[
 a_\infty s^2v=B_2v^2\quad\text{almost everywhere in }I.
\]
If $a_\infty=0$, this contradicts $v\ge1$.  If $a_\infty>0$, division by
$v$ gives
\[
 v(s)=\frac{a_\infty}{B_2}s^2\quad\text{almost everywhere in }I.
\]
The right-hand side is strictly increasing and cannot have a
nonincreasing representative.  For instance choose two Lebesgue
points $s_1<s_2$ from the full-measure equality set and use
\eqref{eq:finite-v-limit}.  This is the final contradiction.

The argument applies as well if $T_m=\infty$ for infinitely many
unbounded values of $m$, since one may still choose a solution crossing
$2\kappa m^2$.  Both assertions of the theorem follow.
\end{proof}

\subsection{A spectral gap along the entire bounded branch}

Return to $\overline B_m=C_\ph S_m$ from \eqref{eq:finite-Bm}, and use the
corresponding number $T_m(\overline B_m)$.  Choose $m_1$ sufficiently large that,
for every $m\ge m_1$, the cylinder branch exists globally and
\begin{equation}\label{eq:finite-gamma-choice}
 \gamma_m^{\mathrm c}>0,
 \qquad T_m<\infty,
 \qquad
 \gamma_m^{\mathrm e}:=\frac{(m-2)^2}{4}-2\La T_m>0.
\end{equation}
This is possible by \cref{thm:Sm,thm:finite-annulus}.  For such an $m$
fix
\begin{equation}\label{eq:finite-epsilon-choice}
 \eps_0(m)=\min\left\{m^{-2},
       \frac{\sqrt{\gamma_m^{\mathrm e}}}{2\sqrt\La \overline B_m}\right\}>0.
\end{equation}
The constant $\overline B_m$ is positive by the positivity of the cylinder branch.

\begin{proposition}[Uniform finite-domain spectral gap]
\label{prop:cb-uniform-gap}
Let $m\ge m_1$, $0<\eps\le\eps_0(m)$, and
$(\eps,\delta,u)\in\mathcal E_m(\overline B_m)$.  Set $q=u+\delta$.  Then
\begin{equation}\label{eq:cb-uniform-gap}
 \mathfrak q_{q,D_\eps}[\zeta]
 \ge\eps\sqrt{\La\gamma_m^{\mathrm e}}
             \|\zeta\|_{L^2(D_\eps)}^2
 \qquad(\zeta\in H^1_0(D_\eps)).
\end{equation}
In particular this estimate holds for every bounded minimal solution
$q_\delta$, $0<\delta<\delta_\eps^*$, with the same constant.
\end{proposition}

\begin{proof}
By definition of $T_m$ and axial monotonicity,
$u(r,z)\le u(r,0)\le T_mr^{-2}$.  Hardy's inequality and
\eqref{eq:cb-sectional-gap} therefore imply
\begin{align*}
 \mathfrak q_{q,D_\eps}[\zeta]
 &\ge\int_{D_\eps}
     \left(\frac{\gamma_m^{\mathrm e}}{r^2}
                    +\La\eps^2r^2-2\La\delta\right)\zeta^2\\
 &\ge\left(2\eps\sqrt{\La\gamma_m^{\mathrm e}}
                        -2\La \overline B_m\eps^2\right)
                  \int_{D_\eps}\zeta^2.
\end{align*}
The choice \eqref{eq:finite-epsilon-choice} proves
\eqref{eq:cb-uniform-gap}.  Proposition~\ref{prop:cb-minimal} and
\eqref{eq:cb-delta-range} place every bounded minimal solution in
$\mathcal E_m(\overline B_m)$, which gives the last assertion.
\end{proof}

\section{Singular endpoints and extremal identification}
\label{sec:finite-limit}

Fix $m\ge m_1$ and $0<\eps\le\eps_0(m)$ as in
\eqref{eq:finite-gamma-choice}--\eqref{eq:finite-epsilon-choice}.
All constants in the compactness argument below are allowed to depend
on these fixed values.  We retain the minimal branch $q_\delta$ and
write $u_\delta=q_\delta-\delta$.

\begin{lemma}[Energy compactness at the boundary-parameter endpoint]
\label{lem:cb-energy}
The increasing limit
\[
 q_\eps=\lim_{\delta\uparrow\delta_\eps^*}q_\delta
\]
is finite almost everywhere and belongs to
$H^1(D_\eps)\cap L^3(D_\eps)$.  It satisfies
$q_\eps-\delta_\eps^*\in H^1_0(D_\eps)$ and
\begin{equation}\label{eq:cb-global-weak}
 \int_{D_\eps}\nabla q_\eps\cdot\nabla\zeta
 =\La\int_{D_\eps}(q_\eps+q_\eps^2)\zeta
 \quad
 \left(\zeta\in C^2(\overline{D_\eps}),\
              \zeta|_{\Gamma_\eps}=0\right).
\end{equation}
Furthermore
\begin{equation}\label{eq:cb-limit-stability}
 \int_{D_\eps}
   \bigl(|\nabla\zeta|^2-\La(1+2q_\eps)\zeta^2\bigr)
 \ge\eps\sqrt{\La\gamma_m^{\mathrm e}}
                 \int_{D_\eps}\zeta^2
 \qquad(\zeta\in C_c^1(D_\eps)).
\end{equation}
\end{lemma}

\begin{proof}
For each $0<\delta<\delta_\eps^*$ the minimal solution $q_\delta$ is
bounded and classical.  Hence the potential $\La(1+2q_\delta)$ is bounded
on $D_\eps$, and the quadratic form
$\mathfrak q_{q_\delta,D_\eps}$ is continuous on $H^1_0(D_\eps)$.
Since $C_c^\infty(D_\eps)$ is dense in $H^1_0(D_\eps)$, the
semistability inequality, initially stated for compactly supported smooth
test functions, extends to every $H^1_0(D_\eps)$ test function.  In
particular $u_\delta=q_\delta-\delta\in H^1_0(D_\eps)$ is admissible.

Test \eqref{eq:cb-zero-equation} by $u_\delta$ and test the extended
semistability inequality by the same function.  Their difference is
\begin{equation}\label{eq:cb-L3-cancellation}
 \int_{D_\eps}u_\delta^3
 \le\delta(1+\delta)\int_{D_\eps}u_\delta.
\end{equation}
Since $q_\delta\ge\delta$ and $\delta>0$, the right-hand side of
\eqref{eq:cb-zero-equation} is strictly positive; hence the strong maximum
principle gives $u_\delta>0$ in $D_\eps$.  H\"older's inequality applied to
\eqref{eq:cb-L3-cancellation} therefore gives
\begin{equation}\label{eq:cb-L3-bound}
 \int_{D_\eps}u_\delta^3
 \le[\delta(1+\delta)]^{3/2}|D_\eps|.
\end{equation}
The energy identity
\[
 \int_{D_\eps}|\nabla u_\delta|^2
 =\La\int_{D_\eps}
 \bigl[u_\delta^3+(1+2\delta)u_\delta^2
                  +\delta(1+\delta)u_\delta\bigr]
\]
and \eqref{eq:cb-delta-range} now give a uniform $H^1_0$ bound.
Thus $q_\delta$ is uniformly bounded in $H^1\cap L^3$ as
$\delta\uparrow\delta_\eps^*$.

Pointwise monotonicity defines $q_\eps$.  Fatou's lemma and the $L^3$
bound show that it is finite almost everywhere and belongs to $L^3$.
The weak $H^1$ limit of any subsequence is the same function, and
Rellich compactness yields
\begin{equation}\label{eq:cb-convergences}
 \begin{gathered}
 q_\delta\rightharpoonup q_\eps\quad\hbox{in }H^1(D_\eps),
 \qquad
 q_\delta\to q_\eps\quad\hbox{in }L^2(D_\eps),\\
 q_\delta^2\to q_\eps^2\quad\hbox{in }L^1(D_\eps).
 \end{gathered}
\end{equation}
The last assertion also follows directly from the product estimate
\[
 \|q_\delta^2-q_\eps^2\|_{L^1}
 \le\|q_\delta-q_\eps\|_{L^2}
       (\|q_\delta\|_{L^2}+\|q_\eps\|_{L^2}).
\]
Weak closedness of $H^1_0$ gives
$q_\eps-\delta_\eps^*\in H^1_0(D_\eps)$.  For every test function in
\eqref{eq:cb-global-weak}, its boundedness and the convergences above
allow passage in the classical identity for $q_\delta$.  This proves
the global weak formulation, not merely its compactly supported
version.  Finally pass to the limit in \eqref{eq:cb-uniform-gap} for a
fixed compactly supported $\zeta$ to obtain
\eqref{eq:cb-limit-stability}.
\end{proof}

\begin{theorem}[Singular constant-boundary endpoint]
\label{thm:finite-singular}
For every fixed $m\ge m_1$ and $0<\eps\le\eps_0(m)$, set
$\delta_\eps=\delta_\eps^*$.  Then
\begin{equation}\label{eq:cb-singular-endpoint}
 \begin{gathered}
 0<\delta_\eps\le \overline B_m\eps^2,\qquad
 q_\eps\in H^1(D_\eps)\cap L^3(D_\eps)
                      \setminus L^\infty(D_\eps),\\
 -\Delta q_\eps=\La(q_\eps+q_\eps^2)\quad\hbox{in }D_\eps,
 \qquad q_\eps-\delta_\eps\in H^1_0(D_\eps),\\
 q_\eps>\delta_\eps\quad\hbox{almost everywhere in }D_\eps.
 \end{gathered}
\end{equation}
The solution is semistable, and the function
\begin{equation}\label{eq:cb-scaled-solution}
 U_\eps(y,x_N)
  =\frac{q_\eps(y/\eps,x_N/\eps)}{\delta_\eps}-1
\end{equation}
is precisely $u^*(\Omega_\eps,f_{\delta_\eps})$.  Its extremal
parameter is
\begin{equation}\label{eq:cb-extremal-parameter}
 \lambda^*(\Omega_\eps,f_{\delta_\eps})
   =\frac{\La(1+\delta_\eps)}{\eps^2}.
\end{equation}
\end{theorem}

\begin{proof}
\emph{Step 1: the endpoint is unbounded.}
The parameter range, weak equation, trace, energy membership and
semistability follow from \cref{prop:cb-minimal,lem:cb-energy}.
Suppose, for a contradiction, that $q_\eps$ is bounded.  Since
$0\le q_\delta\le q_\eps$, the classical equations give uniform global
$W^{2,p}$ bounds on $u_\delta$, for every finite $p$, on the fixed smooth
domain $D_\eps$.  Sobolev embedding and elliptic bootstrap give
uniform $C^{k,\alpha}$ bounds for every fixed $k$ and $0<\alpha<1$.
The endpoint is consequently a classical solution and inherits all
symmetries and monotonicities from \cref{prop:cb-minimal}.
It follows that
\[
 (\eps,\delta_\eps^*,q_\eps-\delta_\eps^*)
       \in\mathcal E_m(\overline B_m).
\]
By \cref{prop:cb-uniform-gap}, the Dirichlet linearization at $q_\eps$
has a strictly positive first eigenvalue.

Let
\[
 C^{2,\alpha}_D(\overline{D_\eps})
 =\{v\in C^{2,\alpha}(\overline{D_\eps}):
                                  v|_{\Gamma_\eps}=0\}.
\]
Consider the smooth map
\[
 \begin{split}
 \mathcal F:C^{2,\alpha}_D(\overline{D_\eps})\times\R
        &\longrightarrow C^{0,\alpha}(\overline{D_\eps}),\\
 \mathcal F(v,\delta)
        &=-\Delta v-
          \La\bigl[v^2+(1+2\delta)v+\delta(1+\delta)\bigr].
 \end{split}
\]
At $(q_\eps-\delta_\eps^*,\delta_\eps^*)$ the derivative in $v$ is
$L_{q_\eps}$.  Its positive first eigenvalue, the Dirichlet Fredholm
alternative and the Schauder estimate show that it is an isomorphism
between the indicated H\"older spaces.  The implicit-function theorem
therefore produces classical solutions for $\delta$ on both sides of
$\delta_\eps^*$.  This is the same local continuation principle that
underlies the classical positive-solution branch theory of
\cite{CR75,MP}; the preceding annulus estimate is what supplies the
uniform positive eigenvalue needed at the present endpoint.

These nearby solutions remain nonnegative.  Indeed the endpoint
$v_*=q_\eps-\delta_\eps^*$ is positive in the interior and, by the
Hopf boundary lemma, satisfies $\partial_\nu v_*<0$ at every boundary
point.  Since $\Gamma_\eps$ is compact, there is $c_0>0$ such that
$-\partial_\nu v_*\ge 2c_0$ on $\Gamma_\eps$.  By continuity of
$\nabla v_*$ and the tubular-neighborhood theorem, after decreasing
$\rho>0$ if necessary we have
\[
 -\nabla v_*(y-t\nu(y))\cdot\nu(y)\ge \frac32c_0
 \qquad
 (y\in\Gamma_\eps,\ 0\le t\le\rho),
\]
where $\nu(y)$ is the outward unit normal.  On the compact interior set
$K_\rho:=\{x\in\overline{D_\eps}:\dist(x,\Gamma_\eps)\ge\rho\}$,
we have $m_\rho:=\min_{K_\rho}v_*>0$.  Hence every zero-trace function
$v$ sufficiently close to $v_*$ in $C^1(\overline{D_\eps})$ satisfies
\[
 -\nabla v(y-t\nu(y))\cdot\nu(y)\ge c_0
 \quad(0\le t\le\rho),
 \qquad
 v\ge \frac12m_\rho\quad\text{on }K_\rho.
\]
Since $v(y)=0$ on $\Gamma_\eps$, integration along the inward normal
segment gives
\[
 v(y-t\nu(y))
 =\int_0^t-\nabla v(y-s\nu(y))\cdot\nu(y)\,ds
 \ge c_0t\ge0.
\]
Thus $v\ge0$ also in the boundary collar.  Consequently the continuation
gives a nonnegative bounded classical solution for some
$\delta>\delta_\eps^*$, contradicting \eqref{eq:cb-delta-star}.
This proves $q_\eps\notin L^\infty(D_\eps)$.

\emph{Step 2: positivity above the boundary constant.}
Since $q_\delta\ge\delta$, its limit satisfies
$q_\eps\ge\delta_\eps>0$.  Let $\tau_\eps$ solve
\[
 -\Delta\tau_\eps=1\quad\hbox{in }D_\eps,
 \qquad \tau_\eps|_{\Gamma_\eps}=0.
\]
This is a positive classical function.  In distributions,
\[
 -\Delta(q_\eps-\delta_\eps)
 =\La(q_\eps+q_\eps^2)
 \ge\La\delta_\eps(1+\delta_\eps).
\]
The weak maximum principle in $H^1_0$ gives
\[
 q_\eps-\delta_\eps
 \ge\La\delta_\eps(1+\delta_\eps)\tau_\eps>0
 \quad\text{almost everywhere}.
\]
For completeness, subtract the right-hand side, test its distributional
superharmonicity by smooth nonnegative approximations of its negative
part, and pass in $H^1_0$.  Its negative part has zero Dirichlet energy.

\emph{Step 3: scaling and the Br\'ezis--V\'azquez characterization.}
Define $U_\eps$ by \eqref{eq:cb-scaled-solution} and set
\[
 \lambda_{\eps,0}=\frac{\La(1+\delta_\eps)}{\eps^2}.
\]
The exact change of variables in \eqref{eq:inversechange} gives
\[
 U_\eps\ge0,\qquad
 U_\eps\in H^1_0(\Omega_\eps)\cap L^3(\Omega_\eps)
                        \setminus L^\infty(\Omega_\eps),
 \qquad
 -\Delta U_\eps=\lambda_{\eps,0}f_{\delta_\eps}(U_\eps).
\]
This is a weak solution in the sense required by
\cite[Equation~(1.5) and Theorem~3.1]{BV}.  To see this explicitly,
$f_{\delta_\eps}(U_\eps)\in L^{3/2}(\Omega_\eps)\subset L^1(\Omega_\eps)$,
and scaling \eqref{eq:cb-global-weak} and integrating by parts using
$U_\eps\in H^1_0$ gives
\begin{equation}\label{eq:cb-BV-weak}
 \int_{\Omega_\eps}U_\eps(-\Delta\eta)
   =\lambda_{\eps,0}\int_{\Omega_\eps}
                    f_{\delta_\eps}(U_\eps)\eta
 \quad
 \left(\eta\in C^2(\overline{\Omega_\eps}),\
                        \eta|_{\partial\Omega_\eps}=0\right).
\end{equation}
In particular the additional weighted integrability
$f_{\delta_\eps}(U_\eps)\operatorname{dist}(\cdot,\partial\Omega_\eps)
\in L^1$ holds.

Semistability is preserved with its exact scaling factor.  In fact
\[
 \lambda_{\eps,0}f_{\delta_\eps}'(U_\eps(x))
 =\eps^{-2}\La\bigl[1+2q_\eps(x/\eps)\bigr].
\]
For $\psi\in C_c^1(\Omega_\eps)$ put
$\zeta(\xi)=\psi(\eps\xi)$.  With $N=m+1$,
\begin{equation}\label{eq:cb-scaled-stability}
 \begin{split}
 &\int_{\Omega_\eps}
     \bigl[|\nabla\psi|^2
       -\lambda_{\eps,0}f_{\delta_\eps}'(U_\eps)\psi^2\bigr]\\
 &\qquad=\eps^{N-2}\int_{D_\eps}
           \bigl[|\nabla\zeta|^2-\La(1+2q_\eps)\zeta^2\bigr]
 \ge0.
 \end{split}
\end{equation}
For the fixed positive number $\delta_\eps$, all hypotheses on the
nonlinearity are verified in \eqref{eq:nonlinearity-properties}.
We have therefore obtained a nonnegative, unbounded, semistable
$H^1_0$ weak solution.  Br\'ezis and V\'azquez
\cite[Theorem~3.1]{BV} imply simultaneously
\[
 \lambda_{\eps,0}=\lambda^*(\Omega_\eps,f_{\delta_\eps}),
 \qquad U_\eps=u^*(\Omega_\eps,f_{\delta_\eps}).
\]
The theorem follows.  The characterization is applied separately for
each fixed $\eps$; no uniform superlinearity as $\delta_\eps\to0$ is
required.
\end{proof}

\begin{proof}[Proof of \cref{thm:main}]
Take the dimension threshold and thickness in
\eqref{eq:finite-gamma-choice}--\eqref{eq:finite-epsilon-choice} and set
$C_m=\overline B_m$.  Theorem~\ref{thm:finite-singular} gives every assertion for
each fixed $\eps$ in this range.  For any $\eps_j\downarrow0$, choose
$\delta_j=\delta_{\eps_j}$.  Then
$0<\delta_j\le \overline B_m\eps_j^2\to0$, and the stated singular extremal
sequence follows.
\end{proof}

\begin{proof}[Proof of \cref{cor:brezis-op61}]
Fix $N=m+1\ge N_0$ and choose $\eps>0$ below both the thin-domain
threshold in Dancer's theorem stated in the Introduction and the number $\eps_0(m)$ in
\cref{thm:main}.  Dancer's theorem gives
\[
 u^*(\Omega_\eps,e^{(\cdot)})\in L^\infty(\Omega_\eps),
\]
whereas \cref{thm:main}, proved above through
\cref{thm:finite-singular}, supplies $\delta_\eps>0$ such that
\[
 u^*(\Omega_\eps,f_{\delta_\eps})\notin L^\infty(\Omega_\eps).
\]
The properties in \eqref{eq:nonlinearity-properties} show that
$f_{\delta_\eps}$ satisfies the positivity, monotonicity, convexity and
superlinearity hypotheses in Br\'ezis' formulation.  Thus the same smooth
strictly convex ellipsoid gives a negative answer to the Gelfand question
and an affirmative answer to the replacement-nonlinearity question in
Open Problem~6.1.
\end{proof}

\section*{Declaration of generative AI and AI-assisted technologies in the manuscript preparation process}

\noindent\textbf{Statement}: During the preparation of this manuscript, the authors used ChatGPT (OpenAI) to assist in verifying proof ideas, converting the proof arguments into preliminary drafts, and checking the manuscript for potential logical errors. The authors subsequently reviewed and revised all AI-assisted content and take full responsibility for the content of the manuscript.

\section*{Acknowledgments}

We thank Professor Xi-Nan Ma for valuable guidance and discussions. This work was supported by the National Key R\&D Program of China (Grant No. 2025YFA1017601).

\appendix
\section{Radial variation of constants and the normalized nonindicial inverse}\label{app:radial-voc}

This appendix records the radial ODE facts used in the Lyapunov--Schmidt reduction.  The weighted H\"older spaces below are the radial, $z$--independent versions of \eqref{eq:weighted-holder}.  Thus, for example,
\[
 \|u\|_{C^{k,\alpha}_\tau(\R^m)}
 :=\|u\|_{C^{k,\alpha}(\{r<2\})}
   +\sup_{R\ge1}R^\tau\|u(R\,\cdot)\|_{C^{k,\alpha}(\{1<|Y|<2\})}.
\]

\begin{lemma}[Radial variation of constants and axis regularity]\label{lem:radial-voc}
Let $m\ge3$ and
\[
 \mathcal Lu:=-u''-\frac{m-1}{r}u'-P(r)u,
\]
where $P$ is bounded and $C^{0,\alpha}$ near $r=0$.  Suppose that $Z$ is a nonvanishing regular radial solution of $\mathcal LZ=0$, normalized only by $Z(0)\ne0$.  Put, for any fixed $r_*>0$,
\[
 Z_2(r):=Z(r)\int_{r_*}^r\frac{\rho^{1-m}}{Z(\rho)^2}\dd\rho.
\]
Then $Z,Z_2$ are a fundamental pair on $(0,\infty)$ and
\[
 W(Z,Z_2)(r):=Z(r)Z_2'(r)-Z'(r)Z_2(r)=r^{1-m}.
\]
Moreover
\[
 Z(r)=Z(0)+O(r^2),
 \qquad
 Z_2(r)=-\frac{1}{(m-2)Z(0)}r^{2-m}\bigl(1+o(1)\bigr)
 \quad(r\downarrow0),
\]
so $Z_2$ is singular at the axis.  If $f$ is locally $C^{0,\alpha}$, then
\[
 u_0(r)
 :=Z(r)\int_0^r Z_2(s)f(s)s^{m-1}\dd s
   -Z_2(r)\int_0^r Z(s)f(s)s^{m-1}\dd s
\]
satisfies $\mathcal Lu_0=f$, extends as a regular radial $C^{2,\alpha}$ function through $r=0$, and obeys
\[
 u_0(0)=u_0'(0)=0.
\]
Consequently, for every $a\in\R$, the unique regular solution with prescribed center value $u(0)=a$ is
\[
 u(r)=u_0(r)+\frac{a}{Z(0)}Z(r).
\]
\end{lemma}

\begin{proof}
Differentiating the definition of $Z_2$ gives
\[
 Z_2'=Z'\int_{r_*}^r\frac{\rho^{1-m}}{Z(\rho)^2}\dd\rho
       +\frac{r^{1-m}}{Z},
\]
and hence $W(Z,Z_2)=r^{1-m}$.  Abel's identity then shows that $Z_2$ solves the same homogeneous equation; alternatively this follows by direct differentiation.  Since a regular radial solution satisfies $Z'(0)=0$ and $P$ is bounded, the equation gives $Z(r)=Z(0)+O(r^2)$.  Therefore
\[
 \int_{r_*}^r\frac{\rho^{1-m}}{Z(\rho)^2}\dd\rho
 =-\frac{1}{(m-2)Z(0)^2}r^{2-m}\bigl(1+o(1)\bigr),
\]
which yields the stated expansion of $Z_2$.

For the inhomogeneous equation, write it in standard form
\[
 u''+\frac{m-1}{r}u'+P(r)u=-f.
\]
The usual variation-of-constants formula and the identity
$W(Z,Z_2)=r^{1-m}$ give exactly the displayed expression for $u_0$.  Near the axis,
\[
 Z_2(s)f(s)s^{m-1}=O(s),
 \qquad
 Z(s)f(s)s^{m-1}=O(s^{m-1}),
\]
so the two integrals are respectively $O(r^2)$ and $O(r^m)$.  Since
$Z_2(r)=O(r^{2-m})$, both terms in $u_0$ are $O(r^2)$ and their derivatives are $O(r)$.  Substitution in the equation then gives the regular radial $C^{2,\alpha}$ extension and $u_0(0)=u_0'(0)=0$.

Every homogeneous solution is $c_1Z+c_2Z_2$.  Axis regularity forces $c_2=0$, while the condition $u(0)=a$ fixes $c_1=a/Z(0)$.  This proves both existence and uniqueness.
\end{proof}

\begin{lemma}[Euler tail variation estimate]\label{lem:euler-tail-voc}
Let
\[
 E_\lambda u:=-u''-\frac{m-1}{r}u'-\frac{\lambda}{r^2}u,
\]
and assume that the indicial equation
\[
 \beta(m-2-\beta)=\lambda
\]
has two distinct positive roots $0<\beta_-<\beta_+$.  Fix $0<\sigma<\beta_-$ and $R\ge1$.  For $g=O_{C^{0,\alpha}_{\rm sc}}(r^{-\sigma-2})$ on $\{r\ge R\}$, the unique solution of
\[
 E_\lambda u=g,
 \qquad u(R)=u'(R)=0,
\]
is given by
\[
 \begin{split}
 (\mathcal T_Rg)(r)
 =\frac{1}{\beta_+-\beta_-}\Bigg[
 &-r^{-\beta_-}\int_R^r s^{m-1-\beta_+}g(s)\dd s\\
 &+r^{-\beta_+}\int_R^r s^{m-1-\beta_-}g(s)\dd s
 \Bigg],
 \end{split}
\]
and satisfies
\[
 \|\mathcal T_Rg\|_{C^{2,\alpha}_\sigma(\{r\ge R\})}
 \le C\|g\|_{C^{0,\alpha}_{\sigma+2}(\{r\ge R\})},
\]
with $C$ independent of $R\ge1$ after the natural scaled identification of the tail annuli.
\end{lemma}

\begin{proof}
The homogeneous pair is $r^{-\beta_-},r^{-\beta_+}$ and its Wronskian equals
\[
 -\bigl(\beta_+-\beta_-\bigr)r^{1-m}.
\]
The displayed formula is therefore the variation-of-constants formula with both integration constants chosen so that the value and radial derivative vanish at $r=R$.  If
$|g(s)|\le M s^{-\sigma-2}$, then, using
$m-2=\beta_-+\beta_+$,
\[
 \int_R^r s^{m-1-\beta_+}|g(s)|\dd s
 \le CM\int_R^r s^{\beta_--\sigma-1}\dd s
 \le CM r^{\beta_--\sigma},
\]
and similarly
\[
 \int_R^r s^{m-1-\beta_-}|g(s)|\dd s
 \le CM r^{\beta_+-\sigma}.
\]
Because $\sigma<\beta_-$, both terms are $O(Mr^{-\sigma})$.  Differentiating the formula once gives the corresponding $O(Mr^{-\sigma-1})$ estimate; the equation then gives $u''=O(Mr^{-\sigma-2})$.  Finally, rescaling each dyadic annulus to $1<|Y|<2$ and applying the ordinary interior radial Schauder estimate upgrades these pointwise derivative bounds to the stated $C^{2,\alpha}_\sigma$ estimate.  The strict inequality $\sigma<\beta_-$ is exactly what prevents an indicial logarithm.
\end{proof}

\begin{lemma}[Augmented Euler tail estimate]\label{lem:euler-tail-augmented}
Under the assumptions of \cref{lem:euler-tail-voc}, fix
\[
 \beta_-<\tau<\beta_+.
\]
Let $g\in C^{0,\alpha}_{\tau+2}(\{r\ge R\})$ and prescribe Cauchy data
$u(R)=b_0$, $u'(R)=b_1$.  Then the unique solution of
\[
 E_\lambda u=g\qquad (r\ge R)
\]
belongs to
\[
 \spanop\{r^{-\beta_-}\}\oplus C^{2,\alpha}_\tau(\{r\ge R\}).
\]
More precisely, there is a unique coefficient $c_-(u)$ such that
\begin{equation}\label{eq:euler-augmented-tail-est}
 |c_-(u)|+
 \|u-c_-(u)r^{-\beta_-}\|_{C^{2,\alpha}_\tau(\{r\ge R\})}
 \le C_R\bigl(|b_0|+R|b_1|\bigr)
      +C\|g\|_{C^{0,\alpha}_{\tau+2}(\{r\ge R\})},
\end{equation}
where the forcing constant $C$ is independent of $R\ge1$ under the
natural scaled identification of the tail annuli.
\end{lemma}

\begin{proof}
Let $H_b=A_-r^{-\beta_-}+A_+r^{-\beta_+}$ be the unique homogeneous
Euler solution with $H_b(R)=b_0$ and $H_b'(R)=b_1$.  The Wronskian
\[
 W(r^{-\beta_-},r^{-\beta_+})(R)
 =-(\beta_+-\beta_-)R^{1-m}\ne0
\]
shows that $A_\pm$ are uniquely determined and
\[
 |A_-|+\|A_+r^{-\beta_+}\|_{C^{2,\alpha}_\tau(r\ge R)}
 \le C_R\bigl(|b_0|+R|b_1|\bigr).
\]
Write the solution as
\[
 u=H_b+\mathcal T_Rg,
\]
where $\mathcal T_R$ is the zero-Cauchy variation-of-constants operator
from \cref{lem:euler-tail-voc}.  Since $\tau>\beta_-$, the first
integral in that formula converges at infinity.  Set
\[
 c_g:=-\frac1{\beta_+-\beta_-}
 \int_R^\infty s^{m-1-\beta_+}g(s)\dd s.
\]
Using $m-2=\beta_-+\beta_+$, one obtains the exact decomposition
\[
 \mathcal T_Rg-c_gr^{-\beta_-}
 =\frac1{\beta_+-\beta_-}\left[
 r^{-\beta_-}\int_r^\infty s^{m-1-\beta_+}g(s)\dd s
 +r^{-\beta_+}\int_R^r s^{m-1-\beta_-}g(s)\dd s
 \right].
\]
If $|g(s)|\le Ms^{-\tau-2}$, then
\[
 |c_g|\le CM R^{\beta_--\tau}\le CM,
\]
and the two terms on the right are respectively bounded by
$CMr^{-\tau}$.  Differentiating once and using the equation for the
second derivative gives the corresponding weighted derivative bounds;
rescaled Schauder estimates on dyadic annuli give
\[
 |c_g|+
 \|\mathcal T_Rg-c_gr^{-\beta_-}\|_{C^{2,\alpha}_\tau(r\ge R)}
 \le C\|g\|_{C^{0,\alpha}_{\tau+2}(r\ge R)}.
\]
Taking $c_-(u)=A_-+c_g$ proves \eqref{eq:euler-augmented-tail-est}.
Uniqueness follows from $r^{-\beta_-}\notin C^{2,\alpha}_\tau$ and
$r^{-\beta_+}\in C^{2,\alpha}_\tau$.
\end{proof}

\begin{proposition}[Normalized inverse for a two-root Euler tail]\label{prop:normalized-voc}
Let $m\ge3$ and
\[
 \mathcal Lu=-u''-\frac{m-1}{r}u'-P(r)u
\]
be a radial operator whose coefficient is regular at the axis.  Assume:
\begin{enumerate}[label=\textup{(\roman*)},leftmargin=2.3em]
\item the regular homogeneous kernel is one-dimensional and is spanned by a nonvanishing solution $Z$ with $Z(0)\ne0$;
\item for some $\lambda>0$ and $\delta>0$,
\[
 P(r)=\frac{\lambda}{r^2}+q(r),
 \qquad
 q=O_{C^{0,\alpha}_{\rm sc}}(r^{-2-\delta});
\]
\item the indicial equation $\beta(m-2-\beta)=\lambda$ has distinct positive roots $0<\beta_-<\beta_+$.
\end{enumerate}
Then, for every $0<\sigma<\beta_-$,
\[
 u\longmapsto(\mathcal Lu,u(0)):
 C^{2,\alpha}_{\sigma,\mathrm{rad}}(\R^m)
 \longrightarrow
 C^{0,\alpha}_{\sigma+2,\mathrm{rad}}(\R^m)\times\R
\]
is an isomorphism.  In particular,
\[
 \|u\|_{C^{2,\alpha}_\sigma}
 \le C\bigl(\|\mathcal Lu\|_{C^{0,\alpha}_{\sigma+2}}+|u(0)|\bigr).
\]
\end{proposition}

\begin{proof}
Fix a large radius $R$ to be chosen below.  On the fixed core $0\le r\le R$, \cref{lem:radial-voc} gives, for every $f$ and $a$, a unique regular solution of
\[
 \mathcal Lu=f,\qquad u(0)=a,
\]
and ordinary ODE/Schauder estimates give
\[
 |u(R)|+R|u'(R)|
 \le C_R\bigl(\|f\|_{C^{0,\alpha}(\{r\le R\})}+|a|\bigr).
\]
Set $b_0=u(R)$ and $b_1=u'(R)$.  Let $H_b=A_-r^{-\beta_-}+A_+r^{-\beta_+}$ be the unique Euler homogeneous solution with
$H_b(R)=b_0$ and $H_b'(R)=b_1$.  Since $\sigma<\beta_-$,
\[
 \|H_b\|_{C^{2,\alpha}_\sigma(\{r\ge R\})}
 \le C_R(|b_0|+|b_1|).
\]

On the tail the equation is
\[
 E_\lambda u=f+qu.
\]
Because $\mathcal T_R$ in \cref{lem:euler-tail-voc} has zero Cauchy data at $r=R$, a solution with the prescribed data $(b_0,b_1)$ is a fixed point of
\[
 u=H_b+\mathcal T_R(f+qu).
\]
Multiplication by $q$ satisfies
\[
 \|qu\|_{C^{0,\alpha}_{\sigma+2}(\{r\ge R\})}
 \le CR^{-\delta}\|u\|_{C^{2,\alpha}_\sigma(\{r\ge R\})}.
\]
Hence, after choosing $R$ so large that the product of this constant with the norm of $\mathcal T_R$ is at most $1/2$, the right-hand side is a contraction on the tail weighted space.  Therefore there is a unique tail solution and
\[
 \|u\|_{C^{2,\alpha}_\sigma(\{r\ge R\})}
 \le C_R(|b_0|+|b_1|)
   +C\|f\|_{C^{0,\alpha}_{\sigma+2}(\{r\ge R\})}.
\]
Combining this with the fixed-core estimate proves the global bound and surjectivity.

For injectivity, suppose $\mathcal Lu=0$ and $u(0)=0$.  By assumption (i), every regular homogeneous solution is $cZ$; since $Z(0)\ne0$, the center condition gives $c=0$.  Thus the kernel is trivial, completing the proof.
\end{proof}

\begin{remark}\label{rem:LE-voc-application}
For the Lane--Emden linearization used in the Lyapunov--Schmidt reduction above,
\[
 P(r)=2c_0V(r)=4(m-4)r^{-2}+O(r^{-\kappa}),
 \qquad \kappa\ge\beta_->2,
\]
so assumption (ii) holds with $\lambda=4(m-4)$ and $\delta=\kappa-2>0$.  The regular homogeneous kernel is generated by the positive scaling field $Z=2V+rV'$, with $Z(0)=2$.  Hence \cref{prop:normalized-voc} applies for every $2<\sigma<\beta_-$, exactly as used in the Lyapunov--Schmidt argument.
\end{remark}

\section{Uniform massive product resolvents on the cylinder}
\label{app:massive-product-resolvent}

The estimates in this appendix are a concrete product-cylinder version of
standard weighted elliptic theory on noncompact manifolds; compare
Lockhart--McOwen~\cite{LM}.  We include the proof because the exact
polynomial H\"older weights, the fibrewise projection $P^\perp$, and the
uniform limit $\rho\downarrow0$ are specific to the present argument.

We shall repeatedly use the elementary equivalence between the scaled
annular norm in \cref{def:weighted-spaces} and derivative-by-derivative
weighted norms.  For every fixed integer $k\ge0$,
\begin{equation}\label{eq:app-weighted-norm-equivalence}
 \begin{split}
 \|u\|_{C^{k,\alpha}_\gamma(\Cyl)}\asymp{}&
 \|u\|_{C^{k,\alpha}(\{r<2\}\times(-1,1))}\\
 &+\sum_{|\nu|+j\le k}
   \sup_{R\ge1}R^{\gamma+|\nu|}
   \|\partial_X^\nu\partial_z^j u
     \|_{C^{0,\alpha}_{\rm sc}(A_R)} .
 \end{split}
\end{equation}
The constants depend only on $k,\alpha$ and the fixed reference annulus.
Indeed, after the change of variables $X=RY$, this is just the standard
finite-order equivalence of H\"older norms on the fixed annulus
$\mathcal A$.  Thus commuting $Y$-derivatives through the product
resolvent below gives exactly the scaled annular estimates used in the
main text; no additional unscaled derivative gain is being assumed.

\begin{proposition}[Uniform massive product resolvent]
\label{prop:massive-product-resolvent}
Let $H_z=-\partial_{zz}-\La$ on the even Dirichlet subspace of
$(-1,1)$ orthogonal to $\ph$.  Fix $\gamma>0$, $\rho_0>0$, and an integer
$k\ge2$.

\begin{enumerate}[label=(\roman*)]
\item For $F\in\mathcal Y_\gamma^\perp$, the problem
\[
 (-\Delta_Y+H_z)w=F,\qquad w(Y,\pm1)=0,\qquad P_1w=0,
\]
has a unique solution in the graph domain $\mathcal G_\gamma^\perp$ and
\[
 \|w\|_{\mathcal G_\gamma^\perp}\le C\|F\|_{C^{0,\alpha}_\gamma}.
\]
Its graph norm controls the ordinary $C^{2,\alpha}$ norm on smaller fixed
cylinders and
\[
 \sup_{r\ge0}(1+r)^\gamma
 \left(\|w(r,\cdot)\|_{C^1([-1,1])}
       +\sup_{|z|<1}\frac{|w(r,z)|}{\ph(z)}\right)
 \le C\|w\|_{\mathcal G_\gamma^\perp}.
\]
If $F\in P^\perp C^{2,\alpha}_\gamma$, then
\[
 \|w\|_{C^{2,\alpha}_\gamma}\le C\|F\|_{C^{2,\alpha}_\gamma}.
\]

\item For every $0\le\rho\le\rho_0$ and every radial-even
$f\in\mathcal F_\gamma^k$ with $P_1f=0$, the problem
\[
 (H_z-\rho\Delta_Y)w=f,\qquad w(Y,\pm1)=0,\qquad P_1w=0,
\]
has a unique solution in $\mathcal W_{\gamma,\perp}^{(k)}$ and
\[
 \|w\|_{\mathcal W_\gamma^{(k)}}
 \le C_k\|f\|_{\mathcal F_\gamma^k},
\]
with $C_k$ independent of $\rho\in[0,\rho_0]$.
\end{enumerate}
\end{proposition}

\begin{proof}
On $\ph^\perp$ the even transverse spectrum of $H_z$ is bounded below by
$8\La>0$.  For $\rho\ge0$ define
\begin{equation}\label{eq:app-product-resolvent}
 \mathcal R_\rho f
 :=\int_0^\infty e^{\rho t\Delta_Y}e^{-tH_z}P^\perp f\,dt.
\end{equation}
The projected one-dimensional Dirichlet heat kernel decays exponentially
for $t\ge1$.  For $0<t\le1$ its $z$-H\"older norm is
$O(1+t^{-\alpha/2})$, and one $z$ derivative is
$O(t^{-(1+\alpha)/2})$.  These bounds follow directly from the sine/cosine
expansion of the Dirichlet kernel.

Polynomial weights are stable under translation with polynomial loss:
for every $a\in\R^m$,
\[
 \|F(\,\cdot-a,\cdot)\|_{C^{0,\alpha}_\gamma}
 \le C(1+|a|)^{\gamma+\alpha}\|F\|_{C^{0,\alpha}_\gamma}.
\]
After averaging against the Gaussian kernel with variance $\rho t$,
$0\le\rho\le\rho_0$, this yields at most polynomial growth in $t$.
The exponential transverse factor therefore makes
\eqref{eq:app-product-resolvent} integrable uniformly in $\rho$.  It
follows that
\[
 \|\mathcal R_\rho f\|_{C^{0,\alpha}_\gamma}
 \le C\|f\|_{C^{0,\alpha}_\gamma}.
\]
The representation has zero Dirichlet trace and, when $P_1f=0$, solves
the unprojected equation.

For part (i), take $\rho=1$.  Local interior and flat-boundary Schauder
estimates give the compact-cylinder $C^{2,\alpha}$ control and closedness
of the graph.  On unit strips centered at large $r$, the same boundary
estimate gives
\[
 \|w\|_{C^{2,\alpha}(\mathcal U_{r,1/2})}
 \le C(1+r)^{-\gamma}\|w\|_{\mathcal G_\gamma^\perp}.
\]
Since $w(r,\pm1)=0$ and $\ph(z)\asymp1-|z|$, this implies the displayed
boundary-factor estimate.  If $F$ has two weighted $Y$ derivatives,
commute them through \eqref{eq:app-product-resolvent}; one $z$ derivative
uses the integrable short-time bound above, while
$w_{zz}=-\La w-\Delta_Yw-F$ supplies the remaining second derivative.

For part (ii), commute $\partial_Y^\nu$, $|\nu|\le k$, through
\eqref{eq:app-product-resolvent}.  The shifted weight
$\gamma+|\nu|$ absorbs the polynomial Gaussian moments.  The first
$z$-derivative terms are integrable as above, and for $|\nu|\le k-2$ use
\[
 \partial_{zz}\partial_Y^\nu w
 =-\La\partial_Y^\nu w
  -\rho\Delta_Y\partial_Y^\nu w-\partial_Y^\nu f.
\]
This proves the uniform mixed estimate.

For uniqueness, decompose into transverse eigenfunctions.  If $\rho>0$,
each homogeneous coefficient solves $(-\rho\Delta_Y+\mu_j)u_j=0$ with
$\mu_j>0$ and polynomial decay; the maximum principle on expanding balls
gives $u_j=0$.  For $\rho=0$, uniqueness is the fibrewise invertibility
of $H_z$ on $\ph^\perp$.
\end{proof}

\section{Parameter regularity in the weighted Lyapunov--Schmidt reduction}
\label{app:LS-parameter}

\begin{proposition}[Parameter regularity of the small fixed point]
\label{prop:LS-parameter-regularity}
Let $(\xi_\rho,\omega_\rho)$ be the small fixed point determined, for
$0\le\rho\le\rho_0$, by \eqref{eq:LS-omega-fixed-point} and
\eqref{eq:LS-xi-equation}.  After reducing $\rho_0$ if necessary,
\[
 \|\xi_\rho\|_{C^{6,\alpha}_\sigma}
 +\|\omega_\rho\|_{\mathcal W_\sigma^{(6)}}\le C\rho,
\]
and $\rho\mapsto(\xi_\rho,\omega_\rho)$ is $C^1$ into
$C^{2,\alpha}_\sigma\times\mathcal W_\sigma$, with uniformly bounded
derivative.  The weighted affine tail is differentiable in the same
sense.
\end{proposition}

\begin{proof}
For each fixed integer $k\ge2$, the normalized scalar inverse satisfies
\[
 \|\mathcal S f\|_{C^{k,\alpha}_\sigma}
 \le C_k\|f\|_{\mathcal F_{\sigma+2}^{k-2}}.
\]
Indeed, the $C^{2,\alpha}_\sigma$ estimate is
\cref{prop:normalized-voc}; on each enlarged annulus, scale $Y=R\xi$ and
apply the ordinary interior Schauder estimate.  The Lane--Emden ODE and
its tail give bounded weighted derivatives of $V$ of all fixed orders,
and the same is true for
$w_0=\La H_z^{-1}P^\perp(V^2\ph^2)$.
Leibniz's rule and \cref{prop:massive-product-resolvent} then yield, for
$k=4,6$,
\[
 \|\Omega_\rho(\xi)\|_{\mathcal W_\sigma^{(k)}}
 \le C_k(\rho+\|\xi\|_{C^{k,\alpha}_\sigma}),
\]
and the analogous Lipschitz estimate.  The highest $Y$ derivatives of
$\omega$ occur with factors $\rho$ or $\rho^2$ in
\eqref{eq:LS-omega-fixed-point}; the scalar equation depends on $\omega$
with the factor $\rho$.  The same two contractions therefore close in
orders $k=4,6$, and uniqueness in the original two-derivative space
identifies the fixed points.  This proves the six-derivative estimate.

For the parameter derivative, write
\[
 v_\rho=V+\xi_\rho,\qquad w_\rho=w_0+\omega_\rho,
 \qquad B_\rho=v_\rho\ph+\rho w_\rho.
\]
Differentiating the exact projected equations formally gives
\begin{equation}\label{eq:app-LS-parameter-system}
 \begin{cases}
 L_{LE}U=a_\rho U+\rho\mathcal B_\rho W+f_\rho,&U(0)=0,\\
 (H_z-\rho\Delta_Y)W
  =\mathcal C_\rho U+\rho\mathcal D_\rho W+g_\rho,
  &P_1W=0,\ W|_{z=\pm1}=0,
 \end{cases}
\end{equation}
where
\[
 \begin{split}
 a_\rho&=2\La\left(B_2\xi_\rho
           +\rho\int_{-1}^1w_\rho\ph^2\,dz\right),\\
 \mathcal B_\rho k&=2\La\int_{-1}^1B_\rho\ph k\,dz,
 \qquad
 \mathcal C_\rho h=2\La P^\perp(B_\rho\ph h),\\
 \mathcal D_\rho k&=2\La P^\perp(B_\rho k),\\
 f_\rho&=2\La\int_{-1}^1B_\rho w_\rho\ph\,dz,
 \qquad
 g_\rho=\Delta_Yw_\rho+2\La P^\perp(B_\rho w_\rho).
 \end{split}
\]
The six-derivative estimate implies
\[
 \|a_\rho U\|_{C^{0,\alpha}_{\sigma+2}}
 \le C\rho\|U\|_{C^{2,\alpha}_\sigma},
 \qquad
 \|f_\rho\|_{C^{0,\alpha}_{\sigma+2}}
 +\|g_\rho\|_{\mathcal F_\sigma^2}\le C,
\]
with the analogous bounded estimates for the three off-diagonal
operators.  First invert
$I-\rho\mathcal R_\rho\mathcal D_\rho$ in the second row using
\cref{prop:massive-product-resolvent}; substitute into the first row and
apply $\mathcal S$.  The remaining operator has norm $O(\rho)$, so a
second Neumann series gives a uniform inverse of
\eqref{eq:app-LS-parameter-system}.

To justify formal differentiation, fix $h$ with
$\rho,\rho+h\in[0,\rho_0]$ and set
\[
 U_h:=\frac{v_{\rho+h}-v_\rho}{h},\qquad
 W_h:=\frac{w_{\rho+h}-w_\rho}{h},\qquad
 B_+:=B_{\rho+h},\quad B:=B_\rho.
\]
Since
\[
 \frac{B_+-B}{h}=U_h\ph+\rho W_h+w_{\rho+h},
\]
subtracting the two exact projected systems and dividing by $h$ gives the
following \emph{exact} difference-quotient system:
\begin{equation}\label{eq:app-LS-dq-system}
 \begin{cases}
 L_{LE}U_h=a_{\rho,h}U_h
       +\rho\mathcal B_{\rho,h}W_h+f_{\rho,h},&U_h(0)=0,\\[1mm]
 (H_z-\rho\Delta_Y)W_h
       =\mathcal C_{\rho,h}U_h
        +\rho\mathcal D_{\rho,h}W_h+g_{\rho,h},
       &P_1W_h=0,\ W_h|_{z=\pm1}=0,
 \end{cases}
\end{equation}
where
\begin{align*}
 a_{\rho,h}
 &=\La\int_{-1}^1(B_++B-2V\ph)\ph^2\,dz,\\
 \mathcal B_{\rho,h}k
 &=\La\int_{-1}^1(B_++B)\ph k\,dz,\\
 \mathcal C_{\rho,h}u
 &=\La P^\perp\bigl((B_++B)\ph u\bigr),\qquad
 \mathcal D_{\rho,h}k
 =\La P^\perp\bigl((B_++B)k\bigr),\\
 f_{\rho,h}
 &=\La\int_{-1}^1(B_++B)w_{\rho+h}\ph\,dz,\\
 g_{\rho,h}
 &=\Delta_Yw_{\rho+h}
   +\La P^\perp\bigl((B_++B)w_{\rho+h}\bigr).
\end{align*}
These formulas are obtained simply from
$B_+^2-B^2=(B_++B)(B_+-B)$ and from
\[
 (H_z-(\rho+h)\Delta_Y)w_{\rho+h}
 -(H_z-\rho\Delta_Y)w_\rho
 =(H_z-\rho\Delta_Y)(w_{\rho+h}-w_\rho)
   -h\Delta_Yw_{\rho+h}.
\]
The already proved estimate
$B_\rho=V\ph+O(\rho)$ in the weighted algebra gives, uniformly for
$|h|$ small,
\begin{equation}\label{eq:app-LS-dq-coefficients}
 \|a_{\rho,h}\|\le C(\rho+|h|),\qquad
 \|\mathcal B_{\rho,h}\|+
 \|\mathcal C_{\rho,h}\|+
 \|\mathcal D_{\rho,h}\|\le C,
\end{equation}
and the six-derivative bound yields
\begin{equation}\label{eq:app-LS-dq-sources}
 \|f_{\rho,h}\|_{\mathcal F_{\sigma+2}^{k-2}}
 +\|g_{\rho,h}\|_{\mathcal F_\sigma^k}
 \le C_k,\qquad k=2,4,
\end{equation}
with the evident interpretation of the first norm when $k=2$.
The only term in \eqref{eq:app-LS-dq-sources} containing two additional
$Y$-derivatives is $\Delta_Yw_{\rho+h}$; it is precisely here that the
four- and six-derivative estimates obtained above are used.

The same two-stage Neumann inversion used for
\eqref{eq:app-LS-parameter-system} therefore applies uniformly to
\eqref{eq:app-LS-dq-system} at derivative orders two and four and gives
\begin{equation}\label{eq:app-LS-dq-bound}
 \|U_h\|_{C^{k,\alpha}_\sigma}
 +\|W_h\|_{\mathcal W_\sigma^{(k)}}
 \le C_k,\qquad k=2,4.
\end{equation}
Equivalently,
\[
 \|v_{\rho+h}-v_\rho\|_{C^{k,\alpha}_\sigma}
 +\|w_{\rho+h}-w_\rho\|_{\mathcal W_\sigma^{(k)}}
 \le C_k|h|,\qquad k=2,4.
\]
The $k=4$ estimate implies that the coefficients and sources in
\eqref{eq:app-LS-dq-system} converge, as $h\to0$, in the
source topologies required by the $k=2$ inverse.  Their limits are
exactly the coefficients and sources displayed in
\eqref{eq:app-LS-parameter-system}.  Hence $(U_h,W_h)$ converges to the
unique solution of that system.  This proves $C^1$ dependence, including
the one-sided derivative at $\rho=0$.  Differentiability of the affine
tail follows from the same weighted derivative bounds and the fact that
its leading coefficient $D_mr^{-2}\ph$ is parameter-independent.
\end{proof}

\end{document}